\documentclass[11pt,twoside, a4paper]{article}
\usepackage{latexsym}
\usepackage[english]{babel}
\usepackage{t1enc}
\usepackage{mathrsfs}
\usepackage{ifthen}
\usepackage{url}
\usepackage{epsfig}
\usepackage{amsbsy}
\usepackage{extarrows}
\usepackage{amsfonts}
\usepackage{amsmath}
\usepackage{amssymb}
\usepackage{amsthm}
\usepackage{amsxtra}
\usepackage{bbm}
\usepackage{color}
\usepackage{relsize} 
\usepackage{enumitem}
\usepackage{graphicx}
\usepackage{caption}
\usepackage{subcaption}
\usepackage{placeins}
\usepackage{wrapfig}
\usepackage{float}

\usepackage[margin=2.5cm]{geometry}

\usepackage{verbatim}

\numberwithin{equation}{section}
\numberwithin{figure}{section}

\theoremstyle{plain} 
\newtheorem{theorem}{Theorem}[section]
\newtheorem*{theorem*}{Theorem}
\newtheorem{lemma}[theorem]{Lemma}
\newtheorem*{lemma*}{Lemma}

\newtheorem*{corollary*}{Corollary}
\newtheorem{proposition}[theorem]{Proposition}
\newtheorem*{proposition*}{Proposition}

\newtheorem*{definition*}{Definition}

\newtheorem*{example*}{Example}

\newtheorem*{remark*}{Remark}
\newtheorem*{remarks*}{Remarks}

\theoremstyle{definition}
\newtheorem{definition}[theorem]{Definition}

\theoremstyle{remark}
\newtheorem{remark}[theorem]{Remark}

\definecolor{olivegreen}{rgb}{0,0.6,0.1}

\newcommand{\R} {\mathbb{R}}
\newcommand{\C} {{\mathbb{C}}}
\newcommand{\N} {{\mathbb{N}}}
\newcommand{\Cp} {\mathbb {C}_+}

\newcommand{\ii}{\mathrm{i}} 
\newcommand{\dd}{\mathrm{d}}

\newcommand{\avgB}[1]{\Big\langle #1 \Big\rangle}

\DeclareMathOperator{\re}{Re}
\DeclareMathOperator{\im}{Im}
\DeclareMathOperator{\supp}{supp}
\DeclareMathOperator{\dist}{dist}
\DeclareMathOperator{\diag}{diag}
\DeclareMathOperator{\spn}{Span}

\DeclareMathOperator{\Ran}{Ran}
\DeclareMathOperator{\sign}{sign}
\DeclareMathOperator{\Spec}{Spec}

\usepackage{hyperref}

\begin{document}
\renewcommand{\thefootnote}{\fnsymbol{footnote}}

\title{Non-symmetric Vector Dyson Equations}

\author{
Jiaoyang Huang\textsuperscript{\(\dagger\)}
\and
Zhonggen Su\textsuperscript{\(\ddagger\)}
\and
Ruizhe Xu\textsuperscript{\(\S\)}
}

\date{}

\maketitle
\thispagestyle{empty}

\begingroup
\renewcommand{\thefootnote}{\fnsymbol{footnote}}


\footnotetext[2]{Department of Statistics and Data Science,
University of Pennsylvania. Email:
\texttt{huangjy@wharton.upenn.edu}.}

\footnotetext[3]{School of Mathematical Sciences,
Zhejiang University. Email:
\texttt{suzhonggen@zju.edu.cn}.}

\footnotetext[4]{School of Mathematical Sciences,
Zhejiang University; Department of Mathematics,
The University of Hong Kong. Email:
\texttt{ruizhexu.math@gmail.com}.}

\endgroup

\vspace{-0.5cm}

\thispagestyle{empty} 

\maketitle
\thispagestyle{empty} 

\maketitle
\thispagestyle{empty} 
\vspace{-0.5cm}

\begin{abstract}
    We study the vector Dyson equation
    \begin{align*}
        -\frac1{m(z)}=z\mathbf1-\mathbf a+Sm(z),
    \end{align*}
    with parameter $z$ in the complex upper half-plane $\Cp$, where
    \(\mathbf a\in\mathbb R^d\) and \(S\) is a nonnegative matrix, not necessarily
    symmetric. This equation has a unique vector solution \(m(z)\in\Cp^d\), for which we establish a complete measure decomposition and prove regularity.
    We then develop a graph-theoretic approach to the singularity and stability problem for non-symmetric matrices \(S\). The graph structure of \(S\) identifies the
    possible degeneracies of the stability operator as $z$ approaches the real axis. In particular, for non-backtracking matrices, we prove
    square-root growth at regular edges, cubic-root growth at regular cusps, and
    complete stability estimates. We also obtain the corresponding estimates for
    symmetric matrices in the periodic setting.
\end{abstract}

\section{Introduction}
\subsection{Motivation}\label{sec.motivation}
The \emph{Dyson equation}, in its finite-dimensional vector form often called the
\emph{quadratic vector equation}, is the nonlinear system
\begin{equation*}
    -\frac{1}{m_i(z)}
    =
    z-a_i+\sum_{j=1}^d s_{ij}m_j(z),
    \qquad i=1,\dots,d,
\end{equation*}
where \(z\in \mathbb C_+\) and the solution
\(m(z)=(m_1(z),\dots,m_d(z))\) takes values in
\(\mathbb C_+^d\), with
$
    \mathbb C_+ := \{z\in \mathbb C:\operatorname{Im} z>0\}$.
Here \(\mathbf a=(a_i)_{i=1}^d\in \mathbb R^d\), and
\(S=(s_{ij})_{i,j=1}^d\) is an entrywise nonnegative matrix. Equivalently, we write
\begin{equation}\label{eq:QVE}
    -\frac1{m(z)}=z\mathbf 1-\mathbf a+Sm(z),
\end{equation}
where \(1/m(z)\) is understood componentwise.

A fundamental source of the Dyson equation is the theory of
\emph{Wigner-type random matrices}. A classical Wigner matrix is an
\(N\times N\) Hermitian or real symmetric random matrix
$
    H=(h_{ij})_{i,j=1}^N,
$
whose upper-triangular entries are independent, centered random variables,
up to the symmetry constraint, with variances of order \(1/N\). In the
simplest normalization,
$    \mathbb E [h_{ij}]=0$
and
    $\mathbb E |h_{ij}|^2=1/N$
   for $i\neq j$,
together with suitable assumptions on the diagonal entries. In a
Wigner-type matrix, the variances are allowed to depend on the indices:
  $  \mathbb E |h_{ij}|^2=s_{ij}.$
The matrix \(S=(s_{ij})_{i,j=1}^N\) is then called the variance profile of
the ensemble; it is typically symmetric, entrywise nonnegative, and has
row sums of order one. In the homogeneous case, the Dyson equation reduces
to the scalar self-consistent equation for the Stieltjes transform of
Wigner's semicircle law. Thus the general symmetric case may be viewed as
an inhomogeneous extension of the semicircle law.

More broadly, vector Dyson equations, or closely related self-consistent
canonical equations, have appeared in various forms in the study of
Wigner-type matrices, correlated random matrices, and stochastic canonical
equations
\cite{Ber1973,Wegner1979,Girko2001,KhorunzhyPastur1994,
Shlyakhtenko1996,Guionnet2002,AndersonZeitouni2005,AndersonZeitouni2008,hachem2006empirical}.
Their solutions encode the limiting spectral density of states through
the boundary values of a family of Stieltjes transforms.

Beyond the macroscopic density of states, the Dyson equation also plays a
central role in the modern theory of \emph{local laws} and
\emph{universality} for random matrices. In these applications, one needs
much more than a description of the limiting density: it is necessary to
control the solution quantitatively down to very small spectral scales and
to understand its stability under perturbations. A key object is the
linearized stability operator associated with the Dyson equation. Bounds
on the inverse of this operator quantify the sensitivity of the solution
to perturbations and form a fundamental deterministic input in the proof
of local laws and universality.

The qualitative behavior of the corresponding densities, including the
possible occurrence of square-root edges and cubic-root cusps, was
classified in \cite{AEK2017singularities}. Quantitative refinements and
the associated stability theory, which are key inputs for local laws and
universality, were developed in \cite{AEK2019quadratic}; see also
\cite{AEK2016universality,AEK2016correlated,alt2017local,erdHos2020cusp,cipolloni2019cusp,alt2020correlated}.

In the examples discussed above, the matrix \(S\) in the vector Dyson
equation is symmetric. A second natural source of Dyson equations, and the
main motivation for the present work, comes from \emph{Jacobi-type
operators} on infinite trees. These models naturally lead to
non-symmetric matrices \(S\), and therefore to a class of vector Dyson
equations whose stability theory is not covered by the symmetric setting.

Let \(\mathcal G\) be a locally finite graph with vertex set \(\mathbb V_{\mathcal G}\).
A Jacobi-type operator on \(\mathcal G\) acts on functions \(f\) on
\(\mathbb V_{\mathcal G}\) by
\begin{equation}\label{eq.jacobi-operator}
    (Hf)(x)=a_x f(x)+\sum_{y\sim x} b_{xy} f(y),
    \qquad x\in \mathbb V_{\mathcal G},
\end{equation}
where \(y\sim x\) denotes adjacency, \(a_x\in \mathbb R\) is a vertex
potential, and the edge weights satisfy \(b_{xy}\in \mathbb R\) and
\(b_{xy}=b_{yx}\). Weighted adjacency operators, graph Laplacians, and
discrete Schrödinger operators are all special cases of this form. Thus
Jacobi-type operators provide a natural graph-theoretic analogue of
differential operators and play a central role in the spectral theory of
discrete structures.

A particularly important class arises from \emph{universal covers} of
finite graphs. Given a connected finite graph \(\mathcal G\), its universal
cover is an infinite tree \(\mathcal T\), together with a covering map
\(\pi:\mathcal T\to\mathcal G\) that locally preserves the graph structure.
Intuitively, \(\mathcal T\) is obtained by unfolding all cycles of
\(\mathcal G\): every local neighborhood in \(\mathcal T\) is identified
with the corresponding neighborhood in \(\mathcal G\), while globally
\(\mathcal T\) has no cycles. Any local operator on \(\mathcal G\) can be
lifted through \(\pi\) to an operator on \(\mathcal T\). Since the
coefficients are inherited from the finite graph, the lifted operator is
\emph{periodic} in a natural sense and is determined by finitely many parameters.
This perspective includes periodic Jacobi matrices on trees, which were
studied systematically in \cite{AvniBreuerSimon2020}. More recently,
periodic operators on universal covering trees, obtained by lifting local
operators from finite graphs, were investigated in
\cite{BanksGarzaVargasMukherjee2022}, where the point spectrum was
characterized in terms of finite-graph data. More generally, this setting
is closely related to \emph{trees of finite cone type}, also known as
substitution trees, in which an infinite tree is generated from finite
combinatorial data and the associated operators are again described by
finitely many parameters; see
\cite{KellerLenzWarzel2013,KellerLenzWarzel2014,random2002nagnibeda,Sadel2013} and the
references therein.

Although the lifted Jacobi operator on the tree is self-adjoint, its Green
function is most naturally described through recursion relations on
directed edges. Let \(\vec E_{\mathcal G}\) denote the set of oriented
edges of the base graph,
\[
    \vec E_{\mathcal G}
    :=
    \{(i,j): i\sim j \text{ in } \mathcal G\},
\]
so that each undirected edge is represented by its two possible
orientations. For \(e\in \vec E_{\mathcal G}\), write \(t(e)\) and
\(h(e)\) for its tail and head, and let \(\bar e\) denote the reversed
edge. The non-backtracking matrix \( B\) of \(\mathcal G\) is
defined by
\[
   B_{e,f}
    =
    \begin{cases}
    1, & \text{if } h(e)=t(f) \text{ and } f\neq \bar e,\\
    0, & \text{otherwise}.
    \end{cases}
\]
Thus \( B\) records admissible transitions between directed edges,
while forbidding immediate reversal.

Now consider a self-adjoint periodic Jacobi operator of the form
\eqref{eq.jacobi-operator} on the universal cover \(\mathcal T\). The
diagonal Green functions on forward branches satisfy a closed system of
recursion relations. More precisely, for a directed edge \(e=(i,j)\), let \(m_{ij}(z)\) denote
the diagonal Green function at the root \(j\) of the forward subtree
reached by traversing \(e\), with the branch leading back to \(i\) removed.
By the Schur complement formula on trees, these quantities satisfy
\begin{equation}\label{eq.QVE-NB}
    -\frac{1}{m_{ij}(z)}
    =
    z-a_j
    +
    \sum_{\substack{k\sim j\\ k\ne i}}
    b_{jk}^2 m_{jk}(z).
\end{equation}
Equivalently, if one indexes the vector \(m\) by directed edges, then
\eqref{eq.QVE-NB} has the Dyson-equation form
\[
    -\frac{1}{m_e}
    =
    z-a_{h(e)}
    +
    \sum_{f\in \vec E_{\mathcal G}} S_{e f} m_f,
\]
where
$
    S_{e f}
    =
    b_{t(f)h(f)}^2\, \mathbf 1_{\{h(e)=t(f),\, f\neq \bar e\}}$.
Thus the matrix \(S\) is a weighted non-backtracking matrix on directed
edges. In general it is non-symmetric, even though the original Jacobi
operator \eqref{eq.jacobi-operator} is self-adjoint. The non-symmetry arises because the forward
Green function recursion excludes immediate backtracking.

Vector Dyson equations also appear in the study of quantum ergodicity on
large graphs. In this context, self-consistent equations for Green
functions on limiting tree-like structures are used to analyze spatial
delocalization and quantum ergodicity of graph eigenfunctions; see
\cite{AnantharamanSabri2019}.

A further important extension is the matrix-valued analogue, known as the
\emph{matrix Dyson equation}. This equation arises naturally in random
matrix models with correlated entries and has become a central
deterministic object in the proof of local laws and spectral universality.
We refer to
\cite{Erdos2019MDE,ajanki2019stability,erdHos2019random} for overviews of
the matrix Dyson equation and its applications to random matrices. Related
operator-valued and noncommutative versions of Dyson equations have been
studied in \cite{HeltonRashidiFarSpeicher2007,Mai2022Dyson}, where
positivity constraints, operator-valued free probability, and regularity
properties of the associated densities of states play a central role.

\subsection{Challenges, New Ideas and Contributions}
A central task in our analysis is to control the difference between
solutions of the Dyson equation and its perturbed counterpart.
Specifically, suppose that $g(z)$ satisfies the perturbed Dyson equation
\begin{equation}\label{eq.perturbed DE}
    -\frac{1}{g(z)}
    =
    z\mathbf{1}-\mathbf{a}+Sg(z)+\mathcal{R}(z),
\end{equation}
where $\mathcal{R}(z)$ is a small error term, and let $m(z)$ denote the
solution of the unperturbed equation \eqref{eq:QVE}.
Our goal is to estimate
$
    u(z):=g(z)-m(z)
$
in terms of the perturbation $\mathcal{R}(z)$.
Subtracting the unperturbed Dyson equation from the
perturbed one, we obtain the exact identity
\[
    \bigl(I-m(z)^2S\bigr)u(z)
    =
    m(z)^2\mathcal{R}(z)+\frac{u(z)^2}{g(z)},
\]
where all vector operations are understood componentwise, that is, $m(z)^2\mathcal{R}(z)=(m_i(z)^2\mathcal{R}_i(z))_{i=1}^d$, $u(z)^2/g(z):=(u_i(z)^2/g_i(z))_{i=1}^d$ and
\begin{equation}\label{eq.stability-matrix}
    I-m(z)^2S
    :=
    I-\operatorname{diag}(m_1(z)^2,\dots,m_d(z)^2)S.
\end{equation}

Thus, a key ingredient in controlling $u(z)$ is the analysis
of the linearized stability matrix $I-m(z)^2S$.
When this matrix has a bounded inverse, a small perturbation
$\mathcal{R}(z)$ leads to controlled deviations $u(z)$.
However, near edges and cusps of the limiting density,
the stability matrix becomes nearly singular, and small
errors may be strongly amplified.
Understanding this loss of linear stability, together with
the nonlinear terms in the equation for $u(z)$, is therefore
essential for obtaining sharp stability estimates in these regimes.

\textit{Linear Stability.}
To analyze the linearized stability operator, we use the phase--modulus
decomposition introduced in
\cite{AEK2017singularities,AEK2019quadratic}.
Define
\begin{equation}\label{def.B}
    R(z):=U(z)-F(z), \qquad\mbox{where }
    U(z):=\diag\left(\frac{|m_1(z)|^2}{m_1(z)^2},\dots,\frac{|m_d(z)|^2}{m_d(z)^2}\right)
\end{equation}
and
\begin{equation}\label{def.F}
    F(z):=\diag(|m_1(z)|,\dots,|m_d(z)|)~S \diag(|m_1(z)|,\dots,|m_d(z)|).
\end{equation}
The matrix \(R(z)\) is equivalent to the original linearization up to
multiplication by diagonal matrices:
\begin{equation}\label{eq.I-m2S-transform}
    I-m(z)^2S
    =
    \diag\left(\frac{m_1(z)^2}{|m_1(z)|},\dots,\frac{m_d(z)^2}{|m_d(z)|}\right)
    R(z)
    \diag\left(\frac{1}{|m_1(z)|},\cdots,\frac{1}{|m_d(z)|}\right).
\end{equation}

Although the stability estimates proved in this paper have the same
general form as their analogues in the symmetric setting, the underlying
arguments are substantially different. Several elementary structural
features used in the symmetric theory fail once \(S\) is allowed to be
non-symmetric.

The first difficulty concerns the nonnegative matrix $F(z)$.
This matrix plays a central role in the analysis of \(R(z)\) and hence of
the linearized Dyson equation. When \(S\) is symmetric, \(F(z)\) is
symmetric as well, and \cite[Lemma 4.6]{AEK2019quadratic} gives the
estimate
\[
    \|F(z)\|_2\le 1,\qquad z\in\mathbb C_+ .
\]
This estimate implies, in particular, that the solution \(m(z)\) can
become large only when \(z\) approaches one of the components of
\(\mathbf a\), thereby ruling out blow-up at unrelated spectral locations;
see \cite[Lemma 4.5]{AEK2019quadratic}. The symmetry of \(F\) also yields
singular-gap estimates \cite[Lemma 5.5 and Lemma 5.6]{AEK2017singularities},
which are essential inputs in the stability analysis of \(R\).

These inputs are no longer available in the non-symmetric setting. Indeed,
although \(\|U(z)\|_2=1\) by construction, the matrix \(F(z)\) is no longer
self-adjoint and its largest singular value may exceed one. Thus the basic
comparison between \(U\) and \(F\) underlying the symmetric argument breaks
down, and the singular-gap estimate is lost. Consequently, the key lemma
\cite[Lemma 5.8]{AEK2017singularities} used to deduce stability in the
symmetric case is no longer applicable. The same obstruction also appears
in the cusp analysis: without additional structure, the coefficient of the
cubic term in the effective scalar equation may degenerate, and the
resulting cubic equation need not have the stability properties familiar
from the symmetric theory.

At a more basic level, one should not expect a satisfactory stability
theory for an arbitrary nonnegative matrix \(S\). For example, if \(S\) is
a cyclic permutation matrix, the Dyson equation still admits the
semicircle solution and hence has a strictly positive bulk density, but
the stability matrix \(I-m(z)^2S\) becomes unstable at a bulk point; see
Appendix~\ref{app.Counterexamples} for details. This example shows that
bulk stability is not a consequence of nonnegativity and irreducibility
alone.

Motivated by this obstruction, our stability analysis focuses on two
structurally natural classes of matrices: symmetric matrices and
non-backtracking matrices. As explained in
Subsection~\ref{sec.motivation}, the symmetric case arises from
Wigner-type random matrices, while the non-backtracking case originates
from periodic Jacobi operators on universal covers. These two classes
provide the principal motivating examples for the present work.

Our first new ingredient is a preliminary structural analysis of the
solution itself; we refer to Theorem~\ref{thm.decompose of measure} for details. 
Since the non-symmetric equation does not reveal the
possible blow-up locations as explicitly as in the symmetric case, we use
algebraic and Puiseux-type arguments, detailed in Lemma \ref{lem:puiseux}, to study the boundary behavior of
\(m\) near the real axis. This analysis yields a finite exceptional set
and a decomposition of the representing measures into absolutely
continuous parts and finitely many atoms. Thus, even though the possible
blow-up locations cannot be read off directly from the equation, they are
shown to form a finite set. The same argument also shows that the regular
support of the density is a finite union of intervals. This provides the
global qualitative framework needed for the subsequent stability analysis.

The second new ingredient is a replacement for the missing Hilbert-space
estimate \(\|F\|_2\le 1\). Instead of relying on singular values of \(F\),
we use the left and right Perron--Frobenius eigenvectors to normalize
\(F\) into a Markov transition matrix. The
spectral analysis of \(I-F\) is then reduced to a mixing-space estimate for this
Markov operator on the complement of the Perron direction. This allows us
to recover the necessary gap information even though \(F\) is
non-self-adjoint and may have large singular norm. We refer to Lemma \ref{lem:Q-second-singular-gap-lambda} and Proposition \ref{pro.I-F-second} for details.

The third new ingredient is a graph-theoretic analysis of the phase matrix
\(U\). The possible small singular values of
 $   R=U-F$
are governed by resonances of the directed support graph of \(S\). We
introduce a resonance set, detailed in Definition \ref{def.resonance set}, that characterizes the phase configurations for
which the stability matrix may become singular. This perspective replaces
the direct norm comparison used in the symmetric theory and identifies the
graph structures responsible for possible bulk instabilities or
degeneracies of the cusp coefficient. In this way, the stability problem
is reduced to a structural question about the support graph of \(S\).
We refer to Lemma \ref{lem:UminusQ-graph-resonance}, Proposition \ref{prop.bulk-stability-symmetric} and Proposition \ref{pro.bulk-stability-NB} for details.

The resonance method is flexible enough to treat both classes considered
in this paper. Although our main motivation comes from the
non-backtracking setting, we also revisit the symmetric case. This serves
two purposes. First, it shows that the graph-theoretic resonance approach
is not specific to non-backtracking matrices, but provides a general
mechanism for stability analysis. Second, it allows us to remove the
primitivity assumption often imposed in the symmetric theory, thereby
yielding a complete characterization of stability in the symmetric
setting.

\textit{Nonlinear Stability.}
Another challenge arises in the nonlinear stability analysis near regular edges and cusps, where different solution branches of the Dyson equation approach one another.
In these regimes, a small perturbation $\mathcal{R}(z)$ in the Dyson equation does not necessarily imply that an approximate solution $g(z)$ is close to the desired solution $m(z)$, since $g(z)$ may lie near a different branch. 
To overcome this difficulty, we first determine the distances between different solution branches in these regimes. We then develop a branch-selection argument based on continuity: starting from a point in the upper half-plane where $g(z)$ is known to be close to $m(z)$, we use a continuous path toward the singular energy, together with sufficiently small perturbation $\mathcal{R}(z)$ along the path, to ensure that $g(z)$ remains on the correct branch. 
This allows us to establish quantitative nonlinear stability estimates even in regions where the linearized Dyson equation becomes nearly singular.

\subsection{Preliminaries}\label{sec.preliminary}
In this subsection, we collect the notation, conventions, and auxiliary results used throughout the paper.

\paragraph{Notation} We denote by $\R^d$ and $\C^d$ the spaces of real and complex valued vectors of dimension $d$, respectively, and by $\R^{d\times d}$ and $\C^{d\times d}$ the corresponding spaces of real and complex matrices. 
We write $|\cdot|_2$ for the Euclidean norm on $\R^d$ or $\C^d$, and $\|\cdot\|_2$ for the operator norm. 
Furthermore, $\langle\cdot,\cdot\rangle$ denotes the standard complex inner product on $\mathbb C^d$, namely
$
\langle x,y\rangle:=\sum_{i=1}^d \overline{x_i}\,y_i .
$
For a vector $u\in\mathbb C^d$, we use the shorthand notation
\[
\langle u\rangle:=\langle \mathbf 1,u\rangle=\sum_{i=1}^d u_i,
\qquad \mathbf 1:=(1,\dots,1).
\]
In particular, if $x,y\in\mathbb R^d$, then
$
\langle x,y\rangle=\langle xy\rangle,
$
where $xy$ denotes the componentwise product.
We denote by
$
    \mathbb S^1:=\{z\in\mathbb C:\ |z|=1\}
$
the unit circle in the complex plane.
For \(\tau\in\mathbb R\) and \(r>0\), we denote by
\[
    N_r(\tau):=\{z\in\overline{\mathbb C_+}: |z-\tau|\le r\}
\]
the closed upper half-disc of radius \(r\) centered at \(\tau\). 
Actually, all neighborhoods and domains considered in this paper are understood as subsets of \(\overline{\mathbb C_+}\).
We write \(\eta:=\im z\) for the imaginary part of \(z\).

We denote the spectral radius of a matrix by $\rho(\cdot)$. The smallest singular value with respect to the Euclidean norm \(|\cdot|_2\) is denoted by  \(s_{\rm min}(\cdot)\). For a square matrix or linear operator, $\Spec(\cdot)$ denotes the set of all eigenvalues and $\Ran(\cdot)$ denotes its image.

 We write $\spn\{v_1,\dots,v_k\}$ for the linear span of the vectors $v_1,\dots,v_k$ over the underlying scalar field. 
For a vector $v=(v_1,\dots,v_d)\in\mathbb C^d$, we denote by  $\supp v := \{\, i\in\{1,\dots,d\}: v_i\ne 0\,\}$ the support of $v$.
For a nonnegative matrix $A\in\R^{d\times d}$, we define its support by
\begin{equation}\label{def.E}
  E_A:=\{(i,j)\in\{1,\dots,d\}^2:\ A_{ij}>0\}.
\end{equation}
A directed cycle \(\mathscr{C}\) of \(E_A\) means a sequence
\begin{equation}\label{def.cycle}
  \mathscr{C}:=i_0\to i_1\to \cdots \to i_{\ell-1}\to i_\ell=i_0
\end{equation}
such that \((i_j,i_{j+1})\in E_A\) for all \(j=0,\dots,\ell-1\).
We associate to each nonnegative matrix $A\in\mathbb R^{d\times d}$ its support graph $\mathcal{G}_A$, defined as the directed graph on $\{1,\dots,d\}$ with an edge $i\to j$ if and only if $A_{ij}>0$. We shall freely pass between the matrix $A$ and its support graph when referring to graph-theoretic notions such as irreducibility, periods, and cycles. For an undirected graph \(\mathcal G\), the notation \(x\sim y\) means that \(x\) and \(y\) are adjacent vertices in \(\mathcal G\).

The symbol $\sqcup$ denotes a disjoint union; unlike the usual union symbol $\cup$, it emphasizes that the sets being joined are pairwise disjoint.

\paragraph{Conventions}
Products, quotients, and powers of vectors are understood componentwise. Thus, for vectors $u,v\in\mathbb C^d$ with nonzero entries of $v$,
\[
uv=(u_1v_1,\dots, u_dv_d),\qquad 
\frac{u}{v}=\left(\frac{u_1}{v_1},\dots,\frac{u_d}{v_d}\right).
\]
For matrices, quotients are interpreted by multiplication with the inverse: 
$
A/B:=AB^{-1},
$
whenever $B$ is invertible. 
Similarly, scalar functions applied to vectors are understood componentwise. Thus,
for \(v\in\mathbb C^d\), quantities such as
$
  \re v,\im v
$
and $|v|$
are defined by applying the corresponding scalar operation to each
component. For example,
\[
  \re v=(\re v_1,\dots,\re v_d),\quad
  \im v=(\im v_1,\dots,\im v_d),\quad |v|=(|v_1|,\dots,|v_d|).
\]
The componentwise modulus \(|v|\) should not be confused with $|v|_2$ which denotes the Euclidean norm of \(v\).
Inequalities between vectors or matrices are also understood componentwise. In particular, for a vector $v=(v_i)_{i=1}^d\in\R^d$ or a matrix $A=(a_{ij})_{i,j=1}^d\in\R^{d\times d}$, the inequalities $v\ge0$ and $A\ge 0$ mean $v_i\ge0$ for all $i$ and  $a_{ij}\ge 0$ for all $i,j$. The notation $v>0$ and $A>0$ is defined similarly. 

When a directed edge is used as a subscript, we write \(ij\) for the directed edge \((i,j)\); for instance, \(m_{ij}\) denotes \(m_{(i,j)}\).

For a vector $v=(v_1,\dots,v_d)\in\C^d$, we write
\[
D(v):=\operatorname{diag}(v_1,\dots,v_d)
\]
for the associated diagonal matrix. When no confusion can arise, we also identify a vector with its associated diagonal matrix in matrix products; for example, $vA$ means $D(v)A$.

For positive quantities $f$ and $g$, we write $f\lesssim g$ if $f\le Cg$ for some constant $C>0$ depending only on fixed model parameters and absolute constants. 
Throughout the paper, $c$ and $C$ denote finite positive constants that may depend only on absolute constants and fixed model parameters, and their values may change from line to line.
We also write $f\gtrsim g$ if $g\lesssim f$, and $f\sim g$ if both relations hold. In particular, $f\sim 1$ means that $f$ is bounded above and below by positive constants depending only on fixed model parameters and absolute constants.

\paragraph{Auxiliary Results}
The following two standard facts from Perron--Frobenius theory will be used throughout the paper. The first is the cyclic decomposition, or equivalently the irreducible imprimitive normal form, of a nonnegative matrix; see \cite[Chapter 3]{BrualdiRyser1991} and \cite[Chapter 2]{BermanPlemmons1994}. The second is the Perron--Frobenius theorem for irreducible nonnegative matrices; see \cite[Chapter 2]{BermanPlemmons1994} or \cite[Chapter 8]{HornJohnson2013}.

\begin{lemma}[Cyclic decomposition and eventual positivity]\label{lem:cyclic-decomp}
Let $A\in\mathbb R^{d\times d}$ be a nonnegative and irreducible matrix. Define the \emph{period} of $A$ by
\begin{equation*}\label{eq:period-def}
d_* \;:=\; \gcd\bigl\{\, n\ge 1:\ (A^n)_{ii}>0\,\bigr\},
\end{equation*}
where $i\in\{1,\dots,d\}$ is arbitrary.
Then there exists a partition
\[
\{1,\dots,d\}=\mathcal C_1\sqcup\mathcal C_2\sqcup\cdots\sqcup\mathcal C_{d_*}
\]
such that, after a simultaneous permutation of rows and columns, $A$ has the cyclic block form
\begin{equation*}\label{eq:cyclic-form}
A=
\begin{pmatrix}
0 & A_{12} & 0 & \cdots & 0\\
0 & 0 & A_{23} & \cdots & 0\\
\vdots & & \ddots & \ddots & \vdots\\
0 & 0 & \cdots & 0 & A_{d_*-1,d_*}\\
A_{d_*,1} & 0 & \cdots & 0 & 0
\end{pmatrix},
\end{equation*}
where each block $A_{i,i+1}$ is nonzero and the index $i$ is understood cyclically modulo $d_*$.

Consequently, we have
\begin{equation*}\label{eq:Sd-block-diag}
A^{d_*}=\mathrm{diag}(A_1,\dots,A_{d_*}),
\qquad 
A_i:=A_{i,i+1}A_{i+1,i+2}\cdots A_{i+d_*-1,i}.
\end{equation*}
Here each $A_i$ is primitive, i.e., for every $i$ there exists $n_i\ge 1$ such that
$
A_i^{n_i}>0.
$
\end{lemma}

\begin{lemma}[Perron-Frobenius Theorem]\label{lem.PF}
Let \(A\in \mathbb{R}^{d\times d}\) be an irreducible nonnegative matrix. Then the following statements hold:
\begin{enumerate}
    \item The spectral radius \(\rho(A)\) is a positive eigenvalue of \(A\), which is denoted by $\lambda_{\max}(A)$.
    Moreover, every eigenvalue on the spectral circle \(|\lambda|=\rho(A)\) is simple.

    \item There exist left and right eigenvectors \(l>0\) and \(r>0\) associated with \(\lambda_{\max}(A)\) such that
    \[
    Ar=\lambda_{\max}(A)\,r,
    \qquad
    l^\top A=\lambda_{\max}(A)\,l^\top.
    \]
\end{enumerate}
\end{lemma}

\subsection*{Acknowledgements}
We thank Torben Kr{\"u}ger for helpful discussions, and Charles Bordenave and Wenbo Li for informing us of the earlier work of Nagnibeda and Woess~\cite{random2002nagnibeda}. R.X. gratefully acknowledges the hospitality and support of the University of Pennsylvania during his visit. The research of Z.S. was partially supported by NSFC grants Nos. 12271475 and U23A2064. The research of J.H. was partially supported by NSF grants DMS-2246664 and DMS-2331096, and by a Sloan Research Fellowship.

\subsection*{AI Statement}
This project began in September 2025, during Ruizhe Xu's visit to Jiaoyang Huang at the University of Pennsylvania. All the mathematical ideas and proof strategies underlying the present work had been developed and discussed by the end of this visit in February 2026.

The use of AI tools in preparing this manuscript was limited to improving the presentation and language, generating figures (all figures in this paper were created using AI), and assisting with the verification of proofs. All mathematical ideas, proof strategies, and scientific contributions originated from the authors.

\section{Main results}\label{sec:main}

We begin with the basic well-posedness result for the Dyson equation. When the matrix $S$ is nonnegative, the equation \eqref{eq:QVE} admits, for each $z\in\Cp$, a unique solution in $\Cp^d$. This solution depends analytically on $z$ and each of its components admits a Stieltjes transform representation. Moreover, the solution is automatically uniformly bounded away from the real axis, and any possible singular behavior can only occur as \(\eta\downarrow0\).

\begin{theorem}[Existence and uniqueness]\label{thm:Existence and uniqueness}
Suppose $S$ is nonnegative. For each $z\in \Cp$, the Dyson equation \eqref{eq:QVE} admits a unique solution $m(z)\in \Cp^d$. Moreover, the family of solutions defines an analytic map
\[
m:\Cp\to\Cp^d,\qquad z\mapsto m(z).
\]

Furthermore, for each component $1\le i\le d$, there exists a unique probability measure $\tilde{v}_i(\dd\tau)$ such that
\begin{equation}\label{def.v}
    m_i(z)=\int_{\R} \frac{\tilde{v}_i(\dd \tau)}{\tau-z}, \qquad   z \in \Cp.
\end{equation}
All these measures are supported in the compact interval $[-\Sigma,\Sigma]$ where
\begin{equation}\label{def.Sigma}
\Sigma:=|\mathbf {a}|_2 + 2 \|S\|_\infty^{1/2}.
\end{equation}
In particular, each component satisfies the elementary upper bound
\begin{equation}\label{eq.rough-boundedness}
     |m_i(z)|\le \eta^{-1},\qquad z\in\Cp.
\end{equation}
\end{theorem}

For the proof of the above theorem, we refer the reader to \cite[Section 4]{AEK2019quadratic}. Although an additional symmetry assumption is imposed there, the argument uses only the nonnegative property.

\medskip

We next introduce a structural assumption on the matrix \(S\). 
\begin{definition}
    A nonnegative
matrix \(A\) is called \emph{irreducible} if, for any
\(i,j\in\{1,\dots,d\}\), there exists an integer \(n=n(i,j)\ge 1\) such that
$
    (A^n)_{ij}>0 .
$
\end{definition}
In finite dimension, an irreducible nonnegative matrix admits the classical
cyclic decomposition; see Lemma~\ref{lem:cyclic-decomp} for details.

Actually, the irreducibility assumption is natural in what follows. If \(S\) is
reducible, then the Dyson equation can be triangularized according to the irreducible
components of \(S\) and studied recursively block by block. Thus the
irreducible case is the basic indecomposable case.

\begin{theorem}[Decomposition of measures]\label{thm.decompose of measure}
  Suppose $S\in\R^{d\times d}$ is nonnegative and irreducible. 
  The singular set
  \[
  \mathcal A_j
  :=
  \Big\{\tau\in\mathbb R:
  \lim_{\eta\downarrow0}|m_j(\tau+\mathrm i\eta)|=\infty
  \Big\}
  =
  \Big\{\tau\in\mathbb R:
  \lim_{\eta\downarrow0}\im m_j(\tau+\mathrm i\eta)=\infty
  \Big\}
  \]
  is finite.
  For every $\tau\notin\mathcal{A}_j$, the boundary value
  \begin{equation}\label{def.mj}
     m_j(\tau):=\lim_{\eta\downarrow0}m_j(\tau+\mathrm i\eta)
  \end{equation}
  exists and is finite. In particular, if $m(z)$ is uniformly bounded, then each component \(m_i(z)\), originally defined on \(\Cp\), admits a unique and continuous extension to \(\overline{\Cp}\).

  The probability measure $\tilde{v}_j(\dd\tau)$ defined in \eqref{def.v} has the following decomposition: there exists a finite collection of real numbers $\{\tau_{j,k}\}_{k=1}^{N_j}\subset\mathcal{A}_j$ such that
  \begin{equation}\label{eq.measure-decomposition}
    \tilde{v}_j(\dd\tau)=v_j(\tau)\dd\tau+\sum_{k=1}^{N_j}\Delta_{j,k}\delta_{\tau_{j,k}}(\dd\tau),
  \end{equation}
  where the atomic weights $\Delta_{j,k}$ and the possibly $\infty$-valued density $v_j(\tau)$ are determined by the boundary values of \(m_j\):
  \begin{equation}\label{def.vj-Delta}
    \Delta_{j,k}:=\lim_{\eta\downarrow 0}\eta\im m_j(\tau_{j,k}+\mathrm i\eta),\quad v_j(\tau):=\frac{1}{\pi}\lim_{\eta\downarrow 0}\im \Big(m_j(\tau+\mathrm i\eta)-\sum_{k=1}^{N_j}\frac{\Delta_{j,k}}{\tau_{j,k}-(\tau+\mathrm i\eta)}\Big).
  \end{equation}
  The regular support of the generating density $v_j(\tau)$ is defined by
  \begin{equation}\label{def.Sj}
    \mathfrak{S}_j
    :=
    \bigl\{\tau\notin\mathcal{A}:~ v_j(\tau)>0\bigr\},\qquad \mathcal{A}:=\cup_{j=1}^d\mathcal{A}_j.
  \end{equation}
  Then $v_j(\tau)$ is continuous on $\mathfrak{S}_j$. Moreover, $\mathfrak{S}_j$ is independent of $j$. We therefore suppress the component index and write $\mathfrak{S}$. Finally, \(\mathfrak{S}\) is a finite union of open intervals.
\end{theorem}

We now classify points on the real line using the notation of the preceding theorem.

\begin{definition}[Regular bulk, edge and cusp points]
For a point \(\tau\in\mathbb R\), we classify it as follows.
\begin{itemize}
  \item \emph{Regularity.}
  A point $\tau\in\mathbb R$ is called a \emph{non-regular} point if 
  $
    \langle \im m(\tau)\rangle=\infty.
  $
  Otherwise, \(\tau\) is called a \emph{regular} point.
  \item \emph{Types of regular points.}
  Regular points in $\overline{\mathfrak{S}}$ are further classified into the following three types.
  \begin{itemize}
    \item \textbf{Bulk.}
    A regular point $\tau$ is called a bulk point if $\tau\in\mathfrak{S}$.

    \item \textbf{Edge.}
    A regular point \(\tau\in\partial\mathfrak{S}\) is called a regular edge if it is an endpoint of exactly one open interval of
    \({\mathfrak S}\).

    \item \textbf{Cusp.}
    A regular point \(\tau\in\partial\mathfrak{S}\) is called a regular cusp if it is not an edge. In this case, \(\tau\) is the shared endpoint of two open intervals of  \({\mathfrak S}\).
  \end{itemize}
\end{itemize}
  By the definition of \(\mathfrak S\) in Theorem \ref{thm.decompose of measure}, if a regular point \(\tau\) is either an edge or a cusp, then necessarily
  $
    \langle \im m(\tau)\rangle=0.
  $
\end{definition}

\begin{figure}[htbp]
    \centering
    \includegraphics[scale=0.2]{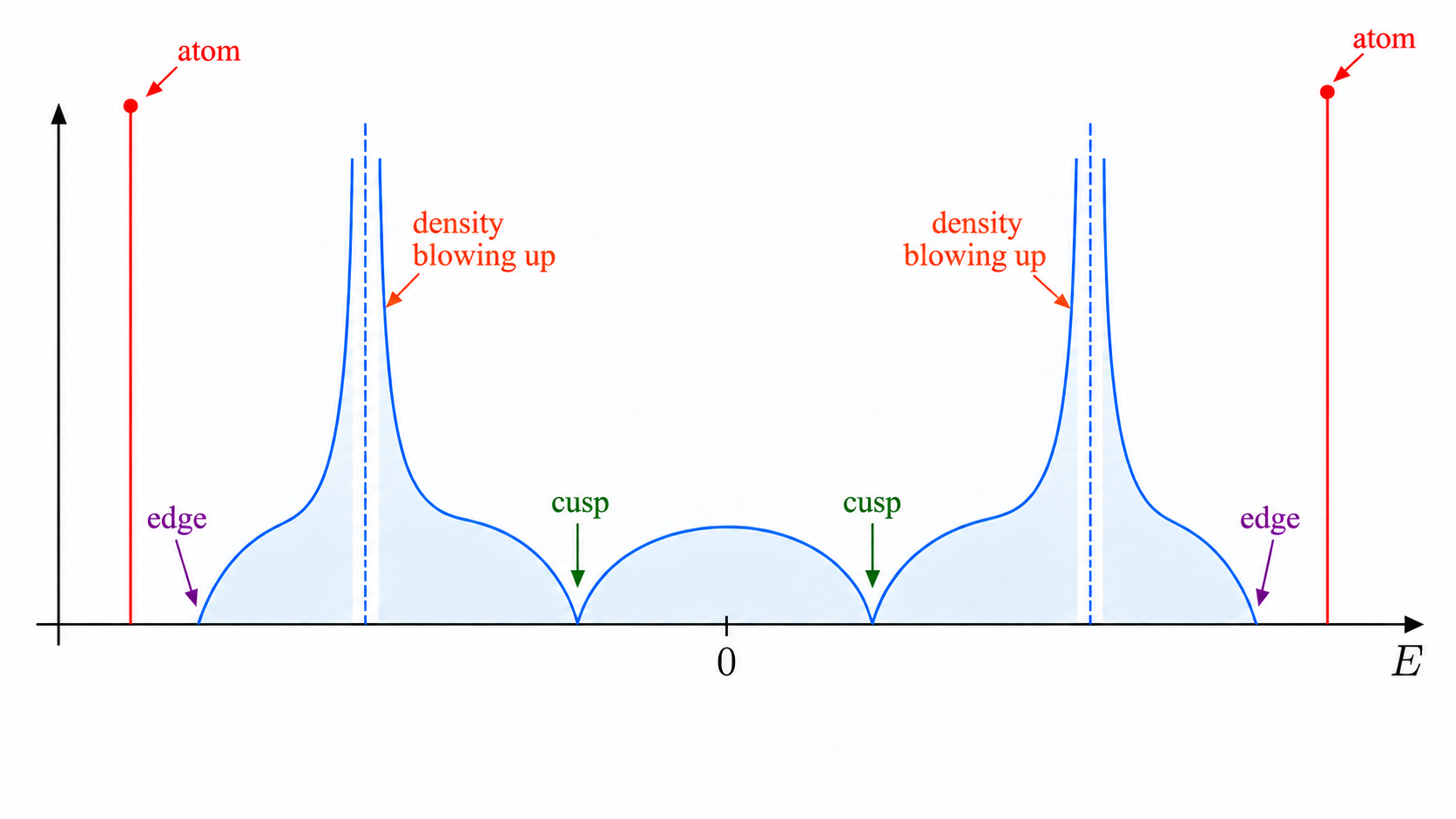}
    \caption{Classification of the local behaviors of the limiting spectral density.}
    \label{fig:DE}
\end{figure}

Before stating the next results, we introduce two pieces of notation. For an 
interval \(I\subset\mathbb R\), we set
\[
    \mathbb C_+(I):=\{z\in\mathbb C_+:\re z\in I\}.
\]
Thus \(\mathbb C_+(I)\) is the part of the complex upper half-plane lying above
the interval \(I\).

We say that the solution \(m(z)\) is \emph{uniformly bounded} on \(\mathbb C_+(I)\) if
there exists a constant \(\Phi_I<\infty\) such that
\[
    \sup_{z\in\mathbb C_+(I)} |m(z)|_2\le \Phi_I .
\]
Here \(\Phi_I\) is regarded as a model parameter.
Whenever this boundedness condition holds, Theorem~\ref{thm.decompose of measure} implies
that \(m(z)\) extends uniquely and continuously from \(\mathbb C_+(I)\) to $\overline{\mathbb C_+}(I)$.
We shall henceforth identify \(m\) with this extension whenever the boundedness condition is in force.

\medskip
The following result shows that local boundedness already implies the correct
componentwise size of \(m\), as well as comparability of the imaginary parts
across all components.
\begin{theorem}[Boundedness of Solution]\label{thm.bound}
    Suppose $S$ is nonnegative and irreducible, and the solution $m(z)$ is uniformly bounded on $\C_+(I)$, where $I\subset\R$ is a closed interval.
 Then we have, for every $1\le i\le d$,
  \begin{equation}\label{eq.m-bound}
    |m_i(z)|\sim\frac{1}{1+|z|},\qquad z\in\overline{\Cp}(I).
  \end{equation}
  Moreover, the imaginary parts of all components of the solution $m(z)$ are comparable, i.e., for every $1\le i<j\le d$,
  \begin{equation}\label{eq.m-Im}
    \im m_i(z)\sim\im m_j(z),\qquad z\in\overline{\Cp}(I).
  \end{equation}
\end{theorem}

\medskip
The following theorem shows that on intervals disjoint from the singular set \(\mathcal A\), local uniform
boundedness yields absolute continuity of the measures and H\"older regularity of both the densities and the solution.

\begin{theorem}[Regularity of Solution]\label{thm.regularity}
  Suppose $S$ is nonnegative, irreducible and the solution $m(z)$ is uniformly bounded on $\overline{\C_+}(I)$, where $I\subset\R$ is a compact interval.
  Then we have
  \begin{enumerate}
    \item On the interval \(I\), the measure \(\widetilde v_i\) has a Lebesgue density, that is,
    $
        \widetilde v_i(\dd \tau)=v_i(\tau)\,\dd \tau
    $
    for $\tau\in I$.
    Densities of all components are comparable, i.e., for every $1\le i<j\le d$,
    \begin{equation}\label{eq.v-comparable}
         v_i(\tau)\sim v_j(\tau),\qquad\tau\in I.
    \end{equation}

    \item The density is uniformly $\alpha$-H\"older continuous on $I$, i.e., for every $1\le i\le d$,
    \begin{equation}\label{eq.v-holder}
      \qquad|v_i(\tau_2)-v_i(\tau_1)|\lesssim|\tau_2-\tau_1|^\alpha,\qquad\tau_1,\tau_2\in I
    \end{equation}
    where $\alpha$ is a finite positive exponent depending only on model parameters.

    \item 
    The solution \(m_i(z)\) is uniformly $\alpha$-H\"older continuous, i.e., for every $1\le i\le d$,
    \begin{equation}\label{eq.m-holder}
      |m_i(z_1)-m_i(z_2)|\lesssim |z_1-z_2|^\alpha,\qquad z_1,z_2\in\overline{\Cp}(I).
    \end{equation}
  \end{enumerate}
\end{theorem}

\medskip

We now turn to the singularity and stability theory of the solution. The
results below will be stated for two classes of models which are relevant for
our later analysis. We first specify the assumptions defining these two model
classes.

\noindent\textbf{Assumption \((\mathrm{Sym})\).}
The matrix \(S\) is symmetric.

\medskip

\noindent\textbf{Assumption \((\mathrm{NB})\).}
The matrix \(S\) is the non-backtracking matrix associated with an undirected
connected base graph
$
\mathcal G=(\mathbb V_{\mathcal G},E_{\mathcal G}),
$
and the base graph has minimal degree at least three, i.e.
\[
\deg_{\mathcal G}(x)\ge 3,\qquad x\in\mathbb V_{\mathcal G}.
\]
Moreover, the vector \(\mathbf{a}\) in the Dyson equation \eqref{eq:QVE} is terminal-dependent: 
\[
a_{(x,y)}=a_y, \qquad (x,y)\in\vec E_{\mathcal G}.
\]

In the sequel, we shall simply say that \(S\) satisfies
\((\mathrm{Sym})\) or \((\mathrm{NB})\), instead of repeating these structural
conditions in every statement. 

\medskip

The next result describes the only possible local behaviors at regular
boundary points of the support. Under the structural
assumptions \(\mathrm{(Sym)}\) or \(\mathrm{(NB)}\), regular edges give rise to square-root
density growth, while regular cusps give rise to cubic-root
density growth.

\begin{theorem}[Singularity of Solution]\label{thm:singularities}
    Suppose that \(S\) is a nonnegative, irreducible matrix satisfying either \(\mathrm{(Sym)}\) or \(\mathrm{(NB)}\). Let \(\tau\in\partial\mathfrak S\) be a regular boundary point.
    \begin{itemize}
        \item If \(\tau\) is a regular edge, then each \(v_i\) exhibits a square-root singularity at \(\tau\). More precisely, for all \(1\le i\le d\), there exists a constant \(c_i>0\) depending only on the model parameters such that
        \begin{equation}
        \label{edge singularity}
        v_i(\tau\pm\omega)
        =
        c_i\omega^{1/2}
        +
        O(\omega),
        \qquad
        \omega\downarrow 0,
        \end{equation}
        where the sign is chosen according to whether \(\tau\) is the left or right endpoint.

        \item If \(\tau\) is a regular cusp, then each \(v_i\) exhibits a cubic-root singularity at \(\tau\). More precisely, for all $1\le i\le d$, there exists a finite constant \(c_i>0\) depending only on the model parameters such that
        \begin{equation}\label{cusp singularity}
        v_i(\tau+\omega)
        =
        c_i|\omega|^{1/3}
        +
        O\bigl(|\omega|^{2/3}\bigr),
        \qquad
        \omega\to 0.
        \end{equation}
    \end{itemize}
\end{theorem}

For historical context, we briefly recall an earlier square-root result in
a more specialized setting. Nagnibeda and Woess~\cite{random2002nagnibeda}
established square-root behavior at the outer spectral edges for
nearest-neighbor random walks on trees with finitely many cone types and
irreducible cone-type graph.

\medskip
The next theorem controls the inverse of the linear stability operator
\(I-m(z)^2S\) near edges, cusps, outside the support, and in the strict bulk, which reflects the local singularity scale of the density.
\begin{theorem}[Stability of linear operator]\label{thm.stability}
    Suppose $S$ is a nonnegative and irreducible matrix. For \(z=E+\mathrm i\eta\) and a regular boundary point
    \(\tau\in\partial\mathfrak S\), set
    $
        \kappa:=|E-\tau|.
    $
    Then there exists a sufficiently small constant $\epsilon>0$ depending only on model parameters such that the following estimates hold.
    \begin{itemize}
    \item \textit{Off-support Linear Stability.} Define the off-support regime $\mathfrak{O}:=\{E\in \R: \dist(E,\overline{\mathfrak{S}}\cup\mathcal{A})\ge\epsilon\}$. Then 
    \begin{equation}\label{eq.off-support-stability}
        \|(I-m(z)^2S)^{-1}\|_2\sim 1,\qquad z\in\C_+(\mathfrak{O}).
    \end{equation}
    \end{itemize}
    \vspace{-3mm}
    In addition, assume that \(S\) satisfies \(\textup{(Sym)}\) or \(\textup{(NB)}\). Then the following estimates hold.
    \begin{itemize}
        \item \textit{Edge Linear Stability.}
    If $\tau$ is a regular edge, then
    \begin{equation}\label{eq.edge-stability}
    \|(I-m(z)^2S)^{-1}\|_2
     \sim\max\big\{1,(\kappa+\eta)^{-1/2}\big\},\qquad z\in \C_+\big([\tau-\epsilon,\tau+\epsilon]\big).
    \end{equation}
        \item \textit{Cusp Linear Stability.} If $\tau$ is a regular cusp, then
        \begin{equation}\label{eq.cusp-stability}
          \|(I-m(z)^2S)^{-1}\|_2\lesssim  \max\big\{1,(\kappa+\eta)^{-2/3}\big\},\qquad z\in \C_+\big([\tau-\epsilon,\tau+\epsilon]\big).
        \end{equation}
        \item \textit{Bulk Linear Stability.} Define the strict bulk $\mathfrak{B}:=\{E\in\mathfrak{S}: \dist\bigl(E,\partial\mathfrak{S}\cup\mathcal A\bigr)\ge \epsilon\}$. If the relevant graph is non-bipartite or $\mathbf{a}\notin \R\mathbf 1$, then
        \begin{equation}\label{bulk-stability}
          \|(I-m(z)^2S)^{-1}\|_2\sim 1,\qquad z\in\C_+(\mathfrak{B}).
        \end{equation}
        In the exceptional case where the relevant graph is bipartite and $\mathbf{a}=a\mathbf{1}$ for some scalar $a\in\R$, one has
        \begin{equation}\label{bulk-stability-improved}
          \|(I-m(z)^2S)^{-1}\|_2\sim \max\big\{1,|z-a|^{-1}\big\},\qquad z\in\C_+(\mathfrak{B}).
        \end{equation}
        Here, in the \(\textup{(NB)}\) case the relevant graph is the base graph
        \(\mathcal G\), while in the \(\textup{(Sym)}\) case it is the underlying
        graph associated with \(S\).
  \end{itemize}
\end{theorem}

Theorem~\ref{thm.stability} gives a complete description of the linear stability behavior in the
symmetric case and in the non-backtracking matrix case for $z$ away from the singular points. 
Actually, besides the example in Appendix~\ref{app.Counterexamples}, \eqref{bulk-stability-improved} gives another example of the loss of bulk stability for a general matrix.
In the symmetric case, a more complete characterization of the uniform boundedness of
the solution is available in~\cite[Section 6]{AEK2019quadratic}, which in turn yields a complete description of the
singular set and makes the stability picture fully explicit. Since the main focus of this paper is the genuinely non-symmetric non-backtracking case, we do not discuss this additional symmetric theory in detail. We also note that, in the bipartite exceptional case, uniform boundedness near \(z=a\) can occur only when the two parts of the bipartition have the same number of vertices; otherwise a delta mass appears at \(z=a\); see \cite[Example 7.3]{AvniBreuerSimon2020} and \cite[Section 3.1]{BanksGarzaVargasMukherjee2022} for details.

Theorem \ref{thm.stability} shows that the linearized stability operator may become singular in the bipartite case, as well as near regular edges and cusps of the limiting density. 
The following main theorem addresses these exceptional regimes, where linear stability estimates alone are insufficient, by establishing nonlinear stability estimates. 
Since the arguments for the (Sym) and (NB) cases are essentially the same, we state and prove the result only for the (NB) case. For simplicity, we restrict our attention to the case
$\mathbf a\equiv 0$. The general case can be treated similarly.

Before stating the main theorem, we introduce
the Dyson residual, which quantifies the size of the perturbation in the Dyson equation, and the associated vertex quantities, which are used to handle the possible singularity at zero in the bipartite case.
Let $g(z)=(g_e(z))_{e\in\vec E_{\mathcal G}}$ be a vector with
nonzero entries, where $\mathcal G$ is the base graph with minimum degree at least three. Its Dyson residual $\mathcal R(g,z)$ is
defined as 
\begin{equation}\label{eq.Dyson-residual}
    \mathcal{R}(g,z):=-\frac{1}{g(z)}-z-Sg(z),
\end{equation}
so that $g(z)$ satisfies the perturbed Dyson equation \eqref{eq.perturbed DE}. It is clear that $\mathcal{R}(m,z)=0$.
We further define the associated vertex quantities
$g^{\rm ver}(z)=(g_i(z))_{i\in\mathbb V_{\mathcal G}}$ by
\begin{equation}\label{eq:q-vertex-reconstruction}
    g_i(z)
    :=
    -\frac{1}{z+\sum_{j:j\sim i}g_{ij}(z)},
    \qquad i\in\mathbb V_{\mathcal{G}}.
\end{equation}
We next specify the spectral domain for our main theorem.
For any fixed $\epsilon>0$, define
\[
    \mathcal I_\epsilon
    :=
    \left\{
    E\in[-2d_{\max},2d_{\max}]:
    \dist(E,\mathcal A)
    \ge\epsilon
    \right\}.
\]
Theorem \ref{thm:Existence and uniqueness} implies that the measures associated with the Dyson equation
are supported in $-\Sigma,\Sigma$, where
$
    \Sigma=|\mathbf a|_2+2\|S\|_\infty^{1/2}.
$
In the (NB) case with $\mathbf a=0$, the non-backtracking
matrix $S$ satisfies $\|S\|_\infty\le d_{\max}-1$ where $d_{\max}$ is the maximum degree of the base graph $\mathcal G$, and hence
$
    \Sigma\le 2\sqrt{d_{\max}-1}<2d_{\max}.
$
Therefore, the interval $[-2d_{\max},2d_{\max}]$ contains
the entire spectral support of the Dyson equation.
The restriction to $\mathcal I_\epsilon$ thus simply excludes
fixed sufficiently small neighborhoods of the singular energies within this
interval.
\begin{theorem}[Stability of the solution]\label{lem.stability2}
Suppose $S$ satisfies (NB) with $\mathbf a=0$.
For the edge/cusp assertion below, fix a height $\eta_0>0$.
There exists a sufficiently small constant $c_0>0$
depending only on the model parameters and $\eta_0$,
such that the following estimates hold for every
$z\in\Cp(\mathcal I_\epsilon)$ and every vector
$g(z)=(g_e(z))_{e\in\vec E_{\mathcal G}}$
with nonzero entries satisfying
$
    |g(z)-m(z)|_\infty\le c_0$ and $
    |\mathcal R(g,z)|_\infty\le c_0.
$
\begin{itemize}
\item
\emph{Stable bulk and off-support.}
If $\mathcal G$ is non-bipartite, or if $\mathcal G$ is bipartite and $z$ stays
away from $0$, then, uniformly for $z$ in the stable bulk and stable
off-support, 
\begin{equation}
\label{eq:nonlinear-bulk-stability}
|g(z)-m(z)|_\infty
\lesssim
|\mathcal{R}(g,z)|_\infty.
\end{equation}
If $\mathcal G$ is bipartite and $0$ is a regular bulk point, let
$T_i=\pm1$ on the two bipartition classes of $\mathbb V_{\mathcal G}$.
Then, for $z$ sufficiently close to $0$, provided that
$
\left|
\sum_{i\in\mathbb V_{\mathcal G}}T_i
\bigl(g_i(z)-m_i(z)\bigr)
\right|
\le c_0,
$
we have
\begin{equation}
\label{eq.bulk-0-stability}
|g(z)-m(z)|_\infty
\lesssim
|\mathcal{R}(g,z)|_\infty
+
\left|
\sum_{i\in\mathbb{V}_{\mathcal G}}T_i
\bigl(g_i(z)-m_i(z)\bigr)
\right|,
\end{equation}
where $m_i$ and $g_i$ are the vertex quantities defined in
\eqref{eq:q-vertex-reconstruction}, corresponding to the edge-indexed
vectors $m=(m_e)_{e\in\vec E_{\mathcal G}}$ and
$g=(g_e)_{e\in\vec E_{\mathcal G}}$, respectively.
\item
\emph{Regular edges and cusps.}
Let $\tau$ be a regular edge or cusp, and let
$z=E+\mathrm i\eta\in\Cp$ be sufficiently close
to $\tau$, with $\eta<\eta_0$.
Assume further that there exists a continuous family
$(g_t)_{\eta\le t\le\eta_0}$ satisfying
$g_\eta=g(z)$ and
\begin{equation}\label{eq.continuity}
\sup_{\eta\le t\le\eta_0}
|g_t-m(E+\mathrm i t)|_\infty\le c_0,
\qquad
\sup_{\eta\le t\le\eta_0}
|\mathcal{R}(g_t,E+\mathrm i t)|_\infty\le\mathcal{E}.
\end{equation}
Then, for $|z-\tau|$ sufficiently small,
\begin{equation}\label{eq:nonlinear-boundary-stability}
|g(z)-m(z)|_\infty
\lesssim
\begin{cases}
\displaystyle
\frac{\mathcal{E}}
     {(\mathcal E+|z-\tau|)^{1/2}},
&
\text{if $\tau$ is a regular edge},
\\[3mm]
\displaystyle
\frac{\mathcal{E}}
     {(\mathcal{E}+|z-\tau|)^{2/3}},
&
\text{if $\tau$ is a regular cusp}.
\end{cases}
\end{equation}
\end{itemize}
\end{theorem}

\section{Decomposition, boundedness and regularity}\label{sec.decomposition}

We prove Theorem \ref{thm.decompose of measure}, Theorem \ref{thm.bound} and Theorem \ref{thm.regularity} in this section. We first introduce the following Puiseux expansion lemma, which gives us an explicit expansion of solution $m(z)$.

\begin{lemma}[Proof of Proposition 4.2 in \cite{AnantharamanSabri2019}]\label{lem:puiseux}
  Suppose $S\in\R^{d\times d}$ is nonnegative and $m(z)$ is the solution to the Dyson equation \eqref{eq:QVE}. Every $m_i(z)$ has a convergent \emph{Puiseux expansion} in a neighborhood of any $\tau\in\R$. More precisely, for $z$ sufficiently close to $\tau$,
  \begin{equation}
    \qquad m_i(z)=\sum_{k\ge k_0} c_k(z-\tau)^{k/n},\qquad i=1,\dots, d
  \end{equation}
  for some $k_0\in\mathbb Z$ and some $n\in\N_+$, where the coefficients \(c_k\in\mathbb C\).
\end{lemma}

\bigskip
\begin{proof}[Proof of Theorem \ref{thm.decompose of measure}]

\textit{Step 1: Finite singular points.}
We first show that the possible singular points are finite. Fix $\tau\in\R$. By Lemma \ref{lem:puiseux}, for $z$ in a neighborhood of $\tau$, we have
\begin{equation}\label{eq.mj-puiseux}
   m_j(z)=\sum_{k\ge k_0}c_k(z-\tau)^{k/n}.
\end{equation}
 Consequently, the boundary value
$
  \lim_{\eta\downarrow0} m_j(\tau+\mathrm i\eta)
$
exists for every \(\tau\), possibly with value \(+\infty\). Moreover, in this neighborhood it can be infinite only at \(\tau\). 
Thus the set $\mathcal{A}_j$ is discrete. 
Since \([-\Sigma,\Sigma]\) can be covered by finitely many such neighborhoods, and each of them contains at most one point of \(\mathcal A_j\), the set \(\mathcal A_j\) is finite.
\medskip

\textit{Step 2: Well-defined boundary values.} 
For every $\tau\notin\mathcal{A}_j$, we have $\lim_{\eta\downarrow 0}|m_j(\tau+\mathrm i\eta)|<\infty$. 
This implies that the Puiseux expansion of \(m_j\) at \(\tau\) cannot contain any negative power. 
Otherwise, the leading negative-power term would dominate along \(z=\tau+\mathrm i\eta\) and force \(|m_j(\tau+\mathrm i\eta)|\to\infty\), contradicting \(\tau\notin\mathcal A_j\).
Hence the expansion has the form
\begin{equation}\label{eq.nonnegative-puiseux}
    m_j(z)=\sum_{k\ge0}c_k(z-\tau)^{k/n}  
\end{equation}
possibly after changing the index. Consequently, the limit in \eqref{def.mj} exists and is finite, and the resulting boundary values define the unique continuous extension of \(m_j\) to \(\overline{\mathbb C_+}\).

Next we show $\mathcal{A}_j=\big\{\tau\in\mathbb R:
  \lim_{\eta\downarrow0}\im m_j(\tau+\mathrm i\eta)=\infty
  \big\}$. For every \(\tau\in\mathcal A_j\), the preceding analysis in this step shows
  that the leading power in the Puiseux expansion of \(m_j\) at \(\tau\) is
  $
  \gamma:=k_0/n<0.
  $
  Along \(z=\tau+r e^{\mathrm i\theta}\),
the leading term of \eqref{eq.mj-puiseux} gives
\[
  \im m_j(\tau+r e^{\mathrm i\theta})
  =
  r^\gamma \im \bigl(c_{k_0}e^{\mathrm i\gamma\theta}\bigr)
  +o(r^\gamma).
\]
Since \(\im m_j(z)>0\) when $z\in \Cp$, 
we must have $\im \bigl(c_{k_0}e^{\mathrm i\gamma\theta}\bigr)\geq 0$ if $0<\theta<\pi$. Thus it follows that if \(\theta=\pi/2\), then $\im \bigl(c_{k_0}e^{\mathrm i\gamma\pi/2}\bigr)> 0$. Since \(\gamma<0\), we have
\[
  \im m_j(\tau+\mathrm i\eta)
  =
  \eta^\gamma
  \im \bigl(c_{k_0}e^{\mathrm i\gamma\pi/2}\bigr)
  +o(\eta^\gamma)
  \to\infty.
\]
  This implies $\mathcal{A}_j\subset\big\{\tau\in\mathbb R:
  \lim_{\eta\downarrow0}\im m_j(\tau+\mathrm i\eta)=\infty
  \big\}$. The other direction is trivial as $\im m_j\le |m_j|$, so we have $\mathcal{A}_j=\big\{\tau\in\mathbb R:
  \lim_{\eta\downarrow0}\im m_j(\tau+\mathrm i\eta)=\infty
  \big\}.$

\medskip

\textit{Step 3: Measure decomposition.}
We first identify the possible atoms. For any \(\tau_0\in\mathbb R\), the Stieltjes representation gives
\[
  \eta\,\im m_j(\tau_0+\mathrm i\eta)
  =
  \int_{\mathbb R}
  \frac{\eta^2}{(x-\tau_0)^2+\eta^2}\,
  \widetilde v_j(\dd x).
\]
The integrand is bounded by \(1\) and converges pointwise to
\(\mathbf 1_{\{x=\tau_0\}}\). Hence, the dominated convergence theorem implies that
\begin{equation}\label{eq.delta-mass-weight}
  \lim_{\eta\downarrow0}
  \eta\,\im m_j(\tau_0+\mathrm i\eta)
  =
  \widetilde v_j(\{\tau_0\}).
\end{equation}
In particular, if \(\widetilde v_j(\{\tau_0\})>0\), then
$
  \im m_j(\tau_0+\mathrm i\eta)
  \ge \eta^{-1}\widetilde v_j(\{\tau_0\})
  \to \infty ,
$
and therefore every atom of \(\widetilde v_j\) belongs to \(\mathcal A_j\).
Since \(\mathcal A_j\) is finite, we may define
\[
  \{\tau_{j,k}\}_{k=1}^{N_j}
  :=
  \left\{
    \tau\in\mathcal{A}_j:
    \widetilde v_j(\{\tau\})>0
  \right\},
\]
and set
$
  \Delta_{j,k}:=\widetilde v_j(\{\tau_{j,k}\})>0 .
$
The preceding identity \eqref{eq.delta-mass-weight} gives the first part of \eqref{def.vj-Delta}.

Define the atom-removed measure $\nu_j$ and its Stieltjes transform as follows:
\begin{equation}\label{def.nuj}
  \nu_j  :=  \tilde{v}_j  -  \sum_{k=1}^{N_j}\Delta_{j,k}\delta_{\tau_{j,k}},\qquad m_j^{\mathrm{ac}}(z)  := \int_{\R} \frac{\nu_j(\dd \tau)}{\tau-z}= m_j(z)  -  \sum_{k=1}^{N_j}  \frac{\Delta_{j,k}}{\tau_{j,k}-z}.
\end{equation}
Then, it follows from definition \eqref{def.vj-Delta} that 
\begin{equation}\label{def.hat-vj}
  v_j(\tau):=\frac{1}{\pi}\lim_{\eta\downarrow 0}\im m^{\mathrm{ac}}_j(\tau+\mathrm i\eta)\in[0,\infty],
\end{equation}
where the limit is understood in the extended sense. By the Puiseux expansion, this limit exists for every \(\tau\). 
Moreover, on every compact interval \(I\Subset\mathbb R\setminus\mathcal A_j\), the convergence in \eqref{def.hat-vj} is uniform for \(\tau\in I\), and \( v_j\) is finite on \(I\).
Therefore, interchanging the limit and the integration in the Stieltjes inversion formula yields, for any interval
\((a,b)\subset\mathbb R\setminus\mathcal A_j\),
\[
  \nu_j((a,b))
  =
  \lim_{\eta\downarrow0}
  \frac1\pi\int_a^b
  \im m^{\mathrm{ac}}_j(\tau+\mathrm i\eta)\,\dd\tau
  =
  \int_a^b v_j(\tau)\,\dd\tau ,
\]
which implies that $\nu_j|_{\R\backslash\mathcal{A}_j}(\dd\tau)= v_j(\tau)\dd\tau$.

Thus any singular part of \(\nu_j\) must be supported on the finite set \(\mathcal{A}_j\). 
However, all atoms of \(\tilde{v}_j\) have been removed in the definition of \(\nu_j\). 
Since \(\mathcal{A}_j\) is finite and all atoms have been removed, we have
$
\nu_j(\mathcal{A}_j)=0.
$
Therefore \(\nu_j\) has no singular part, and $\nu_j(\dd\tau)  = v_j(\tau)\,\dd\tau$.
Combining this identity with \eqref{def.nuj} yields \eqref{eq.measure-decomposition}.

Next, we prove the continuity of \(v_j\) on \(\mathfrak S_j\). 
Fix \(\tau_0\in\mathfrak S_j\). Since \(0<v_j(\tau_0)<\infty\) and 
\(\tau_0\) is not an atom, we have \(\tau_0\notin\mathcal{A}_j\). 
Since $\mathcal A_j$ is finite, one may work in a compact neighborhood of $\tau_0$ which avoids $\mathcal A_j$. 
The continuity of $\tau\mapsto \im m_j^{\mathrm{ac}}(\tau+\mathrm i\eta)$ on this neighborhood, together with the uniform convergence in \eqref{def.hat-vj}, implies that $v_j$ is continuous at $\tau_0$.
Since \(\tau_0\in\mathfrak S_j\) was arbitrary, \(v_j\) is continuous on \(\mathfrak S_j\).

\medskip

\textit{Step 4: Support decomposition.}
For every $\tau\notin\mathcal{A}$, $\tau$ cannot be an atom for any $\tilde{v}_j$, thus $\im m_j(\tau)=\pi v_j(\tau)$. 
Moreover, the previous results give that the boundary value $m_j(\tau)$ is well defined and finite for every $1\le j\le d$. 
Taking imaginary parts in the Dyson equation \eqref{eq:QVE} gives, for \(z=\tau+\mathrm i\eta\),
\begin{equation}\label{eq:QVE-imag-bound}
  \frac{\im m_j(z)}{|m_j(z)|^2}
  =
  \eta+\sum_{k=1}^d s_{jk}\im m_k(z),\qquad j=1,\dots,d.
\end{equation}
We claim that $m_j(\tau)\neq 0$ for every $1\le j\le d$. Otherwise, the left-hand side of the
\(j\)-th Dyson equation diverges as \(\eta\downarrow0\). This forces the existence
of some \(k\) with \(S_{jk}>0\) such that
$
   \lim_{\eta\downarrow 0}|m_k(\tau+\mathrm i\eta)|=\infty ,
$
contradicting \(\tau\notin\mathcal A\).

We next prove the all-or-nothing property of the imaginary parts. 
For every \(\tau\notin\mathcal A\), letting \(\eta\downarrow0\) in \eqref{eq:QVE-imag-bound} and using
\(\im m_j(\tau)=\pi v_j(\tau)\), we obtain
\begin{equation}\label{eq:boundary-imag-equation}
  \frac{v_j(\tau)}{|m_j(\tau)|^2}
  =
  \sum_{k=1}^d s_{jk}v_k(\tau),\qquad j=1,\dots,d.
\end{equation}
Since we have shown that, for all $\tau\notin \mathcal{A}$, \(|m_j(\tau)|\) is finite and non-zero, 
the left-hand side of the above equality is well defined.

Define
$
  Z(\tau):=\{j: v_j(\tau)=0\,\}.
$
If \(j\in Z(\tau)\), then the left-hand side of
\eqref{eq:boundary-imag-equation} vanishes. Hence
$
  \sum_{k=1}^d s_{jk}v_k(\tau)=0.
$
Since \(s_{jk}\ge0\) and \(v_k(\tau)\ge0\), this implies
\[
  s_{jk}>0
  \quad\Longrightarrow\quad
  v_k(\tau)=0.
\]
Thus \(Z(\tau)\) is closed under all outgoing edges of the directed graph
associated with \(S\). Since \(S\) is irreducible, the graph is strongly
connected, and no nonempty proper subset can be closed under outgoing edges.
Hence, we have
$
  Z(\tau)=\varnothing
$
or
$ 
  Z(\tau)=\{1,\dots,d\}.
$
Consequently, for every $\tau\notin\mathcal A$, either $v_j(\tau)>0$ for all $j=1,\dots,d$, or $v_j(\tau)=0$ for all $j=1,\dots,d$. 
Therefore, the sets \(\mathfrak S_j\) are independent of \(j\).

Since \(v_j\) is continuous on \(\mathfrak{S}_j=\mathfrak S\), it follows from the definition \eqref{def.Sj} that \(\mathfrak S\) is open.
Moreover, by the Puiseux expansion, the set
$
  \{x\in\mathbb R:0<v_j(x)<\infty\}
$
is a finite union of intervals.
Therefore the set \(\mathfrak S\) has only finitely many boundary points in
each such neighborhood. Since \(\widetilde v_j\) is supported in the compact
interval \([-\Sigma,\Sigma]\), a finite covering argument yields that
\(\partial\mathfrak S\) is finite. Hence \(\mathfrak S\) is a finite union of
open intervals.
\end{proof}

\bigskip
\begin{proof}[Proof of Theorem \ref{thm.bound}]
    Since the proof of Theorem~\ref{thm.bound} is largely parallel to the argument in
    \cite[Section~5]{AEK2019quadratic}, we only point out the minor modification
    needed here. In \cite{AEK2019quadratic}, \(S\) is assumed to be primitive, whereas
    we only use irreducibility. Thus, instead of invoking positivity of a fixed power
    of \(S\), we use
    \[
    \im m\gtrsim \sum_{t=0}^{d-1}S^t\im m .
    \]
    Indeed, irreducibility implies that for any \(i\ne j\) there is a directed path
    from \(i\) to \(j\) of length at most \(d-1\). Since positive entries of \(S\) are
    uniformly bounded from below, every entry of \(\sum_{t=0}^{d-1}S^t\) is bounded
    from below by a model parameter. Hence
    $
    (\im m)_i\gtrsim \langle \im m\rangle
    $ for $i=1,\dots,d$.
    This replaces the use of primitivity in \cite{AEK2019quadratic}; the remaining
    steps are unchanged.    
\end{proof}

\bigskip
\begin{proof}[Proof of Theorem \ref{thm.regularity}]
    The uniform boundedness of \(m\) on $\Cp(I)$ implies \(I\cap\mathcal A=\varnothing\). 
    Hence $\tilde v_i$ is purely absolutely continuous, namely $\tilde v_i(\dd\tau)=v_i(\tau)\dd\tau$, with $v_i(\tau)<\infty$ for all $\tau\in I$. 
    Moreover, by \eqref{def.vj-Delta} and \eqref{def.mj}, we have $v_i(\tau)=\pi^{-1}\im m_i(\tau)$ for every $\tau\in I$. 
    Thus the comparability \eqref{eq.v-comparable} follows directly from \eqref{eq.m-Im}. 
    It remains to prove the uniform H\"older continuity of $m(z)$ and $v(\tau)$.
    
    Since $m$ is uniformly bounded on $\Cp(I)$, repeating the argument leading to \eqref{eq.nonnegative-puiseux}, we obtain the following Puiseux expansion in a neighborhood of $\tau$
    \[
      m_i(z)=\sum_{k\ge0}c_k(z-\tau)^{k/n}.
    \]
    We next upgrade the one-point Puiseux expansion to a two-point Hölder estimate. By the Puiseux expansion, after shrinking the neighborhood of \(\tau\) if necessary, we may write
    \[
      m_i(z)=\sum_{k\ge 0} c_k (z-\tau)^{k/n}
      =F_\tau\bigl((z-\tau)^{1/n}\bigr),
    \]
where
$
  F_\tau(\zeta):=\sum_{k\ge 0} c_k\zeta^k
$
is analytic in a neighborhood of \(0\). Here the branch of
\((z-\tau)^{1/n}\) is chosen consistently in the local Puiseux
neighborhood.

Since \(F_\tau\) is analytic, it is Lipschitz on a sufficiently small
compact neighborhood of \(0\). Hence there exists a constant
\(L_\tau>0\) such that
$
  |F_\tau(\zeta)-F_\tau(\xi)|
  \le L_\tau |\zeta-\xi|
$
for all \(\zeta,\xi\) in this neighborhood. Therefore, for \(z,w\)
sufficiently close to \(\tau\),
\[
\begin{aligned}
  |m_i(z)-m_i(w)|
  =
  \left|
    F_\tau\bigl((z-\tau)^{1/n}\bigr)
    -
    F_\tau\bigl((w-\tau)^{1/n}\bigr)
  \right| 
  \le
  L_\tau
  \left|
    (z-\tau)^{1/n}-(w-\tau)^{1/n}
  \right|.
\end{aligned}
\]
Moreover, the map \(\zeta\mapsto \zeta^{1/n}\) is \(1/n\)-Hölder on
the chosen sector. Thus we have
\[
  \left|
    (z-\tau)^{1/n}-(w-\tau)^{1/n}
  \right|
  \le C |z-w|^{1/n},
\]
which implies
\[
  |m_i(z)-m_i(w)|
  \le C_\tau |z-w|^{1/n}
\]
for all \(z,w\) in a sufficiently small neighborhood of \(\tau\).

Thus the Puiseux expansion yields a genuine local two-point Hölder
estimate, not merely the one-point estimate with base point \(\tau\).
Taking a finite cover of the compact interval $I$ under consideration and
letting \(\alpha\) be the minimum of the corresponding exponents, we
obtain
\[
  |m_i(z)-m_i(w)|\le C |z-w|^\alpha
\]
locally uniformly near the real axis. In particular, since
\(v_i(\tau)=\pi^{-1}\im m_i(\tau)\), the same Hölder estimate holds for
the density \(v_i\) on the real axis.

It remains to show the Hölder continuity away from the real axis. Fix
\(\eta_0>0\) and set
\[
  \C_{\eta_0}:=\{z\in\mathbb C_+:\eta\ge \eta_0\}.
\]
By the Stieltjes representation \eqref{def.v},
we have
\[
  \partial_z m_i(z)=
  \int_{\mathbb R}\frac{\tilde v_i(\dd \tau)}{(\tau-z)^2},\qquad z\in \C_{\eta_0}.
\]
Hence it follows that
\[
  |\partial_z m_i(z)|
  \le
  \int_{\mathbb R}\frac{\tilde v_i(\dd \tau)}{|\tau-z|^2}
  \le
  \frac{\tilde v_i(\mathbb R)}{\eta_0^2}.
\]
Thus \(m_i\) is Lipschitz on \(\C_{\eta_0}\). Indeed, for
\(z,w\in \C_{\eta_0}\), the line segment joining \(z\) and \(w\) remains
inside \(\C_{\eta_0}\), and therefore
\[
  |m_i(z)-m_i(w)|
  \le
  \sup_{\zeta\in \C_{\eta_0}}|\partial_z m_i(\zeta)|\, |z-w|
  \le
  C_{\eta_0}|z-w|.
\]
Since \(\alpha\in(0,1]\), this implies the \(\alpha\)-Hölder estimate
for \(|z-w|\le 1\):
\[
  |m_i(z)-m_i(w)|
  \le
  C_{\eta_0}|z-w|
  \le
  C_{\eta_0}|z-w|^\alpha .
\]
If \(|z-w|>1\), then the uniform boundedness of \(m_i\) gives
\[
  |m_i(z)-m_i(w)|
  \le
  2\|m_i\|_\infty
  \le
  2\|m_i\|_\infty |z-w|^\alpha .
\]
Consequently, we have
\[
  |m_i(z)-m_i(w)|
  \le
  C |z-w|^\alpha,
  \qquad z,w\in \C_{\eta_0}.
\]
Thus \(m_i\) is uniformly \(\alpha\)-Hölder continuous away from the real
axis. Combining this estimate with the local Hölder estimate obtained
near the real axis from the Puiseux expansion yields the desired uniform
\(\alpha\)-Hölder continuity of \(m_i\).

\end{proof}

\section{Stability and Singular Behavior at Edges and Cusps}
This section studies the stability of the Dyson equation near spectral edges and cusps, together with the resulting square-root and cubic-root singular behavior of the density.
Although we cannot directly get information about the singular values and singular vectors of $F$ as it is asymmetric, the lemma below is still very useful in the subsequent analysis since it contains the information about the Perron-Frobenius eigenvalue and eigenvector.

\begin{lemma}\label{lem:F-PF}
  Suppose $S$ is nonnegative and irreducible, and the solution $m(z)$ is uniformly bounded on $\C_+(I)$, where $I$ is a closed interval. Let $z\in\Cp(I)$ and the matrix $F(z)$ be as in \eqref{def.F}.
  Then there exists an eigenvalue $\lambda_F(z)=\rho(F(z))$ and corresponding eigenvectors $r(z),l(z)>0$ such that
  \begin{equation}\label{eq:left-right-eigenvector}
    F(z)r(z)=\lambda_F(z)r(z),\qquad l(z)^\top F(z)=\lambda_F(z) l(z)^\top.
  \end{equation}
  Throughout the rest of the paper, we choose the unique normalization of these left and right eigenvectors such that
  $\langle l(z) r(z)\rangle=1$ and $|r(z)|_2=1$.
  
  Moreover, we have
  \begin{equation}\label{eq.F-eigenvalue}
    \lambda_F(z)=1-\frac{\eta\langle l(z)|m(z)|\rangle}{\langle l(z)|m(z)|^{-1}\im m(z)\rangle}\le 1,
  \end{equation}
  and
  \begin{equation}\label{eq.vector-comparable}
    l_i(z)\sim r_i(z)\sim 1,\qquad i=1,\dots,d.
  \end{equation}
\end{lemma}

\begin{proof}
  Since \(F(z)\) has the same zero pattern as \(S\), \(F(z)\) is nonnegative and irreducible. Then, Lemma \ref{lem.PF} directly implies \eqref{eq:left-right-eigenvector}.
  Then, taking imaginary parts in \eqref{eq:QVE} and multiplying both sides componentwise by $|m|$ yield
  \begin{equation}\label{eq.take-imaginary-Multiplying}
    \frac{\im m}{|m|}=\eta|m|+F\frac{\im m}{|m|}.
  \end{equation}
  Multiplying both sides by $l(z)^\top$  yields 
  \begin{equation*}
    \langle l(z)\frac{\im m}{|m|}\rangle=\eta\langle l(z)|m|\rangle+\lambda_F(z) \langle l(z)\frac{\im m}{|m|}\rangle.
  \end{equation*}
  Rearranging the above equality implies \eqref{eq.F-eigenvalue}.

  By \eqref{eq.m-bound} and the definition of $F(z)=(f_{ij})_{i,j=1}^d$ in \eqref{def.F},  we have
  \begin{equation}\label{eq.Fij}
    f_{ij}(z)=|m_i(z)| s_{ij}|m_j(z)|\sim \frac{s_{ij}}{(1+|z|)^2},
  \qquad 1\le i,j\le d.
  \end{equation}
  Hence, for every fixed $k\ge 1$, one has
  $
  F(z)^k \sim (1+|z|)^{-2k} S^k
  $
  entrywise.
  Since $S$ is nonnegative and irreducible, 
  for any two distinct indices \(i,j\) there exists a directed path from \(i\) to \(j\). Choosing a shortest such path, it has no repeated vertices; otherwise one could remove a cycle and obtain a shorter path. Therefore its length is at most \(d-1\). 
  Hence
  $
  \sum_{k=0}^{d-1} S^k
  $
  has strictly positive entries. Therefore, define
  \[
  T(z)=(t_{ij}(z))_{i,j=1}^d:=\sum_{k=0}^{d-1} (1+|z|)^{2k} F(z)^k ,
  \]
  then we have
  $
  t_{ij}(z)\sim 1.
  $

  Since $r(z)$ and $l(z)$ are right and left eigenvectors of $F(z)$, they are also right and left eigenvectors of $T(z)$. Thus, there exists $\rho(z)>0$ such that
  \[
  T(z)r(z)=\rho(z)r(z),\qquad l(z)^\top T(z)=\rho(z)l(z)^\top.
  \]
  Then, it follows from $t_{ij}(z)\sim 1$ that $r_i(z)\sim r_j(z)$ and $l_i(z)\sim l_j(z)$.
  Finally, using the normalization $\langle l(z)r(z)\rangle=1$ together with $|r(z)|_2=1$, we obtain \eqref{eq.vector-comparable}.
\end{proof}

\bigskip

We shall repeatedly need a quantitative inverse bound for Markov-type
matrices away from their invariant direction. The following lemma provides such
a bound for \(I-\lambda Q\), uniformly for \(0<\lambda\le 1\), after projecting
out the Perron direction \(\mathbf 1\). Its main role is to control the inverse
of the stability operator on the stable subspace.

This lemma will be used directly in Proposition~\ref{pro.I-F-second}. Indeed,
after normalizing \(F\) into a Markov matrix \(Q_F\), the estimate below implies
a uniform bound on \((I-F)^{-1}\mathcal P\), where \(\mathcal P\) removes the
unstable Perron direction. This separation between the unstable
one-dimensional direction and the uniformly stable complement is the basic input
for the subsequent perturbative expansion near edges and cusps.

\begin{lemma}\label{lem:Q-second-singular-gap-lambda}
Let \(Q=(q_{ij})_{i,j=1}^d\in \R^{d\times d}\) be a nonnegative irreducible matrix such that $Q\mathbf 1=\mathbf 1$. Let $\pi$ denote its unique invariant probability vector, and define
the projection matrix $P:=I-\mathbf{1}\pi^\top$.
Then there exist constants \(C>0\) and $N\in\N_+$ depending only on the support of $Q$, such that for every \(\lambda\in(0,1]\),
\begin{equation}
      \|(I-\lambda Q)^{-1}P\|_2\le C\kappa_Q^{-N},
\end{equation}
where $\kappa_Q:=\min\{q_{ij}:q_{ij}>0\}$.
\end{lemma}

\begin{remark}
  For \(\lambda=1\), we interpret \((I-Q)^{-1}P\) as the inverse of \(I-Q\) restricted to \(\Ran P=\{x:\pi^\top x=0\}\).
\end{remark}

\begin{proof}
  We first justify that the vector \(\pi\) used in the definition of \(P\) is well defined.
  Since \(Q\ge 0\) is irreducible and \(Q\mathbf 1=\mathbf 1\), Lemma \ref{lem.PF} implies that \(1\) is the spectral radius of \(Q\), the eigenspace of \(Q\) corresponding to the eigenvalue \(1\) is one-dimensional, and there exists a unique invariant probability vector \(\pi\in \R^d\) such that
  \[
  \pi_i>0,\qquad \sum_{i=1}^d \pi_i=1,\qquad \pi^\top Q=\pi^\top .
  \]

\smallskip
\textit{Step 1: }
Since \(Q\mathbf 1=\mathbf 1\) and \(\pi^\top Q=\pi^\top\), the projection \(P=I-\mathbf 1\pi^\top\) commutes with \(Q\). Hence \(\operatorname{Ran}P\) is \(Q\)-invariant.
Therefore, it suffices to show that there exist constants $c>0$ and $N\in\N_+$ depending only on the support of $Q$ such that
\begin{equation}\label{eq.min-singular}
  \inf_{x\in \Ran P}\frac{|(I-\lambda Q)x|_2}{|x|_2}\ge c\kappa_Q^{N} .
\end{equation}
We define the weighted norm $ |x|_\pi^2:=\sum_{i=1}^d \pi_i |x_i|^2$ with the corresponding weighted inner product $\langle x,y\rangle_\pi:=\sum_{i=1}^{d}\pi_i \overline {x_i}y_i$.
By definition, the norms \(|\cdot|_\pi\) and \(|\cdot|_2\) have the following relation
\begin{equation}\label{eq:pi-euclid-equiv-lambda}
(\min_i\pi_i)^{1/2}|x|_2\le |x|_\pi \le (\max_i\pi_i)^{1/2}|x|_2,\qquad x\in \C^d.
\end{equation}

For any edge \((i,j)\), i.e. \(q_{ij}>0\), the invariance relation
\(\pi^\top Q=\pi^\top\) gives
$
  \pi_j\ge \pi_i q_{ij}.
$
Since \(Q\) is irreducible, for any \(i,j\) there exists a directed path
\[
  i=i_0\to i_1\to\cdots\to i_\ell=j .
\]
Choosing a shortest such path, it has no repeated vertices; otherwise one could
remove a cycle and obtain a shorter path. Hence the path contains at most \(d\)
vertices, and therefore \(\ell\le d-1\). Iterating the
above estimate along this path yields
\[
  \pi_j
  \ge \pi_i \prod_{r=0}^{\ell-1}q_{i_r i_{r+1}}
  \ge \pi_i \kappa_Q^\ell
  \ge \pi_i \kappa_Q^{d-1},
\]
which implies 
$
  (\max_i\pi_i)/(\min_i\pi_i)\le \kappa_Q^{-(d-1)} .
$
Thus, passing from \(|\cdot|_\pi\) to \(|\cdot|_2\) loses at most a factor \(\kappa_Q^{-(d-1)/2}\), which can be absorbed by increasing the exponent \(N\).

We shall prove that there exist constants \(c>0\) and $N\in\N_+$ depending only on the support of \(Q\), such that
\begin{equation}\label{eq:pi-gap-on-X-lambda}
\inf_{x\in\Ran P}\frac{|(I-\lambda Q)x|_\pi }{|x|_\pi}\ge c\kappa_Q^N.
\end{equation}
Once \eqref{eq:pi-gap-on-X-lambda} is proved, the Euclidean statement \eqref{eq.min-singular} follows immediately from \eqref{eq:pi-euclid-equiv-lambda}.  
Thus it remains to prove \eqref{eq:pi-gap-on-X-lambda}.

\medskip
\noindent\textit{Step 2: Period subspace.}
Let \(d_*\) denote the period of \(Q\). By Lemma \ref{lem:cyclic-decomp}, there exists a partition
$
\{1,\dots,d\}=\mathcal{C}_1\sqcup\cdots\sqcup \mathcal{C}_{d_*}
$
such that
\[
q_{ij}>0, i\in \mathcal{C}_k
\quad\Longrightarrow\quad
j\in \mathcal{C}_{k+1},
\qquad k=1,\dots,d_*.
\]

We next show that the invariant measure \(\pi\) assigns the same total mass to each cyclic class. More precisely,
\begin{equation}\label{eq:equal-mass-lambda}
\pi(\mathcal{C}_1)=\cdots=\pi(\mathcal{C}_{d_*})=\frac1{d_*}.
\end{equation}
Indeed, using \(\pi^\top Q=\pi^\top\) and the cyclic support property of \(Q\), we compute
\[
\pi(\mathcal{C}_{k+1})
=
\sum_{i\in \mathcal{C}_{k+1}}\pi_i
=
\sum_{i\in \mathcal{C}_{k+1}}\sum_{j=1}^d \pi_j q_{ji}
=
\sum_{j\in \mathcal{C}_{k}}\pi_j \sum_{i\in \mathcal{C}_{k+1}}q_{ji}
=
\sum_{j\in \mathcal{C}_{k}}\pi_j
=
\pi(\mathcal{C}_{k}).
\]
Since the classes \(\mathcal{C}_1,\dots,\mathcal{C}_{d_*}\) form a partition of \(\{1,\dots,d\}\), \eqref{eq:equal-mass-lambda} follows.

\medskip
The first component of the argument isolates the directions coming purely from the periodic structure. Define
\begin{equation}\label{def.G}
  \mathcal E:=\Bigl\{x\in \C^d:\ x|_{\mathcal{C}_{k}}\equiv c_{k} \text{ for some }c_{k}\in \C,\ k=1,\dots,d_*\Bigr\},
\end{equation}
where \(\mathcal E\) consists of vectors that are constant on each cyclic class. 
We assume $d_*\ge 2$ in the remaining part of this step, and the case \(d_*=1\) will be covered in Step~3.
Let $\omega:=e^{2\pi i/d_*}.$
For \(j=1,\dots,d_*\), define \(e^{(j)}\in \C^d\) by
$
e^{(j)}_i:=\omega^{-(j-1)(k-1)}$
where $i\in \mathcal{C}_{k}$.
Then \(e^{(j)}\in \mathcal E\). Moreover, if \(i\in \mathcal{C}_{k}\), then by the cyclic support property all nonzero entries in the \(i\)-th row of \(Q\) come from columns indexed by \(\mathcal{C}_{k+1}\). It follows from \(Q\mathbf 1=\mathbf 1\) that
\[
\forall i\in\mathcal{C}_k,\qquad
(Qe^{(j)})_i
=
\sum_{n\in \mathcal{C}_{k+1}} q_{in} e^{(j)}_n
=
\omega^{-(j-1)k}\sum_{n\in \mathcal{C}_{k+1}}q_{in}
=
\omega^{-(j-1)k}
=
\omega^{-(j-1)} e^{(j)}_i.
\]
Therefore, we have 
$
Qe^{(j)}=\omega^{-(j-1)} e^{(j)}$ for  $j=1,\dots,d_*
$.

Next we check that they are orthogonal in $\langle\cdot,\cdot\rangle_\pi$. Using \eqref{eq:equal-mass-lambda}, it follows that
\[
\langle e^{(i)},e^{(j)}\rangle_\pi
=
\sum_{k=1}^{d_*} \pi(\mathcal{C}_{k})\overline{\omega^{-(i-1)(k-1)}}\omega^{-(j-1)(k-1)}
=
\frac1{d_*}\sum_{k=1}^{d_*} \omega^{(i-j)(k-1)}=\delta_{ij},
\]
which implies that \(\{e^{(j)}\}_{j=1}^{d_*}\) is an orthogonal basis of \(\mathcal E\). 

\smallskip
Define $\mathcal{E}_0:=\spn\{e^{(2)},\dots,e^{(d_*)}\}\subset\mathcal{E}$. For any $u\in\mathcal{E}_0$, since $\{e^{(j)}\}_{j=2}^{d_*}$ is an orthogonal basis of $\mathcal{E}_0$, 
we can write
$
u=\sum_{j=2}^{d_*} u_{j} e^{(j)}.
$
Using $Qe^{(j)}=\omega^{-(j-1)} e^{(j)}$ gives
\[
(I-\lambda Q)u=\sum_{j=2}^{d_*}(1-\lambda\omega^{-(j-1)})u_{j} e^{(j)}.
\]
Because the vectors \(\{e^{(j)}\}_{j=2}^{d_*}\) are orthogonal with respect to \(\langle\cdot,\cdot\rangle_\pi\), \(|(I-\lambda Q)u|^2_\pi\) is simply the sum of the squared moduli of the coefficients:
\begin{equation}\label{eq:cyc-gap-lambda}
    \qquad |(I-\lambda Q)u|_\pi^2=\sum_{j=2}^{d_*}|1-\lambda\omega^{-(j-1)}|^2 |u_j|^2 |e^{(j)}|_\pi^2\ge c^2_{\rm cyc} |u|^2_\pi,\qquad u\in\mathcal{E}_0.
\end{equation}
where
\begin{equation*}
    c_{\rm cyc}:=\min_{1\le k\le d_*-1}\ \inf_{\lambda\in(0,1]}|1-\lambda\omega^k|.  
\end{equation*}

\medskip
\noindent\textit{Step 3: Mixing subspace.}
We now pass to the complementary directions, where the periodic oscillations
have been removed. 
Define
\[
  \mathcal F
  :=
  \Bigl\{
  x\in\C^d:\ \sum_{i\in\mathcal C_k}\pi_i x_i=0,
  \ k=1,\dots,d_*
  \Bigr\}.
\]
Recall the definition of $\mathcal{E}$ in \eqref{def.G}.
Then, it follows from the definition that $\mathcal{F}=\mathcal E^{\perp_\pi}$ where \(\perp_\pi\) denotes the orthogonal complement with respect to $\langle\cdot,\cdot\rangle_\pi$.
Hence we have
$
  \C^d=\mathcal E\oplus_{\perp_\pi}\mathcal F .
$
Moreover, since \(e^{(1)}=\mathbf 1\) and
$
  \mathcal E=\operatorname{span}\{\mathbf 1\}
  \oplus_{\perp_\pi}\mathcal E_0,
$
we also have
$
  \operatorname{Ran}P
  =
  \{\pi^\top x=0\}
  =
  \mathcal E_0\oplus_{\perp_\pi}\mathcal F .
$
Thus the estimate on \(\operatorname{Ran}P\) splits into the parts
\(\mathcal E_0\) and \(\mathcal F\). The periodic
part \(\mathcal E_0\) has already been treated in Step 2. It remains
to handle the complementary space \(\mathcal F\).

We also claim that \(\mathcal F\) is \(Q\)-invariant. This is important, because it allows us to study \(I-\lambda Q\) separately on \(\mathcal E_0\) and on \(\mathcal F\). Let \(x\in \mathcal F\). Then for each \(1\le k\le d_*\),
\[
\sum_{i\in \mathcal{C}_{k}}\pi_i (Qx)_i
=
\sum_{i\in \mathcal{C}_{k}}\pi_i \sum_{j\in \mathcal{C}_{k+1}}q_{ij}x_j
=
\sum_{j\in \mathcal{C}_{k+1}}x_j \sum_{i\in \mathcal{C}_{k}}\pi_i q_{ij}
=
\sum_{j\in \mathcal{C}_{k+1}}\pi_j x_j
=
0,
\]
where in the penultimate step we used \(\pi^\top Q=\pi^\top\) together with the cyclic support property. Hence \(Qx\in \mathcal F\), proving the claim.

For each \(k=1,\dots,d_*\), let \(\pi_k\) be the normalized restriction of \(\pi\) to \(\mathcal{C}_{k}\):
$\pi_k(i):=d_*\pi_i$ for all $i\in \mathcal{C}_{k}$.
Let $Q_k:=Q^{d_*}|_{\mathcal{C}_{k}}$. Lemma \ref{lem:cyclic-decomp} implies that \(Q_k\) is a primitive matrix satisfying $Q_k\mathbf{1}=\mathbf{1}$, with invariant probability vector \(\pi_k\). 

We next prove a quantitative contraction for a suitable power of \(Q_k\) on the zero-mean subspace
$
L_0^2(\pi_k):=\{f\in L^2(\pi_k): \langle f,\mathbf 1\rangle_{\pi_k}=0\}.
$
Since \(Q_k\) is primitive for $1\le k\le d_*$, there exists an integer \( N\ge 1\) depending only on the support of \(Q\) such that
$
Q_k^{N}\ge \kappa_Q^{Nd_*}
$
entrywise for all \(k\). 

Define
\[
\vartheta_k:=\min_{i,j\in \mathcal C_k}\frac{(Q_k^{N})_{ij}}{\pi_k(j)},\qquad \Pi_k:=\mathbf{1}\pi_k^\top.
\]
By \(Q^N_k\ge\kappa_Q^{Nd_*}\) and \(0<\pi_k(j)\le 1\), we have $\vartheta_k\ge \kappa_Q^{Nd_*}$. 
If $\vartheta_k=1$, then $Q_k^N=\Pi_k$, which directly implies the desired contraction. Thus, we always consider $\vartheta_k<1$.
Decompose $Q_k^{N}$ as follows
\[
Q_k^{N}=\vartheta_k \Pi_k+(1-\vartheta_k)J_k,
\]
where
\[
(J_k)_{ij}:=\frac{(Q_k^{N})_{ij}-\vartheta_k\pi_k(j)}{1-\vartheta_k}\ge 0.
\]
By construction, we always have $J_k\mathbf{1}=\mathbf{1}$ and  $\pi_k^\top J_k=\pi_k^\top$.
In particular, Jensen's inequality implies that \(J_k\) is a contraction on \(L^2(\pi_k)\):
$
|J_k f|_{\pi_k}\le |f|_{\pi_k}
$.

Now let \(f\in L_0^2(\pi_k)\). Since \(\Pi_k f=0\), the above decomposition yields
$
Q_k^{N} f=(1-\vartheta_k)J_k f,
$
and therefore
$
|Q_k^N f|_{\pi_k}
=
|Q_k^{N} f|_{\pi_k}
\le
(1-\vartheta_k)|f|_{\pi_k}.
$
Thus \(Q_k^N\) is a strict contraction on \(L_0^2(\pi_k)\), with a quantitative gap
$
\vartheta_k\ge\kappa_Q^{Nd_*}.
$
Define
\[
\vartheta_*:=\min_{1\le k\le d_*}\vartheta_k\ge\kappa_Q^{Nd_*},
\qquad
\theta_*:=1-\vartheta_*\le 1-\kappa_Q^{Nd_*}.
\]
Then for every \(1\le k\le d_*\) and every \(f\in L_0^2(\pi_k)\), it follows that
$
|Q_k^N f|_{\pi_k}\le \theta_* |f|_{\pi_k}.
$

Now take \(f\in \mathcal F\), and decompose it according to the cyclic classes:
\[
f=f_1+\cdots+f_{d_*},
\qquad \supp f_k\subset \mathcal C_k.
\]
Since \(f\in\mathcal F\), each \(f_k\) has zero \(\pi_k\)-mean, hence \(f_k\in L_0^2(\pi_k)\). Therefore we have
\[
|Q^{Nd_*}f|_\pi^2
=
\frac1{d_*}\sum_{k=1}^{d_*}|Q_k^N f_k|_{\pi_k}^2
\le
\frac{\theta_*^2}{d_*}\sum_{k=1}^{d_*}|f_k|_{\pi_k}^2
=
\theta_*^2 |f|_\pi^2,
\]
which implies that 
\[
|(I-\lambda^{Nd_*}Q^{Nd_*})f|_\pi\ge |f|_\pi-|Q^{Nd_*}f|_\pi\ge (1-\theta_*)|f|_\pi
=\vartheta_*|f|_\pi.
\]
Then the factorization $I-A^n=(I-A)\sum_{k=0}^{n-1}A^k$ and $\|\lambda Q\|_\pi\le 1$ yield
\[
 \vartheta_*|f|_\pi\le|(I- (\lambda Q)^{Nd_*})f|_\pi
 \le
 |(I- \lambda Q)f|_\pi
 \left\|\sum_{j=0}^{Nd_*-1}(\lambda Q)^j\right\|_\pi
 \le Nd_*|(I- \lambda Q)f|_\pi,
\]
where $\|\cdot\|_\pi$ denotes matrix norm from $L^2(\pi)$ to $L^2(\pi)$, induced by $|\cdot|_\pi$.
Therefore, we have
\begin{equation}\label{eq:mix-gap-lambda}
\forall f\in \mathcal F,\qquad |(I-\lambda Q)f|_\pi \ge \frac{\vartheta_*}{Nd_*}|f|_\pi\ge \frac{\kappa_Q^{Nd_*}}{Nd_*}|f|_\pi .
\end{equation}

\medskip
\noindent\textit{Step 4: Combining the periodic and mixing estimates.}
We now combine the estimates on the periodic and mixing subspaces. First note
that the cyclic gaps are strictly positive. Indeed, since for every $1\le k\le d_*-1$ none of the rays \(\{\lambda\omega^k:0<\lambda\le 1\}\) contains the point \(1\), compactness gives
$
  c_{\rm cyc}>0.
$
Thus the constants coming from the cyclic part are positive. Enlarging \(N\)
if necessary, we may assume
$
  \kappa_Q^{Nd_*}/(N d_*)\le c_{\rm cyc}.
$
Combining the mixing estimate \eqref{eq:mix-gap-lambda}, the cyclic estimates
\eqref{eq:cyc-gap-lambda}, and the space
decomposition established at the beginning of Step~3, we complete the proof.

\end{proof}

\begin{proposition}\label{pro.I-F-second}
Suppose $S$ is nonnegative and irreducible, and the solution $m(z)$ is uniformly bounded on $\C_+(I)$. Define the projection matrix $\mathcal{P}(z):=I-r(z)l(z)^\top$ where $r(z)$ and $l(z)$ are defined in \eqref{eq:left-right-eigenvector}. Then, one has
\begin{equation}
  \big\|(I-F(z))^{-1}\mathcal{P}(z)\big\|_2\lesssim 1,\qquad z\in\overline{\Cp}(I).
\end{equation}
where \(F(z)\) is defined in \eqref{def.F}. 
\end{proposition}
\begin{remark}
  If $I-F$ is not invertible, then we again interpret \((I-F)^{-1}\mathcal{P}\) as the inverse of \(I-F\) restricted to \(\Ran \mathcal{P}\).
\end{remark}

\begin{proof}
We first normalize the matrix $F$
\begin{equation}\label{eq:F-normalization}
  Q_F(z):= \lambda_F^{-1}D(r)^{-1}FD(r).
\end{equation}
It follows from \(Fr=\lambda_F r\) that 
$
  Q_F\mathbf 1=\lambda_F^{-1}D(r)^{-1}Fr=\mathbf 1.
$
Moreover, with
$
  (\pi_F)_i:=l_i r_i,
$
the normalization \(\langle lr\rangle=1\) makes \(\pi_F\) a probability vector, and
\[
  \pi_F^\top Q_F=l^\top D(r)\lambda_F^{-1}D(r)^{-1}FD(r)=\lambda_F^{-1}l^\top FD(r)=l^\top D(r)=\pi_F^\top.
\]
Hence \(\pi_F\) is the invariant probability vector of \(Q_F\).
Since \(S\) is nonnegative and irreducible and  diagonal matrices $D(r),D(|m|)$ are positive, \(Q_F(z)\) is also nonnegative and irreducible and has the same support as that of $S$.
Moreover, it follows from \eqref{eq.vector-comparable} and \eqref{eq.Fij} that all positive entries of \(Q_F(z)\) are uniformly comparable. 
By \(Q_F\mathbf 1=\mathbf 1\), we obtain
$(Q_F)_{ij}\sim 1$ whenever $s_{ij}>0$, which implies that $\kappa_{Q_F}\sim 1$.

It follows from the definition that
\begin{align*}
  (I-F)^{-1}\mathcal{P}&=D(r)(I-\lambda_F Q_F)^{-1}(I-\mathbf 1\pi_F^\top)D(r)^{-1}.
\end{align*}
Then, Lemma \ref{lem:Q-second-singular-gap-lambda} implies $\|(I-\lambda_F Q_F)^{-1}(I-\mathbf 1\pi_F^\top)\|_2\lesssim 1$ for every $z\in\overline{\C_+}(I)$.
By \eqref{eq.vector-comparable}, we have 
$
\|D(r)\|_2\,\|D(r)^{-1}\|_2\sim 1
$
for all $z\in\overline{\C_+}(I)$,
which concludes the proof.
\end{proof}
\bigskip
The next lemma gives the perturbative description of the stability matrix near
a zero-density boundary point. At such a point the matrix \(I-F\) has one
unstable direction, generated by the Perron vector, while its restriction
to the complementary subspace remains uniformly stable. 
The following lemma makes this picture quantitative. It gives the expansion of \(R\),
the small eigenvalue \(\beta\), and the corresponding right and left
eigenvectors $\tilde{r},\tilde{l}$. Moreover, the inverse estimates show that
the only possible instability of \(R^{-1}\) comes from the one-dimensional
eigendirection associated with \(\beta\). Once this direction is projected out,
\(R^{-1}\) remains uniformly bounded, while the full inverse has size
\(|\beta|^{-1}\).

\begin{lemma}[Expansion around edge and cusp]
\label{lem:B-bad-direction-nonsym-local}
Suppose that $S\in \R^{d\times d}$ is nonnegative and irreducible. Let \(\tau\in \partial\mathfrak S\) be a regular edge or cusp. There exists $\epsilon>0$ depending only on model parameters such that the following statements hold for every \(z\in N_\epsilon(\tau)\). Expressions involving inverses or ratios at boundary
points are understood according to the convention specified in the remark below.

The complex matrix $R(z)\in\C^{d\times d}$
defined in \eqref{def.B} has a unique simple eigenvalue $\beta(z)$ of smallest modulus, with spectral gap
\begin{equation}
  |\beta'|-|\beta|\gtrsim 1,\qquad\forall\beta'\in\Spec(R)\backslash\{\beta\}.
\end{equation}
Let \(\tilde r(z)\) and \(\tilde l(z)\) be corresponding right and left eigenvectors,
\[
R\tilde r=\beta\tilde r,\qquad  R^*\tilde l=\bar\beta\tilde l,
\]
normalized by $\langle l,\tilde{r}\rangle=\langle r,\tilde{l}\rangle=1$,
where $l(z)$ and $r(z)$ are the real valued eigenvectors defined in \eqref{eq:left-right-eigenvector}.
Then the following expansions hold:
\begin{align}
R
&=
I-F-2\mathrm i\delta D\bigl(pr\bigr)
-2\delta^2D\bigl(r^2\bigr)
+O\bigl(\delta^3+\eta\bigr),
\label{eq:B-expansion-local}
\\
\beta
&=
\mu\frac{\eta}{\delta}
-2\mathrm i\sigma\delta
+2\bigl(\psi-\sigma^2\bigr)\delta^2
+O\bigl(\delta^3+\eta\bigr).
\label{eq:beta-expansion-local}\\
\tilde r
&=
r+2\mathrm i\,\delta(I-F)^{-1}\mathcal{P}\bigl[pr^2\bigr]
+O\bigl(\delta^2+\eta\bigr),
\label{eq:rtilde-expansion-local}
\\
\tilde l
&=
l-2\mathrm i\,\delta\,(I-F^\top)^{-1}\mathcal{P}^\top\bigl[plr\bigr]
+O\bigl(\delta^2+\eta\bigr).
\label{eq:ltilde-expansion-local}
\end{align}
Here all the real valued auxiliary functions are defined 
\begin{equation}
  \delta(z):=\Bigl\langle l\frac{\im m}{|m|}\Bigr\rangle,\qquad
\mu(z):=\langle l|m|\rangle,\qquad  \sigma(z):=\langle plr^2\rangle,\qquad p(z):=\sign\re m
\end{equation}
and 
\begin{equation}
 \psi(z):=\Bigl\langle plr,\frac{I+F}{I-F}\mathcal{P}[pr^2]\Bigr\rangle,\qquad \mathcal{P}(z):=I-rl^\top.
\end{equation}

The function $p(z)$ remains constant vector, with no vanishing component. The perturbed eigenvectors remain uniformly nondegenerate: 
\begin{equation}\label{eq.tilde-rl-length}
|\tilde{r}_i|\sim|\tilde{l}_i|\sim |\langle\tilde{l},\tilde{r}\rangle|\sim 1,\qquad i=1,\dots,d.
\end{equation}
Consequently, defining $\widetilde{\Pi}:=\langle \tilde{l},\tilde{r}\rangle^{-1}\tilde{r}\tilde{l}^*$ and $\widetilde{\mathcal{P}}:=I-\widetilde{\Pi}$, one has
\begin{equation}\label{eq.Rinverse-bound}
  \big\|R^{-1}\widetilde{\mathcal{P}}\big\|_2+\big\|\big(R^{-1}\widetilde{\mathcal{P}}\big)^*\big\|_2\lesssim 1\quad\text{and}\quad \|R^{-1}\|_2\sim|\beta|^{-1}.
\end{equation}
\end{lemma}

\begin{remark}[Boundary conventions]
In the rest of the paper, whenever $\delta=0$, the ratio
$\eta/\delta$ is understood via its continuous extension
$(1-\lambda_F)/\mu$. If $R$ is singular,
$R^{-1}\widetilde{\mathcal P}$ denotes the inverse of $R$
restricted to $\operatorname{Ran}\widetilde{\mathcal P}$,
composed with $\widetilde{\mathcal P}$.
The full inverse estimate
$\|R^{-1}\|_2\sim|\beta|^{-1}$
is asserted only when $\beta\neq0$.
\end{remark}

\begin{proof}
Since $\tau$ is a regular edge or cusp, it follows that $\im m(\tau)=0$. Theorem \ref{thm.decompose of measure} implies that $m(\tau)$ is well defined and finite, and $m(z)$ is continuous in $N_\epsilon(\tau)$. This gives us that $m(z)$ is uniformly bounded on $N_\epsilon(\tau)$.

Theorem \ref{thm.bound} and Lemma \ref{lem:F-PF} imply
\begin{equation}\label{eq.rlsim1}
  |r_i(z)|\sim|l_i(z)|\sim |m_i(z)|\sim 1,
\end{equation}
which yields $\delta(z)\sim \langle\im m(z)\rangle$.
Since \(\im m(\tau)=0\), the continuity of \(m\) allows us to shrink $\epsilon$ so that both $\langle \im m\rangle$ and \(\delta(z)\) are uniformly as small as needed for all \(z\in N_\epsilon(\tau)\). In the remainder of the proof, unless stated otherwise, we always assume that $z\in N_\epsilon(\tau)$.

It follows from  $|m_i|\sim 1$ that there exists a finite constant $c_0>0$ such that 
$\min_i|m_i(z)|\ge c_0$. Since 
\begin{equation*}
  (\re m_i(z))^2=|m_i(z)|^2-(\im m_i(z))^2\ge c_0^2-(\im m_i(z))^2
\end{equation*}
and $\im m_i(z)$ is small enough, $\re m_i$ cannot vanish and $p_i(z)$ remains nonzero constant in $N_\epsilon(\tau)$.

On the other hand, for any fixed \(\tau'\in \mathfrak S\), we have
\(\im m(\tau')>0\). Since all terms in \eqref{eq.F-eigenvalue} admit
continuous extensions from \(\mathbb C_+\) to \(\overline{\mathbb C_+}\), we may
let \(\eta\to0\) in \eqref{eq.F-eigenvalue}. Using again that
$
    |m(\tau')|\sim |l(\tau')|\sim 1,
$
we obtain
$
    \lambda_F(\tau')=1.
$
Choosing a sequence \(\tau'_k\in\mathfrak S\) with \(\tau'_k\to\tau\) and using continuity, this further yields
$
    \lambda_F(\tau)=1.
$

Moreover, \eqref{eq.F-eigenvalue}, together with
\(|l_i(z)|\sim |m_i(z)|\sim 1\), gives
$
    1-\lambda_F(z)\sim \eta/\delta(z).
$
Hence the smallness of \(\eta/\delta(z)\) follows from the continuity of
\(\lambda_F\) at \(\tau\). Indeed, since \(\lambda_F(\tau)=1\), we have
\(1-\lambda_F(z)\to0\) as \(z\to\tau\). The above comparison therefore implies
\[
    \lim_{\substack{z\to \tau\\ ~z\in \overline\Cp}}\frac{\eta}{\delta(z)}=0.
\]
Consequently, after shrinking \(\epsilon>0\) if necessary, both
\(1-\lambda_F(z)\) and \(\eta/\delta(z)\) are sufficiently small for all \(z\in N_\epsilon(\tau)\).

\medskip
\emph{Step 1: Expansion of \(R\).}
Define $
q:=\im m/|m|.
$
Rewriting \eqref{eq.take-imaginary-Multiplying} and applying \(\mathcal{P}\) to both sides yield
\[
(I-F)\mathcal{P}q=\eta\mathcal{P}|m|,
\]
Proposition \ref{pro.I-F-second} and $|m_i(z)|\sim 1$ imply
\[
\mathcal{P}q=\eta\,(I-F)^{-1}\mathcal{P}|m|=O(\eta).
\]
Plugging $\mathcal{P}=I-rl^\top$, \(l^\top q=\delta\) and \(l^\top r=1\) yields
\[
q=\delta\, r+O(\eta)\qquad\text{and}\qquad \frac{\re m}{|m|}=p+O(\delta^2).
\]
For each component, we have
$
m_i
=
|m_i|\Bigl(p_i\sqrt{1-q_i^2}+\ii q_i\Bigr).
$
Thus, 
\[
U_{ii}=\frac{|m_i|^2}{m_i^2}
=
\frac{\overline{m_i}}{m_i}
=
\frac{p_i\sqrt{1-q_i^2}-\ii q_i}{p_i\sqrt{1-q_i^2}+\ii q_i}=1-2\ii p_i q_i-2q_i^2+O(q_i^3),
\]
where we apply the Taylor expansion at \(q_i=0\) in the last equality. 
By \(q_i=\delta r_i+O(\eta)\), we have
\begin{equation}\label{eq.U-expansion}
  U
=
I-2\ii\delta D(pr)-2\delta^2D(r^2)+O(\delta^3+\eta).
\end{equation}
The expansion \eqref{eq:B-expansion-local} follows directly from the above equality and definition $R=U-F$.

\medskip
\textit{Step 2: Spectral gap.}
We decompose $\mathbb C^d=\operatorname{span}\{r\}\oplus \operatorname{Ran}\mathcal P$. 
To make the corresponding block representation precise, we identify \(\mathbb C^d\) with \(\mathbb C\oplus \operatorname{Ran}\mathcal P\) through the linear isomorphism 
\[ \Phi:\mathbb C\oplus \operatorname{Ran}\mathcal P\to \mathbb C^d, \qquad \Phi(\gamma,w)=\gamma r+w . \] 
Thus, the following block matrix is the representation of the conjugated operator \(\Phi^{-1}R(z)\Phi\), i.e. the representation of \(R(z)\) in coordinates adapted to the above direct sum decomposition: 
\[ \Phi^{-1}R(z)\Phi = \begin{pmatrix} A(z) & M^\top(z) \\ N(z) & T(z) \end{pmatrix}. \] The blocks are explicitly given by \[ A(z):=l^\top R(z)r,\qquad M^\top(z):=l^\top R(z)\mathcal P,\qquad N(z):=\mathcal P R(z)r,\qquad T(z):=\mathcal P R(z)\mathcal P|_{\operatorname{Ran}\mathcal P}. \]
As \(\Phi^{-1}R(z)\Phi\) is similar to \(R(z)\), the two operators have the same spectrum. Therefore, studying the spectral properties of
\(\Phi^{-1}R(z)\Phi\) is equivalent to studying those of \(R(z)\).

Since $\mathcal{P}r=0$, $l^\top \mathcal{P}=0$, $\mathcal{P}(I-F)r=0$, and $l^\top(I-F)\mathcal{P}=0$, we get
\[
|A|=|(1-\lambda_F)+l^\top(U-I)r|\lesssim\delta+\frac{\eta}{\delta},
\]
\begin{equation}\label{eq.MN-bound}
  |M^\top|_2=|l^\top(U-I)\mathcal{P}|_2\lesssim\delta,\qquad |N|_2=|\mathcal{P}(U-I)r|_2\lesssim\delta.
\end{equation}
Furthermore, it follows that
\[
T=\mathcal{P}(I-F)\mathcal{P}|_{\Ran \mathcal{P}}+\mathcal{P}(U-I)\mathcal{P}|_{\Ran \mathcal{P}}.
\]
Proposition \ref{pro.I-F-second} and $\|U-I\|_2\lesssim \delta$ imply $\|T^{-1}\|_2\lesssim 1$.
Therefore, we can take $\zeta\in\C$ satisfying $|\zeta|\le c_*\sim 1$ with $c_*$ small enough such that
\begin{equation}\label{eq.D-bound}
  \|(T-\zeta)^{-1}\|_2\lesssim 1.
\end{equation}

Now suppose that $\zeta$ is an eigenvalue of $R$ with $|\zeta|\le c_*$, and write a corresponding eigenvector in the form
\[
x=\gamma r+y,\qquad y\in \Ran \mathcal{P}.
\]
Then $(R-\zeta)x=0$ is equivalent to
\[
(A-\zeta)\gamma+M^\top y=0,
\qquad
\gamma N+(T-\zeta) y=0.
\]
Since $T-\zeta$ is invertible, the second equation yields
$
y=-\gamma(T-\zeta)^{-1}N.
$
Substituting this into the first equation, we obtain the scalar equation
\[
f(\zeta):=A-\zeta-M^\top(T-\zeta)^{-1}N=0.
\]

Define the circle
$
\Gamma:=\{w\in\C: |w|=c_*\}.
$
For $\zeta\in \Gamma$, it follows from \eqref{eq.MN-bound} and \eqref{eq.D-bound} that
\[
|f(\zeta)-(A-\zeta)|=|M^\top(T-\zeta)^{-1}N|\le  C\delta^2,
\]
where $C>0$ is a finite constant depending only on model parameters. By shrinking $\epsilon>0$ if necessary, we may assume that $C\delta^2<c_*/2$ and $|A|<c_*/2$.
Therefore,  we have 
\[
|f(\zeta)-(A-\zeta)|\le  C\delta^2<\frac{c_*}{2}<|A-\zeta|,\qquad\zeta\in\Gamma.
\]
By Rouch\'e's theorem, $f(\zeta)$ and $A-\zeta$ have the same number of zeros inside $\Gamma$. Since $A-\zeta$ has exactly one zero there, $f$ also has exactly one zero there. 

Differentiating $f$, we obtain
$
f'(\zeta)=-1-M^\top (T-\zeta)^{-2}N.
$
Again using \eqref{eq.MN-bound} and \eqref{eq.D-bound}, we get
$
f'(\zeta)=-1+O(\delta^2).
$
Hence $f'(\beta)\neq 0$, so $\beta$ is a simple eigenvalue of $R$.

\medskip
\noindent\emph{Step 3: Expansion of eigenvalue and eigenvectors.}
Since \(\tilde r\) is a right eigenvector of $R$, normalized by $\langle l,\tilde{r}\rangle=1$
we can decompose it as
\begin{equation}\label{eq.rvector-decomposition}
\tilde r=r+w,
\qquad
w\in \Ran \mathcal{P}.  
\end{equation}
Inserting this decomposition into the eigenvalue equations and using $Fr=\lambda_Fr$, we have
\[
(1-\lambda_F)r+(I-F)w+(U-I)r+(U-I)w=\beta(r+w).
\]
Then, applying \(\mathcal{P}\) and using \(\mathcal{P}(I-F)w=(I-F)w\) and $\mathcal{P}r=0$, we obtain
\[
(I-F)w+ \mathcal{P}(U-I)r+\mathcal{P}(U-I)w=\beta \mathcal{P}w.
\]
Since \(I-F\) is invertible on $\Ran \mathcal{P}$, we apply its inverse $(I-F)^{-1}$
and conclude that
\begin{equation}\label{eq.w-expansion}
  w=-(I-F)^{-1}\mathcal{P}(U-I)r-(I-F)^{-1}\mathcal{P}(U-I)w+\beta (I-F)^{-1}\mathcal{P} w.
\end{equation}

Using the following facts:
\begin{equation}\label{eq.general-bound}
  \|(I-F)^{-1} \mathcal{P}\|_2\lesssim 1,
\quad
\|U-I\|_2\lesssim \delta,
\quad
|\beta|\lesssim 1-\lambda_F+\|U-I\|_2\lesssim\delta+\frac{\eta}{\delta},
\end{equation} 
we have the following a priori estimate
\begin{equation}\label{eq.w-order}
  |w|_2\lesssim\delta+(\delta+|\beta|)|w|_2\Longrightarrow |w|_2\lesssim\delta.
\end{equation}
Inserting the expansion \eqref{eq.U-expansion} into \eqref{eq.w-expansion}, it follows that
\begin{equation}
    w=2\mathrm i\delta (I-F)^{-1}\mathcal{P}[pr^2]+O(\delta^2+\eta),
\end{equation}
which together with \eqref{eq.rvector-decomposition} yields \eqref{eq:rtilde-expansion-local}. 

By the same argument, we can decompose $\tilde{l}$ as 
\begin{equation}\label{eq.tilde-l-decomposition}
  \tilde{l}=l+v,\qquad v\in \Ran \mathcal{P}^\top.
\end{equation}
Plugging it into $R^*\tilde{l}=\bar\beta\tilde{l}$ and using expansion \eqref{eq.U-expansion} imply
\begin{equation}\label{eq.v-expansion}
  v=-2\mathrm i\,\delta\,(I-F^\top)^{-1}\mathcal{P}^\top\bigl[plr\bigr]+O\bigl(\delta^2+\eta\bigr),
\end{equation}
which directly yields \eqref{eq:ltilde-expansion-local}.
Since the leading order terms of $\tilde{r}$ and $\tilde{l}$ are $r$ and $l$, \eqref{eq.tilde-rl-length} follows directly from \eqref{eq.vector-comparable} together with $l^\top r=1$.

\medskip
Applying \(l^\top\) to
$
R\tilde r=\beta\tilde r
$
and using \(l^\top(I-F)=(1-\lambda_F)l^\top\) together with \(l^\top w=0\), we obtain
\begin{equation}\label{eq:bad-scalar-right}
\beta
=
(1-\lambda_F)
+l^\top(U-I)r
+l^\top(U-I)w.
\end{equation}
Inserting the expansion \eqref{eq.U-expansion} and \eqref{eq.w-expansion}, we can obtain
\begin{equation}\label{eq.s-expansion}
  \beta=\mu\frac{\eta}{\delta}-2\mathrm i\sigma\delta+\left(-2\langle lr^3\rangle +4\langle plr,(I-F)^{-1}\mathcal{P}[pr^2]\rangle\right)\delta^2+O(\delta^3+\eta)
\end{equation}
Since
\[
\frac{I+F}{I-F}\mathcal{P}
=
\bigl(2(I-F)^{-1}-I\bigr)\mathcal{P}
=
2(I-F)^{-1}\mathcal{P}-\mathcal{P},
\]
we have
\[
\psi
=
2\langle plr,(I-F)^{-1}\mathcal{P} [pr^2]\rangle-\langle plr,\mathcal{P} [pr^2]\rangle.
\]
It follows from \(\mathcal{P}=I-rl^\top\) that
\[
\langle plr,\mathcal{P} [pr^2]\rangle
=
\langle plr,pr^2-r\langle plr^2\rangle\rangle
=
\langle l,r^3\rangle-\sigma^2.
\]
which implies 
\[
-\langle lr^3\rangle+2\langle plr,(I-F)^{-1}\mathcal{P} [pr^2]\rangle =\psi-\sigma^2 .
\]
Plugging the above equality into \eqref{eq.s-expansion} yields \eqref{eq:beta-expansion-local}.

\medskip
\textit{Step 4: Inverse bound.}
Since $\widetilde{\Pi}$ is the rank-one projection matrix on the direction $\tilde r$, by the eigenvector equations we have
$
R\widetilde \Pi=\widetilde \Pi R=\beta \widetilde{\Pi}$ and 
$
R\widetilde{\mathcal{P}}=\widetilde{\mathcal{P}}R .
$
In particular, \(\operatorname{Ran}\widetilde{\mathcal{P}}
=\ker \tilde l^*\) is invariant under \(R\).
It follows from \eqref{eq.tilde-rl-length} that \(\widetilde{\Pi}\) and \(\widetilde{\mathcal{P}}\) are uniformly bounded
$
\|\widetilde{\Pi}\|_2+\|\widetilde{\mathcal{P}}\|_2\lesssim 1 .
$

We now prove the bound on \(R^{-1}\widetilde{\mathcal{P}}\). 
For any \(x\in \Ran \widetilde{\mathcal{P}}\), the following decomposition holds:
\[
x=\theta r+y,
\qquad
y\in \operatorname{Ran}\mathcal{P} .
\]
Since $\Ran \widetilde{\mathcal{P}}=\ker \tilde l^*$, using the decomposition \eqref{eq.tilde-l-decomposition} yields
\[
0=\tilde l^*x=(l+v)^*(\theta r+y)
=\theta+v^*y,
\]
where we used \(l^\top r=1\), \(l^\top y=0\), and \(v^*r=0\). Hence, it follows that
$
|\theta|\le |v|_2|y|_2\lesssim \delta |y|_2,
$
and consequently we have
$
|x|_2\sim |y|_2.
$

Next, we apply \(\mathcal{P}\) to \(Rx\). Using $R=(I-F)+(U-I)$ together with 
$
\mathcal{P}(I-F)r=0$
and
$
\mathcal{P}(I-F)y=(I-F)y,
$
we obtain
\[
\mathcal{P} Rx
=
(I-F)y+\mathcal{P}(U-I)x .
\]
Therefore, rearranging the above identity and applying $(I-F)^{-1}$ yields
\[
y
=
(I-F)^{-1}\mathcal{P} Rx
-
(I-F)^{-1}\mathcal{P}(U-I)x .
\]
Then it follows from 
$
\|(I-F)^{-1}\mathcal{P}\|_2\lesssim 1$
and
$
\|U-I\|_2\lesssim \delta
$
that
$
|y|_2
\lesssim
|Rx|_2+\delta |x|_2
\lesssim
|Rx|_2+\delta |y|_2.
$
Since $\delta$ is small enough, using $|x|_2\sim|y|_2$ and absorbing the last term, we have
\begin{equation}\label{eq.x-Rx}
  |x|_2\lesssim |y|_2\lesssim |Rx|_2,\qquad
\forall x\in \Ran \widetilde{\mathcal{P}} .
\end{equation}

Now take arbitrary \(\mathbf{z}\in\mathbb C^d\) and set
$
\tilde{\mathbf{z}}=R^{-1}\widetilde{\mathcal{P}}\mathbf{z} .
$
Since \(R\) commutes with \(\widetilde{\mathcal{P}}\), we have \(\tilde{\mathbf{z}}\in \Ran\widetilde{\mathcal{P}}\). It follows from \eqref{eq.x-Rx} that
\[
|R^{-1}\widetilde{\mathcal{P}}\mathbf{z}|_2=|\tilde{\mathbf{z}}|_2
\lesssim
|R\tilde{\mathbf{z}}|_2
=
|\widetilde{\mathcal{P}}\mathbf{z}|_2
\lesssim
|\mathbf{z}|_2,
\]
which implies
$
\|R^{-1}\widetilde{\mathcal{P}}\|_2\lesssim 1.
$
Since the Euclidean operator norm is invariant under taking adjoints,
we conclude the first part of \eqref{eq.Rinverse-bound}.

By 
$
I=\widetilde{\Pi}+\widetilde{\mathcal{P}}
$
and \(R\widetilde{\Pi}=\beta \widetilde{\Pi}\), it can be deduced that
\[
R^{-1}
=
\beta^{-1}\widetilde{\Pi}
+
R^{-1}\widetilde{\mathcal{P}}.
\]
Using \(\|\widetilde{\Pi}\|_2\lesssim 1\) and $\|R^{-1}\widetilde{\mathcal{P}}\|_2\lesssim 1$, we obtain
$
\|R^{-1}\|_2
\lesssim
|\beta|^{-1}+1
\lesssim
|\beta|^{-1}.
$
The reverse bound follows from \(R^{-1}\tilde r=\beta^{-1}\tilde r\). This shows the second part of \eqref{eq.Rinverse-bound}. Hence we conclude the proof.

\end{proof}

The previous lemma shows that the possible instability of the stability
operator is governed by the small eigenvalue \(\beta\), whose expansion contains
the two coefficients \(\sigma\) and \(\psi\). In order to use this expansion
near a regular edge or cusp, we need to know that these coefficients do not
degenerate simultaneously. The next proposition provides precisely this
nondegeneracy. First, it shows that \(\psi\) is always nonnegative. More
importantly, under Assumption \textup{(Sym)} or \textup{(NB)}, it proves that
near a regular edge or cusp
$
  \psi+\sigma^2\sim 1 .
$
Thus, when \(\sigma\) becomes small, the coefficient \(\psi\) remains bounded
from below. This is the key input for cusp stability and cusp singularity.

\begin{proposition}\label{prop:psi-positive-and-nondegenerate}
Assume the setting of the previous lemma. Then
$
\psi(z)\ge0
$
for all $z\in\Cp$.
Moreover, let \(\tau\in \partial\mathfrak S\) be a regular edge or cusp.
If, in addition, $S$ satisfies \textup{(Sym)} or \textup{(NB)}, then there exists $\epsilon>0$ depending only on model parameters such that
\begin{equation*}
    \psi(z)+\sigma(z)^2\sim 1,\qquad z\in N_\epsilon(\tau).
\end{equation*}
\end{proposition}

\begin{proof}
  We first prove $\psi\ge 0$.
  In this step, we use the same method as in the proof of Lemma \ref{lem:Q-second-singular-gap-lambda}, that is, normalizing $F$ into a Markov transition matrix.
  Define
  \[
  \pi_i:=l_ir_i,\qquad
  h_i:=p_ir_i,\qquad \Pi:=\mathbf 1\,\pi^\top,
  \qquad
  Q_F:=\lambda_F^{-1}D(r)^{-1}FD(r).
  \]
  Since \(l^\top r=1\), the vector \(\pi=(\pi_i)\) can be viewed as a probability vector. Also, we have
  $
  Q_F\mathbf 1=\mathbf 1$ and
  $
  \pi^\top Q_F=\pi^\top,
  $
  so \(Q_F\) is a Markov matrix with invariant measure \(\pi\).

  Noting that
  $
  plr=D(l)h
  $
  and
  $
  pr^2=D(r)h,
  $
  it follows from \(\mathcal{P}=I-r l^\top\) that
  \[
  \mathcal{P}[pr^2]
  =
  (I-r l^\top)D(r)h
  =
  D(r)(I-\Pi)h.
  \]
  Moreover, since \(F=\lambda_F D(r)Q_FD(r)^{-1}\), we have
  \[
  D(l)\,\frac{I+F}{I-F}\,D(r)
  =
  D(\pi)\,\frac{I+\lambda_F Q_F}{I-\lambda_F Q_F}.
  \]
  Therefore,
  \[
  \psi
  =
  \Big\langle plr,\frac{I+F}{I-F}\mathcal{P}[pr^2]\Big\rangle
  =
  \Big\langle D(l)h,\frac{I+F}{I-F}D(r)(I-\Pi)h\Big\rangle
  =
  \Big\langle h,\frac{I+\lambda_F Q_F}{I-\lambda_F Q_F}(I-\Pi)h\Big\rangle_\pi,
  \]
  where
  $
  \langle u,v\rangle_\pi:=\sum_i \pi_i \overline{u_i} v_i.
  $ 
  We claim that 
  \begin{equation}\label{eq.claim-psi}
    \psi=\Big\langle (I-\Pi)h,\frac{I+\lambda_F Q_F}{I-\lambda_F Q_F}(I-\Pi)h\Big\rangle_\pi.
  \end{equation}
  Indeed, it follows from $\Pi h=\langle h\rangle_\pi\mathbf{1}$ that
  \begin{equation}\label{eq.calculate1}
    \Big\langle \Pi h,\frac{I+\lambda_F Q_F}{I-\lambda_F Q_F}(I-\Pi)h\Big\rangle_\pi=(\pi^\top h)\pi^\top\frac{I+\lambda_F Q_F}{I-\lambda_F Q_F}(I-\Pi)h.
  \end{equation}
  Since \(\pi^\top Q_F=\pi^\top\), we have
  $
    \pi^\top(I-\lambda_F Q_F)
    =
    (1-\lambda_F)\pi^\top .
  $
  Hence, it follows that
  $
    \pi^\top(I-\lambda_F Q_F)^{-1}
    =
    (1-\lambda_F)^{-1}\pi^\top,
  $
  and therefore we have
  \[
    \pi^\top\frac{I+\lambda_F Q_F}{I-\lambda_F Q_F}
    =
    \frac{1+\lambda_F}{1-\lambda_F}\pi^\top .
  \]
  Inserting the above equality back into \eqref{eq.calculate1}, it follows from $\pi^\top\mathbf{1}= 1$ that
  \[
  \Big\langle \Pi h,\frac{I+\lambda_F Q_F}{I-\lambda_F Q_F}(I-\Pi)h\Big\rangle_\pi=(\pi^\top h)\frac{1+\lambda_F}{1-\lambda_F}\pi^\top(I-\Pi)h=(\pi^\top h)\frac{1+\lambda_F}{1-\lambda_F}(1-\pi^\top\mathbf{1})\pi^\top h=0,
  \]
  which proves \eqref{eq.claim-psi}.

  Next define
  $
  y:=(I-\lambda_F Q_F)^{-1}(I-\Pi)h.
  $
  Then, by definition, we have
  \begin{equation}\label{eq.psi-calculation}
    \psi
  =
  \langle (I-\lambda_F Q_F)y,(I+\lambda_F Q_F)y\rangle_\pi
  =
  |y|_\pi^2-\lambda_F^2|Q_Fy|_\pi^2 ,
  \end{equation}
  where $|\cdot|_\pi$ is the norm corresponding to inner product $\langle\cdot,\cdot\rangle_\pi$.
  Using that \(Q_F\) is Markov and \(\pi\) is an invariant measure, Jensen's inequality implies
  $
  |Q_Fy|_\pi^2\le
  |y|_\pi^2.
  $
  Therefore, we have
  $
  \psi\ge (1-\lambda_F^2)|y|_\pi^2\ge 0.
  $

  \medskip
  Next, we consider the quantity $\psi+\sigma^2$. 
  The argument used in the proof of Lemma~\ref{lem:B-bad-direction-nonsym-local} shows that the solution \(m\) is uniformly bounded on \(N_\epsilon(\tau)\).
  Hence all the quantities appearing in Lemma \ref{lem:B-bad-direction-nonsym-local} can be extended to the real line.
  If \(\sigma(\tau)\ne0\), then \(\sigma(\tau)\) is a nonzero model parameter. Therefore \(\sigma(\tau)^2\sim1\), and continuity yields $|\sigma(z)|\sim 1$ for all $z\in N_\epsilon(\tau)$ after shrinking $\epsilon$ if necessary. Combining $\psi\ge 0$ directly yields $\psi(z)+\sigma(z)^2\sim 1$ for all $z\in N_\epsilon(\tau)$.
  
  Then, we focus on the case $\sigma(\tau)=0$. In the remainder of the proof, unless otherwise specified, all quantities without an explicit argument are evaluated at \(\tau\).
  Using $\Pi h=\sigma \mathbf{1}=0$ yields 
  \[
      h=(I-\Pi)h+\Pi h=(I-\Pi)h+\sigma\mathbf{1}=(I-\Pi)h.
  \]  
  By $\lambda_F=1$, it suffices to prove $Q_F$ is a strict contraction in the direction $y=(I-Q_F)^{-1}(I-\Pi)h=(I-Q_F)^{-1}h$ :
  \begin{equation}\label{eq.QF-contraction}
    |Q_Fy|_\pi^2< |y|_\pi^2.
  \end{equation}
  Indeed, if the above inequality holds, then combining $\sigma=0, \lambda_F=1$ and \eqref{eq.psi-calculation} yields
  \[
      \psi+\sigma^2=\psi= |y|_\pi^2-|Q_Fy|_\pi^2>0.
  \]
  Since $\psi+\sigma^2$ is a nonzero model parameter, continuity implies $\psi(z)+\sigma^2(z)\sim 1$ for all $z\in N_\epsilon(\tau)$ after shrinking $\epsilon$ if necessary.

  We prove \eqref{eq.QF-contraction} by contradiction, treating the two cases separately.
  
  \smallskip

  \textit{Assumption} (Sym): Since \(S\) is symmetric, the matrix $F$ is symmetric as well. 
    If $|Q_Fy|_\pi=|y|_\pi$, then equality must hold in Jensen's inequality for every
    row. Hence \(y\) is constant on the support of each row of \(Q_F\). Since
    $Q_F$ and $S$ have the same support, this means that
    \[
        y\in\left\{x\in\C^d: x_j=x_k~
        \text{if there exists } i \text{ such that } s_{ij}>0
        \text{ and } s_{ik}>0\right\}:=\mathcal{K}_1.
    \]
    Equivalently, \(y\) is constant on each distance-two component of the
    support graph of \(S\).

    If the support graph of \(S\) is not bipartite, then its distance-two graph
    is connected. Therefore \(y\) is constant. Since \(Q_F\mathbf1=\mathbf1\), this
    implies
    $
        h=(I-Q_F)y=0,
    $
    contradicting \(h=pr\ne0\).

    It remains to consider the bipartite case. Let
    $
        \{1,\dots,d\}=\mathbb V_+\sqcup\mathbb V_-
    $
    be the bipartition of the support graph. Since \(y\) is constant on each
    distance-two component, there exist constants \(c_+,c_-\) such that
    \[
        y_i=c_+\quad\text{for }i\in\mathbb V_+,
        \qquad
        y_i=c_-\quad\text{for }i\in\mathbb V_- .
    \]
    If \(c_+=c_-\), then \(y\) is constant and again \(h=(I-Q_F)y=0\), a
    contradiction. Thus \(c_+\ne c_-\). For \(i\in\mathbb V_+\), all neighbors
    of \(i\) lie in \(\mathbb V_-\), and hence
    $
        (Q_Fy)_i=c_-.
    $
    Similarly, for \(i\in\mathbb V_-\), we have \((Q_Fy)_i=c_+\). Therefore
    \[
        h_i=((I-Q_F)y)_i
        =
        \begin{cases}
            c_+-c_-, & i\in\mathbb V_+,\\
            c_--c_+, & i\in\mathbb V_-.
        \end{cases}
    \]
    Since \(h_i=p_i r_i\) and \(r_i>0\), it follows that \(r_i=|c_+-c_-|\) for
    all \(i\). Thus \(r\) is constant, and \(p_i\) is constant on each side of
    the bipartition with opposite signs on the two sides. After possibly
    multiplying all signs by \(-1\), we may write
    \[
        p_i=
        \begin{cases}
            1, & i\in\mathbb V_+,\\
            -1, & i\in\mathbb V_-.
        \end{cases}
    \]

    We now use the Perron equation and the Dyson equation to exclude this alternating mode.
    Since \(r\) is constant and \(Fr=r\), we have
    \begin{equation}\label{eq.perren-equation-symmetric}
      |m_i|\sum_j s_{ij}|m_j|=1,\qquad i\in\mathbb V .
    \end{equation}
    On the other hand, it follows from \(\im m=0\) that \(m\in\mathbb R^d\). Moreover, as \(p\) has no vanishing component, \(p^2=\mathbf 1\). Hence
    $
    m=p|m|.
    $
    Since \(p_j=-p_i\) whenever \(s_{ij}>0\), the Dyson equation gives
    \[
        -\frac1{p_i|m_i|}
        =
        \tau-a_i+\sum_j s_{ij}p_j|m_j|
        =
        \tau-a_i
        -
        p_i\sum_j s_{ij}|m_j| .
    \]
    Putting \eqref{eq.perren-equation-symmetric} into the above equation yields
    \[
        \tau-a_i=0, \qquad i\in\mathbb V .
    \]
    Thus the alternating mode can occur only at the exceptional point where
    \(\tau-a_i=0\) for every \(i\). This forces \(a_i\) to be independent of
    \(i\). After shifting the spectral parameter, we may assume without loss of
    generality that \(a_i=0\) for all \(i\). The exceptional point is then
    \(\tau=0\).

    It remains to show that \(\tau=0\) cannot be a zero-density point. Define 
    \[ 
    \widetilde m(z):=-\overline{m(-\overline z)}, \qquad z\in\mathbb C_+. 
    \] 
    Since \(S\) is real, a direct check shows that \(\widetilde m\) is again a physical solution of the Dyson equation on \(\mathbb C_+\). 
    It follows from the uniqueness of the physical solution that  
    \begin{equation}\label{eq.invariant-transformation}
        \widetilde m(z)=m(z),\qquad z\in\mathbb C_+.
    \end{equation}
    Evaluating at \(z=\mathrm i\eta\), where \(-\overline z=z\), yields $ m(\mathrm i\eta)=-\overline{m(\mathrm i\eta)}$.  Hence \(m(\mathrm i\eta)\) is purely imaginary for every \(\eta>0\). Passing to the boundary value at \(0\), we conclude that \(m(0)\) is purely imaginary.
    If
    $
        \langle \im m(0)\rangle=0,
    $
    then we have \(m(0)=0\), contradicting the uniform lower bound \eqref{eq.m-bound}. This excludes
    the alternating mode in the bipartite case.

      \smallskip
      \textit{Assumption}~(NB):  In this case the
      indices are directed edges \(e=(u,v)\in \vec E_{\mathcal G}\), and the Dyson equation has the form
      \begin{equation}\label{eq.qve-NB}
        -\frac1{m_{uv}}
          =
          \tau-a_v+\sum_{\substack{w\sim v\\ w\ne u}}m_{vw}.
      \end{equation}

    We prove \eqref{eq.QF-contraction} by contradiction.  Since
    \(Q_F\) is the Markov normalization of the non-backtracking matrix, its
    support is
    \[
        (u,v)\longrightarrow (v,w),
        \qquad w\sim v,\quad w\ne u .
    \]
    Jensen's inequality implies  $|Q_Fy|_\pi\le |y|_\pi $.
    Equality holds if and only if \(y\) is
    constant on the support of every row of \(Q_F\).  Since every vertex of
    \(\mathcal G\) has degree at least \(3\), this is equivalent to
    \begin{equation}\label{eq:NB-equality-space}
      y\in\left\{y\in\C^{\vec E_{\mathcal G}}:\exists(\varphi_u)_{u\in \mathbb{V}_{\mathcal G}}\in\C^{\mathbb V_{\mathcal G}}~\text{such that}~y_{uv}= \varphi_u\right\}:=\mathcal{K}_2.
    \end{equation}
    Indeed, for fixed \(v\), any two directed edges \((v,w_1)\) and
    \((v,w_2)\) can be put into the same row support by choosing
    \(u\sim v\) with \(u\ne w_1,w_2\). 

    Suppose that \eqref{eq.QF-contraction} fails. Then $y\in\mathcal{K}_2$ implies that $h=(I-\Pi)h=(I-Q_F)y$ has the form
    \begin{equation}\label{eq:h-coboundary}
        h_{uv}=\varphi_u-\varphi_v .
    \end{equation}

    We now show that this representation is incompatible with the Dyson equation and
    the Perron equation.
    Substituting $|m|=pm$ and $r=ph$ into the Perron equation $Fr=r$ yields

    \[
        h_{uv}
        =
        m_{uv}\sum_{\substack{w\sim v\\ w\ne u}}m_{vw} h_{vw} .
    \]
    Then, plugging \eqref{eq:h-coboundary} into the above equation gives
    \begin{equation}\label{eq:bad-perron-terminal}
        \varphi_u-\varphi_v
        =
        m_{uv}\sum_{\substack{w\sim v\\ w\ne u}}
        m_{vw}(\varphi_v-\varphi_w).
    \end{equation}
    For simplicity, we define
    \[
        M_v:=\sum_{w\sim v}m_{vw}(\varphi_v-\varphi_w),\qquad
        N_v:=\sum_{w\sim v}m_{vw}.
    \]
    Adding the missing backtracking term \(m_{uv}m_{vu}(\varphi_v-\varphi_u)\) to both sides, we get
    \begin{equation}\label{eq:center-equation-terminal}
        (1-m_{uv}m_{vu})(\varphi_u-\varphi_v)
        =
        m_{uv}M_v.
    \end{equation}
    On the other hand, the Dyson equation \eqref{eq.qve-NB} gives
    \begin{equation}\label{eq:reverse-identity-terminal}
        1-m_{uv}m_{vu}
        =
        -m_{uv}\big(\tau-a_v+N_v\big).
    \end{equation}
    Combining \eqref{eq:center-equation-terminal} and \eqref{eq:reverse-identity-terminal} yields 
    \begin{equation}
      M_v=-\big(\tau-a_v+N_v\big)(\varphi_u-\varphi_v),\qquad u\sim v.
    \end{equation}

    We first exclude the case
    $
        \tau-a_v+N_v=0 .
    $
    By \eqref{eq:reverse-identity-terminal}, this gives
    $
        m_{uv}m_{vu}=1
    $
    for every $ u\sim v $.
    Then, it follows from \eqref{eq:h-coboundary} that
    \begin{align*}
        m_{uv}(\varphi_u-\varphi_v)&=m_{uv}h_{uv}=|m_{uv}|r_{uv}>0\\
        m_{vu}(\varphi_v-\varphi_u)&=m_{vu}h_{vu}=|m_{vu}|r_{vu}>0 .
    \end{align*}
    Since \(m_{uv}m_{vu}=1\), the two numbers \(m_{uv}\) and \(m_{vu}\) have the
    same sign, which contradicts \(\varphi_u-\varphi_v=-(\varphi_v-\varphi_u)\).  Hence $\tau-a_v+N_v=0$ is impossible.

    It remains to consider the case
    $
        \tau-a_v+N_v\ne0 .
    $
    For fixed \(v\), both \(M_v\) and \(\tau-a_v+N_v\) are independent of the
    neighbor \(u\).  Hence all neighbors \(u\sim v\) have the same value of
    \(\varphi_u-\varphi_v\), and therefore the same value of \(\varphi_u\).
    Equivalently, \(\varphi\) is constant on each distance-two component of the
    base graph.

    If \(\mathcal G\) is non-bipartite, the distance-two graph is
    connected.  Hence \(\varphi\) is constant on all of \(\mathbb{V}_{\mathcal G}\), and therefore
    \[
        h_{uv}=\varphi_u-\varphi_v=0 .
    \]
    This contradicts $h=pr\neq 0$.  Thus in the non-bipartite case \eqref{eq.QF-contraction} holds.

    If $\mathcal G$ is bipartite, we write
    $
    \mathbb V_{\mathcal G}=\mathbb V_+\sqcup \mathbb V_-
    $
    for its bipartition, meaning that every edge of $\mathcal G$ has one endpoint in $\mathbb V_+$ and the other endpoint in $\mathbb V_-$. Then the preceding
    argument shows that \(\varphi\) is constant on each side:
    \[  
        \varphi_i=c_+\quad\text{for }i\in\mathbb V_+,
        \qquad
        \varphi_i=c_-\quad\text{for }i\in\mathbb V_-.
    \]
    If \(c_+=c_-\), then \(y\) is constant and again \(h=0\), a contradiction. Thus $c_+\neq c_-$. Therefore the only possible non-trivial case is the alternating mode.
    For such a mode,
    $
        \varphi_v-\varphi_w=-(\varphi_u-\varphi_v)
    $
    whenever \((u,v)\to(v,w)\) is a non-backtracking transition.  If this
    alternating mode is non-zero, then \eqref{eq:bad-perron-terminal} implies
    \[
        1
        =
        -m_{uv}\sum_{\substack{w\sim v\\ w\ne u}}m_{vw}.
    \]
    On the other hand, the Dyson equation gives
    \[
        -1
        =
        m_{uv}(\tau-a_v)
        +
        m_{uv}\sum_{\substack{w\sim v\\ w\ne u}}m_{vw}.
    \]
    Combining the last two identities yields
    $
        m_{uv}(\tau-a_v)=0 .
    $
    Since \(m_{uv}\ne0\), we obtain
    \[
        \tau-a_v=0,
        \qquad  v\in \mathbb{V}_{\mathcal G}.
    \]
    Thus the only remaining obstruction is again the above exceptional case.
    This is precisely the exceptional case already excluded in the proof of
    Assumption (Sym), after the harmless shift \(\mathbf a\equiv0\), and therefore the same
    argument shows that it cannot occur at a point with
    \(\langle \im m(\tau)\rangle=0\).

    \medskip
    Under either Assumption (Sym) or Assumption (NB), equality in Jensen's
    inequality is impossible. Consequently, we always have \(\psi(\tau)>0\) when $\sigma(\tau)=0$.
    This concludes the proof.

\end{proof}

With the stability expansion and the nondegeneracy estimate
\(\psi+\sigma^2\sim1\) established above, we can now follow the strategy of
\cite{AEK2017singularities} to reduce the local analysis near a zero-density
boundary point to a scalar cubic equation. 
This scalar equation encodes the classification of the singularity. If
\(\sigma(\tau)\ne0\), the quadratic term is dominant and the boundary point is a
regular edge, where the density has square-root behavior. If
\(\sigma(\tau)=0\), then the nondegeneracy estimate forces
\(\psi(\tau)\sim1\), and the cubic term becomes dominant; this is the cusp case,
where the density has cubic-root behavior.

\begin{proposition}[Cubic equation]\label{pro:Cubic equation} 
    Let $\tau \in \partial \mathfrak{S}$ be a regular edge or cusp. Assume that there exists $\epsilon>0$ depending only on model parameters such that $\psi+\sigma^2\sim 1$ for all $z\in N_\epsilon(\tau)$. Then, for all $\omega\in(-\epsilon,\epsilon)$, the following statements hold.
    The complex valued function 
    \begin{equation}\label{def.Theta}
    \Theta(\omega;\tau):=\, \avgB{l(\tau) \frac{m(\tau+\omega)-m(\tau)}{|\1m(\tau)|}},\qquad |\omega|\le\epsilon
    \end{equation}
    describes the change of $m$ around $\tau$ to leading order,
    \begin{equation}\label{eq.leading order of change}
    m(\tau+\omega)-m(\tau)
    =
    \Theta(\omega;\tau)|m(\tau)|r(\tau)+O(|\Theta|^2+|\omega|).
    \end{equation}

    Moreover, the function $\Theta(\omega;\tau)$ solves the approximate cubic equation
    \begin{equation}\label{eq:cubic-nonsym}
    \psi(\tau)\,\Theta(\omega;\tau)^3
    +\sigma(\tau)\,\Theta(\omega;\tau)^2
    +\mu(\tau)\,\omega
    =
    e(\omega;\tau),
    \end{equation}
    where \(\mu,\sigma,p,\psi\) are defined as in Lemma~\ref{lem:B-bad-direction-nonsym-local}.
    The error term $e(\omega;\tau)$ satisfies
    \begin{subequations}
    \label{eq:error-bounds-nonsym}
    \begin{align}
    |e(\omega;\tau)|
    &\lesssim
    |\omega|\,|\Theta(\omega;\tau)|+|\omega|^2,
    \\
    |\im e(\omega;\tau)|
    &\lesssim
    |\omega|\,\im \Theta(\omega;\tau).
    \end{align}
    \end{subequations}
\end{proposition}

\begin{proof}
  We shall use the notation introduced in Lemma~\ref{lem:B-bad-direction-nonsym-local}. By the argument in the proof of Lemma~\ref{lem:B-bad-direction-nonsym-local}, the solution \(m\) is uniformly bounded on \(N_\epsilon(\tau)\).
  We first prove that $m(z)$ is $1/3$-Hölder continuous in $N_\epsilon(\tau)$. 
  The key input is the cubic stability estimate
    $
    \psi+\sigma^2\sim 1
    $
    for all $z\in N_\epsilon(\tau)$.
  Using \eqref{eq.Rinverse-bound} and putting $\psi+\sigma^2\sim 1$ into \eqref{eq:beta-expansion-local} directly yields $\|R^{-1}\|_2\lesssim\langle\im m\rangle^{-2}$.
  The $1/3$-Hölder continuity then follows by the same argument as in \cite[Proof of Proposition~5.1 and Proof of Corollary~2.7]{AEK2017singularities}.

  The remaining argument is essentially identical to the proof of \cite[Proposition 6.2]{AEK2017singularities}. 
  The only difference is notational: in the present non-symmetric setting the unstable direction is described by the right and left Perron eigenvectors \(r\) and \(l\), so that the decomposition is made with the projection \(\mathcal P=I-r l^{\top}\) and the scalar coordinate \(\Theta\) defined in \eqref{def.Theta}, whereas in the symmetric case one has \(r=l=f\) and \(\mathcal P\) reduces to the orthogonal projection \(Q=I-\langle f,\cdot\rangle f\). 
  With these replacements, and using the already established \(1/3\)-Hölder continuity of \(m\), the estimates and the derivation of the cubic equation proceed verbatim as in the symmetric case; we therefore refer the reader to \cite[Proposition 6.2]{AEK2017singularities} for the details.
\end{proof}

\begin{proof}[Proof of Theorem \ref{thm:singularities}]
Assume that $S$ satisfies either (Sym) or (NB).
By Proposition~\ref{prop:psi-positive-and-nondegenerate},
we have $\sigma(z)^2+\psi(z)\sim 1$ with $z\in N_\epsilon(\tau)$,
for sufficiently small $\epsilon>0$.
Hence Proposition~\ref{pro:Cubic equation} applies.
Following the argument in
\cite[Proof of Theorem~2.6]{AEK2017singularities},
we distinguish two cases.

If $\sigma(\tau)\neq 0$, then
$|\sigma(z)|\sim 1$ near $\tau$.
The quadratic term in the cubic equation is dominant,
and the corresponding physical branch exhibits
square-root behavior on one side of $\tau$,
while the other side is a spectral gap.
Thus $\tau$ is a regular edge.

If $\sigma(\tau)=0$, the above nondegeneracy
implies $\psi(\tau)>0$, and hence
$\psi(z)\sim 1$ near $\tau$.
The cubic-root analysis in
\cite[Proof of Theorem~2.6]{AEK2017singularities}
then shows that the density exhibits cubic-root
behavior on both sides of $\tau$.
Thus $\tau$ is a regular cusp.

These two cases exhaust all possibilities and
yield the asserted edge and cusp asymptotics.
\end{proof}

\section{Bulk Stability}
This section investigates bulk stability. We first gives a rough stability estimate. It is useful when $z$ is not so close to the real axis. 
\begin{proposition}[Simple bound for the stability operator]\label{pro.simple-bound-stability}
Suppose \(S\in\R^{d\times d}\) is a nonnegative and irreducible matrix, and the solution $m(z)$ is uniformly bounded on $\Cp(I)$ where $I\subset\R$ is a closed interval.
Then, we have $\|(I-m(z)^2S)^{-1}\|_2\lesssim 1+\eta^{-1}$ for all $z\in\C_+(I)$.
\end{proposition}

\begin{proof}
We first consider the regime of large \(\eta\). Theorem~\ref{thm.bound} implies that 
$
    |m_i(z)|\lesssim (|z|+1)^{-1}\le \eta^{-1}.
$
Therefore, there exists a finite constant $C>0$ depending only on model parameters such that
\[
    \|m^2S\|_2
    \le \|S\|_2\max_i |m_i|^2
    \le C\|S\|_2\eta^{-2}.
\]
Choosing \(\eta_0>0\) depending only on model parameters sufficiently large, we have
$
    \|m^2S\|_2\le 1/2
$ for all $\eta>\eta_0$.
Hence \(I-m^2S\) is invertible, and the Neumann series yields
\begin{equation}\label{eq.large-regime-inverse-bound}
    \|(I-m^2S)^{-1}\|_2
    \le
    \sum_{k=0}^\infty \|m^2S\|_2^k
    =
    \frac1{1-\|m^2S\|_2}
    \le 2.
\end{equation}
Thus
$
    \|(I-m^2S)^{-1}\|_2\lesssim 1+\eta^{-1}
$
for \(\eta>\eta_0\).

It remains to consider \(0<\eta\le\eta_0\). 
Equality \eqref{eq:QVE-imag-bound}, $0<\eta\le\eta_0$ and $S\im m>0$ imply that there exists a finite constant $c>0$ depending  only on model parameters such that
\begin{equation*}
  \forall 0<\eta\le\eta_0,\qquad\frac{|m_i|^2}{\im m_i}=\frac{1}{\eta+(S\im m)_i}\ge c,\qquad i=1,\dots d.
\end{equation*}
Consequently,
\[
   (|m|^2S\im m)_i
    =
    \im m_i-\eta |m_i|^2
    \le
    (1-c\eta)\im m_i,\qquad i=1,\dots,d.
\]

For fixed \(z\in\mathbb C_+\) with \(0<\eta\le\eta_0\), define the weighted max norm
\[
    |x|_{m}:=\max_i \frac{|x_i|}{\im m_i(z)},\qquad x\in\mathbb C^d.
\]
Since \(|x|\le |x|_m \im m\) componentwise and
\(|m|^2S\) is nonnegative, we get
\[
    \big||m|^2Sx\big|
    \le
    |x|_m |m|^2S\im m
    \le
    (1-c\eta)|x|_m \im m,
\]
which implies that
$
    \bigl||m|^2Sx\bigr|_m
    \le
    (1-c\eta)|x|_m .
$

Now take an arbitrary vector \(u\in\mathbb{C}^d\), and we have the decomposition
$
    u=m^2Su+ (I-m^2S)u.
$
Taking absolute values componentwise and using \(S\ge0\), we obtain
$
    |u|\le|m|^2S|u|+|(I-m^2S)u|.
$
Applying the weighted max norm and using the contraction estimate above gives
\[
    |u|_m
    \le
    (1-c\eta)|u|_m+|(I-m^2S)u|_m,
\]
which implies 
$
    |u|_m
    \le
    (c\eta)^{-1}|(I-m^2S)u|_m 
$
for arbitrary $u\in\C^d$ when $0<\eta\le\eta_0$.

It remains to pass from \(|\cdot|_m\) to the Euclidean norm. By definition and direct computation, we have 
\[
    |u|_2
    \le
    \sqrt{d}(\max_i \im m_i)|u|_m\qquad\text{and}\qquad  |(I-m^2S)u|_m
    \le
    \frac{1}{\min_i \im m_i}|(I-m^2S)u|_2 .
\]
Combining the comparability $\im m_i\sim\im m_j$ from Theorem \ref{thm.bound} yields
\[
    |u|_2
    \le
    \frac{\sqrt{d}}{c\eta}
    \frac{\max_i \im m_i}{\min_i \im m_i}
    |(I-m^2S)u|_2
    \le
    \frac{C}{\eta}|(I-m^2S)u|_2,\qquad u\in\C^d.
\]
Thus we have
$
    \|(I-m^2S)^{-1}\|_2\lesssim 1+\eta^{-1}
$
when $0<\eta\le\eta_0$. This concludes the proof.
\end{proof}
\bigskip
The above stability bound deteriorates as \(\eta\downarrow 0\). Therefore, a different argument is needed in the small-\(\eta\) regime. 
In fact, the stability behavior is closely related to the structure of the support digraph $\mathcal{G}_S$ associated with $S$. This motivates the following definition of the resonance set.

\begin{definition}[Resonance set]\label{def.resonance set}
For a nonnegative matrix \(Q\in \mathbb R^{d\times d}\), the resonance set associated with \(Q\) is defined by
\[
\mathscr{R}_Q
:=
\left\{
\diag(v_1,\dots,v_d):
v_1,\dots,v_d\in \mathbb S^1,
\prod_{i\in \mathscr{C}} v_i = 1~
\text{for every directed cycle } \mathscr{C} \text{ of } E_Q
\right\},
\]
where directed cycles are defined in \eqref{def.cycle}.
\end{definition}

The lemma below characterizes the resonance set \(\mathscr R_Q\) as the zero set of the edge defect \(\Delta_Q\). 
Throughout the sequel, we will freely switch between these equivalent descriptions of the resonance set, choosing the one most convenient for the argument at hand.

\begin{lemma}\label{lem.zero-set-of-edge-defect}
Let \(Q\in \mathbb R^{d\times d}\) be a nonnegative irreducible matrix. For every unitary diagonal matrix 
\[
V=\diag(v_1,\dots,v_d),\qquad |v_i|=1,
\]
define
\[
\Delta_Q(V)
:=
\inf_{\gamma\in (\mathbb S^1)^d}
\max_{(i,j)\in E_Q}
|v_i-\bar\gamma_i\gamma_j|.
\]
Then we have
$
\Delta_Q(V)=0$ if and only if $ V\in \mathscr{R}_Q.
$
\end{lemma}

\begin{proof}
We prove the two implications separately.

First assume that \(\Delta_Q(V)=0\). Since \((\mathbb S^1)^d\) is compact and the function
\[
    \gamma\mapsto \max_{(i,j)\in E_Q}|v_i-\bar\gamma_i\gamma_j|
\]
is continuous, the infimum is attained. Hence by definition there exists
\(\gamma=(\gamma_1,\dots,\gamma_d)\in (\mathbb S^1)^d\) such that
$
v_i=\bar\gamma_i\gamma_j
$
for all $(i,j)\in E_Q$.
For any directed cycle
\[
\mathscr{C}:i_0\to i_1\to \cdots \to i_{\ell-1}\to i_\ell=i_0,
\]
it follows that
$
\prod_{r=0}^{\ell-1} v_{i_r}
=
\prod_{r=0}^{\ell-1} \bar\gamma_{i_r}\gamma_{i_{r+1}}
=1.
$
Hence \(V\in \mathscr{R}_Q\) by Definition \ref{def.resonance set}.

Next assume that \(V\in \mathscr{R}_Q\). Since \(Q\) is irreducible, its support digraph is strongly connected. Fix a root vertex \(o\in \{1,\dots,d\}\), and set
$
\gamma_o:=1.
$
For any vertex \(u\), choose a directed path
\[
u=i_0\to i_1\to \cdots \to i_\ell=o
\]
from \(u\) to \(o\), and define
\begin{equation}\label{def.gammau}
  \gamma_u:=\prod_{r=0}^{\ell-1} \bar v_{i_r}.  
\end{equation}

We first show that this definition is independent of the chosen path.
Indeed, let
\[
P_1:\quad u=i_0\to i_1\to \cdots \to i_{\ell_1}=o,
\qquad
P_2:\quad u=j_0\to j_1\to \cdots \to j_{\ell_2}=o
\]
be two directed paths from \(u\) to \(o\). Since the graph is strongly connected, there exists a directed path
\[
P_3:\quad o=k_0\to k_1\to \cdots \to k_{\ell_3}=u
\]
from \(o\) back to \(u\). Then \(P_1\cup P_3\) and \(P_2\cup P_3\) are directed cycles. Since \(V\in \mathscr{R}_Q\), the cycle condition gives
\[
\Bigl(\prod_{r=0}^{\ell_1-1}  v_{i_r}\Bigr)\Bigl(\prod_{s=0}^{\ell_3-1}  v_{k_s}\Bigr)=1,
\qquad
\Bigl(\prod_{r=0}^{\ell_2-1}  v_{j_r}\Bigr)\Bigl(\prod_{s=0}^{\ell_3-1} v_{k_s}\Bigr)=1.
\]
Therefore, it follows that
$
\prod_{r=0}^{\ell_1-1} \bar v_{i_r}
=
\prod_{r=0}^{\ell_2-1} \bar v_{j_r},
$
so \(\gamma_u\) is well defined.

For any \((i,j)\in E_Q\), choose a directed path from \(j\) to \(o\),
\[
j=i_0\to i_1\to \cdots \to i_\ell=o.
\]
Appending the edge \(i\to j\), we obtain a directed path from \(i\) to \(o\). By \eqref{def.gammau},
\[
\bar\gamma_i
=
v_i\Bigl(\prod_{r=0}^{\ell-1}  v_{i_r}\Bigr)
=
v_i\bar\gamma_j.
\]
Thus, it follows that
$v_i=\bar\gamma_i\gamma_j$ for all $(i,j)\in E_Q$,
which implies
$
\max_{(i,j)\in E_Q}|v_i-\bar\gamma_i\gamma_j|=0.
$
Hence we have \(\Delta_Q(V)=0\), and this proves
$
\Delta_Q(V)=0
\Longleftrightarrow
V\in \mathscr{R}_Q.
$
\end{proof}

\bigskip
The following lemma provides the basic stability estimate for a Markov matrix in terms of its resonance set. 
It states that \(U-\lambda Q\) cannot have a small singular direction unless the diagonal phase \(U\) is close to the resonance set \(\mathscr R_Q\); the lower bound contains both the trivial contribution \(1-\lambda\) and the genuinely graph-theoretic contribution \(\dist(U,\mathscr R_Q)^2\). 

We will apply this lemma to the Markov matrix \(Q_F\) obtained by normalizing the matrix \(F\), as in Proposition~\ref{pro.I-F-second}.
This normalization allows us to transfer the stability analysis of the original operator to a Markov-type problem, where the only possible obstruction is precisely described by the resonance set.

\begin{lemma}\label{lem:UminusQ-graph-resonance}
Let \(Q=(q_{ij})_{i,j=1}^d\in\mathbb R^{d\times d}\) be a nonnegative irreducible matrix such that \(Q\mathbf 1=\mathbf 1\).
Then there exists a constant \(c>0\) depending only on the support of \(Q\), such that for every \(0\le \lambda\le 1\), every diagonal unitary matrix
\[
U=\diag(u_1,\dots,u_d),\qquad |u_i|=1,
\]
and every \(\mathbf z\in\mathbb C^d\),
\begin{equation}\label{eq.bulk-stability-lemma}
  |(U-\lambda Q)\mathbf z|_\infty
\ge\max\left\{1-\lambda,c\kappa_Q^{2d}\dist(U,\mathscr{R}_Q)^2\right\}|\mathbf z|_\infty,
\end{equation}
where
$
\kappa_Q:=\min\{q_{ij}:q_{ij}>0\}
$
and $\dist(U,\mathscr{R}_Q):=\min_{V\in\mathscr{R}_Q}\max_{i}|u_i-v_i|$.
\end{lemma}

\begin{proof}
By homogeneity, we assume \(|\mathbf z|_\infty=1.\)
Since \(Q\) is irreducible, its support digraph $\mathcal{G}_Q$ is strongly connected. Choose \(i_*\in\{1,\dots,d\}\) such that
$
|z_{i_*}|=1.
$
From
\[
u_i z_i-\lambda\sum_{k=1}^d q_{ik}z_k=((U-\lambda Q)\mathbf z)_i
\]
and the triangle inequality, we obtain
\begin{equation}\label{eq:uz-triangle}
  1=|u_{i_*}z_{i_*}|
\le \lambda\sum_{k=1}^d q_{i_*k}|z_k|+|(U-\lambda Q)\mathbf z|_\infty
\le \lambda+|(U-\lambda Q)\mathbf z|_\infty.
\end{equation}
Therefore, it follows that
\begin{equation}\label{eq:rho-close-1}
1-\lambda\le |(U-\lambda Q)\mathbf z|_\infty.
\end{equation}

It remains to prove $|(U-\lambda Q)\mathbf z|_\infty\ge c\kappa_Q^{2d}\dist(U,\mathscr{R}_Q)^2$.
Since \(\dist(U,\mathscr{R}_Q)\le 2\) and \(0\le 1-\lambda\le 1\), the desired estimate is trivial whenever
$
|(U-\lambda Q)\mathbf z|_\infty\ge \kappa_Q^{2d}\epsilon_0,
$
provided \(\epsilon_0>0\) is chosen sufficiently small. Thus it remains to consider the case
\begin{equation}\label{eq:small-epsilon-graph}
|(U-\lambda Q)\mathbf z|_\infty\le \kappa_Q^{2d}\epsilon_0,
\end{equation}
where $\kappa_Q^{2d}\epsilon_0$ is small enough.
In particular, using \eqref{eq:rho-close-1}, we have
$
\lambda\ge 1-\kappa_Q^{2d}\epsilon_0\ge \frac12.
$

Next, we prove
\[
|(U-\lambda Q)\mathbf z|_\infty
\ge c\kappa_Q^{2d}\dist(U,\mathscr{R}_Q)^2\,|\mathbf z|_\infty
\]
in three steps.

\medskip
\noindent\textit{Step 1: propagation of the moduli.}
For each \(i\), define
$
h_i:=1-|z_i|\in[0,1].
$
The same argument as \eqref{eq:uz-triangle} implies
\[
1-h_i=|z_i|
\le \lambda\sum_{k=1}^d q_{ik}|z_k|+|(U-\lambda Q)\mathbf z|_\infty
=\lambda\Bigl(1-\sum_{k=1}^d q_{ik}h_k\Bigr)+|(U-\lambda Q)\mathbf z|_\infty.
\]
Rearranging the above inequality yields
\[
\lambda\sum_{k=1}^d q_{ik}h_k
\le h_i-(1-\lambda)+|(U-\lambda Q)\mathbf z|_\infty
\le h_i+|(U-\lambda Q)\mathbf z|_\infty.
\]
Since \(\lambda\ge \frac12\), it follows that
\begin{equation}\label{eq:weighted-h}
\sum_{k=1}^d q_{ik}h_k\le 2\bigl(h_i+|(U-\lambda Q)\mathbf z|_\infty\bigr).
\end{equation}
If \(q_{ij}>0\), then \eqref{eq:weighted-h} implies
$
\kappa_Q h_j\le q_{ij}h_j\le 2\bigl(h_i+|(U-\lambda Q)\mathbf z|_\infty\bigr),
$
and therefore
\begin{equation}\label{eq:edge-h-bound}
h_j\le 2\kappa_Q^{-1}\bigl(h_i+|(U-\lambda Q)\mathbf z|_\infty\bigr),
\qquad(i,j)\in E_Q.
\end{equation}

Since \(Q\) is irreducible, for any \(i,j\) there exists a directed path
\[
  i=i_0\to i_1\to\cdots\to i_\ell=j .
\]
Choosing a shortest such path, it has no repeated vertices; otherwise one could
remove a cycle and obtain a shorter path. Hence the path contains at most \(d\)
vertices, and therefore \(\ell\le d-1\).
Starting from
$
h_{i_*}=1-|z_{i_*}|=0,
$
and iterating \eqref{eq:edge-h-bound} along the shortest directed path from \(i_*\) to any given \(i\), we obtain
$
h_i\le C\kappa_Q^{-(d-1)}|(U-\lambda Q)\mathbf z|_\infty,
$
where \(C>0\) depends only on the support of \(Q\). Equivalently,
\begin{equation}\label{eq:modulus-close-1}
|z_i|\ge 1-C\kappa_Q^{-(d-1)}|(U-\lambda Q)\mathbf z|_\infty,
\qquad i=1,\dots,d.
\end{equation}

\medskip
\noindent\textit{Step 2: one-step phase alignment on every positive edge.}
Fix \((i,j)\in E_Q\). For simplicity, define
\[
s_i:=\sum_{k=1}^d q_{ik}z_k.
\]
Then, it follows from $u_i z_i-\lambda s_i=((U-\lambda Q)\mathbf z)_i$ that
$
|u_i z_i-\lambda s_i|\le |(U-\lambda Q)\mathbf z|_\infty.
$
By \eqref{eq:modulus-close-1}, \eqref{eq:small-epsilon-graph} and \(\epsilon_0\) sufficiently small,
\begin{equation}\label{eq:s-lowerbound}
  |s_i|\ge \lambda|s_i|\ge |u_i z_i|-|(U-\lambda Q)\mathbf z|_\infty\ge 1-C\kappa_Q^{-(d-1)}|(U-\lambda Q)\mathbf z|_\infty>1-C\kappa_Q^{d+1}\epsilon_0>0.
\end{equation}

For each fixed \(i\), write the complex number \(s_i\) in polar form as
$
s_i=e^{\mathrm i\theta_i}|s_i|,
$
where \(\theta_i=\arg(s_i)\).
We then rotate the complex plane by the angle \(-\theta_i\), and express the rotated quantity \(z_k\) in terms of its real and imaginary parts by setting
\[
e^{-\mathrm i\theta_i}z_k=:x_k+\mathrm i y_k,\qquad x_k,y_k\in\mathbb R.
\]
Therefore, it follows from \eqref{eq:s-lowerbound} that
\[
\sum_{k=1}^d q_{ik}x_k
=
e^{-\mathrm i\theta_i}\sum_{k=1}^d q_{ik}z_k
=
|s_i|
\ge 1-C\kappa_Q^{-(d-1)}|(U-\lambda Q)\mathbf z|_\infty.
\]
Since \(x_k\le |z_k|\le 1\), we get
\[
0\le \sum_{k=1}^d q_{ik}(1-x_k)\le C\kappa_Q^{-(d-1)}|(U-\lambda Q)\mathbf z|_\infty.
\]
If $q_{ij}>0$, it follows from \(q_{ij}\ge \kappa_Q\) that
$
1-x_j\le C\kappa_Q^{-d}|(U-\lambda Q)\mathbf z|_\infty.
$
Using \(x_j^2+y_j^2=|z_j|^2\le 1\), we obtain
\[
y_j^2\le 1-x_j^2\le 2(1-x_j)\le C\kappa_Q^{-d}|(U-\lambda Q)\mathbf z|_\infty.
\]
Therefore, since we always have
$
|(U-\lambda Q)\mathbf z|_\infty\le C|(U-\lambda Q)\mathbf z|_\infty^{1/2},
$
it can be deduced that
\begin{equation}\label{eq.triangle-ze}
  \begin{aligned}
    \forall(i,j)\in E_Q,\qquad
|z_j-e^{\mathrm i\theta_i}|
&=
|e^{-\mathrm i\theta_i}z_j-1|
\le |x_j-1|+|y_j| \\
&\le C\kappa_Q^{-d}|(U-\lambda Q)\mathbf z|_\infty
+C^{1/2}\kappa_Q^{-d/2}|(U-\lambda Q)\mathbf z|_\infty^{1/2}\\
&\le C\kappa_Q^{-d}|(U-\lambda Q)\mathbf z|_\infty^{1/2}.
\end{aligned}
\end{equation}
On the other hand, using the triangle inequality and $|s_i|=e^{-i\theta_i}s_i$ yields
\begin{equation}\label{eq.triangle-uz}
  |u_i z_i-e^{\mathrm i\theta_i}|
\le |u_i z_i-\lambda s_i|+|\lambda s_i-e^{\mathrm i\theta_i}|= |u_i z_i-\lambda s_i|+\bigl|\lambda|s_i|-1\bigr|.
\end{equation}
Since we have
$
\big||z_i|-\lambda|s_i|\big|\le |u_iz_i-\lambda s_i|\le \big|(U-\lambda Q)\mathbf{z}\big|_\infty,
$
it follows from the triangle inequality and \eqref{eq:modulus-close-1} that
\begin{equation*}
  \big|\lambda|s_i|-1\big|\le\big|\lambda|s_i|-|z_i|\big|+\big||z_i|-1\big|\le C\kappa_Q^{-(d-1)}\big|(U-\lambda Q)\mathbf{z}\big|_\infty.
\end{equation*}
Putting the above inequality into \eqref{eq.triangle-uz} yields
\begin{equation}\label{eq.triangle-uz-final}
  |u_i z_i-e^{\mathrm i\theta_i}|\le C\kappa_Q^{-(d-1)}|(U-\lambda Q)\mathbf z|_\infty\le C\kappa_Q^{-(d-1)}|(U-\lambda Q)\mathbf z|_\infty^{1/2}.
\end{equation}
Combining \eqref{eq.triangle-ze} and \eqref{eq.triangle-uz-final}, we can obtain
\begin{equation}\label{eq:edge-phase-alignment}
    |z_j-u_i z_i|\le|z_j-e^{\mathrm i\theta_i}|+|e^{\mathrm i\theta_i}-u_iz_i| \le C\kappa_Q^{-d}|(U-\lambda Q)\mathbf z|_\infty^{1/2},\qquad(i,j)\in E_Q.
\end{equation}

Define
$
\gamma_i:=z_i/|z_i|\in\mathbb S^1
$ for $1\le i\le d$. This definition is well posed, since \eqref{eq:modulus-close-1} implies that
$
    |z_i|\ge 1-C\kappa_Q^{d+1}\epsilon_0
    \ge \frac12,
$
provided that \(\epsilon_0>0\) is chosen sufficiently small.
Then, it follows from \eqref{eq:modulus-close-1} that
\begin{equation}\label{eq:z-close-gamma}
  |\gamma_i-z_i|=\bigl|1-|z_i|\bigr|=h_i\le C\kappa_Q^{-(d-1)}|(U-\lambda Q)\mathbf z|_\infty,\qquad i=1,\cdots,d.
\end{equation}
Hence  \eqref{eq:edge-phase-alignment} and \eqref{eq:z-close-gamma} imply
\[
\forall (i,j)\in E_Q,\qquad
|u_i\gamma_i-\gamma_j|
\le |u_i\gamma_i-u_i z_i|+|u_i z_i-z_j|+|z_j-\gamma_j|
\le C\kappa_Q^{-d}|(U-\lambda Q)\mathbf z|_\infty^{1/2},
\]
Using \(|\gamma_i|=1\) and taking the supremum over \((i,j)\in E_Q\), we obtain
\begin{equation}\label{eq:edge-defect-bound}
\max_{(i,j)\in E_Q}|u_i-\bar{\gamma}_i\gamma_j|=\max_{(i,j)\in E_Q}|u_i\gamma_i-\gamma_j|
\le C\kappa_Q^{-d}|(U-\lambda Q)\mathbf z|_\infty^{1/2}.
\end{equation}

\noindent\textit{Step 3: Relation to the resonance set $\mathscr{R}_Q$}. 
In this step, we use \eqref{eq:edge-defect-bound} to show \eqref{eq.bulk-stability-lemma}.

Since $Q\mathbf 1=\mathbf 1$, every row of $Q$ contains at least one positive entry. As $Q$ is also irreducible, its support digraph is strongly connected, and hence it contains at least one directed cycle. Removing repetitions from any directed cycle, we obtain a simple directed cycle. 
Since the vertex set is finite, there are only finitely many simple directed cycles.

Now fix any simple directed cycle 
\[
\mathscr{C}: i_0\to i_1\to \cdots \to i_{\ell-1}\to i_\ell=i_0.
\]
Define the auxiliary variables
$
\xi_j:=u_{i_j}\gamma_{i_j}\bar\gamma_{i_{j+1}}
$ for $0\le j\le \ell-1$, where $i_j$ is the vertex of the above simple cycle $\mathscr{C}$.
Since $\mathscr{C}$ is a cycle, we have 
$
  \prod_{j=0}^{\ell-1}\xi_j=\prod_{j=0}^{\ell-1}u_{i_j}.
$
It follows from \eqref{eq:edge-defect-bound} and $|\gamma_i|=1$ that 
\begin{equation}
  |\xi_j-1|=|u_{i_j}-\bar{\gamma}_{i_j}\gamma_{i_{j+1}}|\le C\kappa_Q^{-d}|(U-\lambda Q)\mathbf z|_\infty^{1/2},\qquad j=0,\dots,\ell-1.
\end{equation}
Then, using repeatedly that
\[
|zw-1|\le |z-1|+|w-1|,
\qquad \mbox{when}~|z|=|w|=1,
\]
we obtain, for every simple directed cycle $\mathscr{C}$,
\begin{equation}\label{eq.u-cycle-bound}
  \left|\prod_{i\in \mathscr{C}}u_i-1\right|
=
\left|\prod_{k=0}^{\ell-1}\xi_k-1\right|
\le \sum_{k=0}^{\ell-1}|\xi_k-1|
\le C\kappa_Q^{-d}|(U-\lambda Q)\mathbf z|_\infty^{1/2}.
\end{equation}

Next, since $U$ is a unitary matrix, we can write
$
u_i=e^{\mathrm i\hat\theta_i}
$ 
with
$\hat\theta:=(\hat\theta_1,\dots,\hat\theta_d)\in\mathbb R^d.
$
For each simple directed cycle $\mathscr{C}$, let $\mathbf{e}(\mathscr{C})\in\mathbb Z^d$ be its vertex-counting vector, namely
\[
[\mathbf{e}(\mathscr{C})]_k=
\begin{cases}
1,& k\in \mathscr{C},\\
0,& k\notin \mathscr{C}.
\end{cases}
\]
Then, it follows that
\begin{equation}\label{eq.u-cycle-transform}
  \Big|\prod_{i\in \mathscr{C}}u_i-1\Big|=\big| e^{\mathrm i\langle \mathbf{e}(\mathscr{C}),\hat\theta\rangle}-1\big|.
\end{equation}

Denote by 
\[
\mathscr{L}:=\spn_{\mathbb Z}\{\mathbf{e}(\mathscr{C}): \mathscr{C}\text{ is a simple directed cycle}\}\subset\mathbb Z^d
\]
the subgroup of $\mathbb{Z}^d$ generated by $\mathbf{e}(\mathscr{C})$ over all simple directed cycles $\mathscr{C}$, and choose a $\mathbb Z$-basis
$
\mathbf{a}_1,\dots,\mathbf{a}_r\in\mathbb Z^d.
$
Since each $\mathbf{a}_k$ is an integer linear combination of the vectors $\mathbf{e}(\mathscr{C})$, we obtain
\[
\mathbf{a}_k=\sum_{\mathscr{C}} C_{\mathscr{C}}^{(k)}\mathbf{e}(\mathscr{C}),
\qquad C_{\mathscr{C}}^{(k)}\in\mathbb Z,
\]
which implies
\[
e^{i\langle \mathbf{a}_k,\hat\theta\rangle}
=
\prod_{\mathscr{C}}\Bigl(e^{i\langle \mathbf{e}(\mathscr{C}),\hat\theta\rangle}\Bigr)^{C_{\mathscr{C}}^{(k)}}.
\]
Because only finitely many simple cycles appear, and all coefficients depend only on the support of $Q$, \eqref{eq.u-cycle-transform} and \eqref{eq.u-cycle-bound} imply
$
|e^{\mathrm i\langle \mathbf{a}_k,\hat\theta\rangle}-1|\le C\kappa_Q^{-d}|(U-\lambda Q)\mathbf z|_\infty^{1/2}
$ for all $1\le k\le r$.
By the inequality
$
\dist(t,2\pi\mathbb Z)\le C\,|e^{\mathrm it}-1|
$
for all $t\in\R$,
we get
\[
\dist(\langle \mathbf{a}_k,\hat\theta\rangle,2\pi\mathbb Z)\le C\kappa_Q^{-d}|(U-\lambda Q)\mathbf z|_\infty^{1/2},\qquad k=1,\dots,r.
\]
Let $A\in\mathbb Z^{r\times d}$ be the matrix whose rows are $\mathbf{a}_1,\dots,\mathbf{a}_r$.
Then one can choose $\mathbf{n}=(n_1,\dots,n_r)\in\mathbb Z^r$ such that
$
\mathbf{b}:=A\hat\theta-2\pi \mathbf{n}
$
satisfies
\begin{equation}\label{eq.b-bound}
  |\mathbf{b}|_\infty\le C\kappa_Q^{-d}|(U-\lambda Q)\mathbf z|_\infty^{1/2}.
\end{equation}

Since the rows of $A$ are linearly independent, $A$ has full row rank, hence admits a right inverse
$
A^\dagger\in\mathbb R^{d\times r}
$
such that
$
AA^\dagger=I_r.
$
Set
$
\tilde{\theta}:=\hat\theta-A^\dagger\mathbf{b}.
$
Then, we have
\[
A\tilde{\theta}=A\hat\theta-AA^\dagger\mathbf{b}=A\hat\theta-\mathbf{b}=2\pi \mathbf{n}\in 2\pi\mathbb Z^r,
\]
which implies that $e^{\mathrm i\langle \mathbf{a}_k,\tilde{\theta}\rangle}=1$ for every basis $\mathbf{a}_k$. Therefore,
by the characterization of $\mathscr{R}_Q$ through the cycle constraints, it follows that
\[
V:=\diag(v_1,\dots,v_d)\in\mathscr{R}_Q,\qquad\text{where}~v_j:=e^{\mathrm i\tilde{\theta}_j}.
\]
Moreover, using that $\|A^\dagger\|_\infty$ only depends on the support of $Q$, it follows from \eqref{eq.b-bound} that
\[
|\tilde{\theta}-\hat\theta|_\infty\le \|A^\dagger\|_\infty |\mathbf{b}|_\infty\le C\kappa_Q^{-d}|(U-\lambda Q)\mathbf z|_\infty^{1/2},
\]
and thus
\[
\max_{1\le j\le d}|u_j- v_j|
=
|e^{\mathrm i\hat\theta}-e^{\mathrm i\tilde{\theta}}|_\infty
\le |\hat\theta-\tilde{\theta}|_\infty
\le  C\kappa_Q^{-d}|(U-\lambda Q)\mathbf z|_\infty^{1/2}.
\]
Then, we conclude that
\[
\dist(U,\mathscr{R}_Q)\le \max_j |u_j-v_j|
\le C\kappa_Q^{-d}|(U-\lambda Q)\mathbf z|_\infty^{1/2},
\]
which proves the desired result.
\end{proof}

\bigskip
Using Lemma \ref{lem:UminusQ-graph-resonance}, the following proposition gives the bulk stability in the symmetric case.

\begin{proposition}[Bulk stability for symmetric matrices]\label{prop.bulk-stability-symmetric}
Suppose \(S\in\mathbb R^{d\times d}\) is nonnegative, symmetric, irreducible and the solution $m(z)$ is uniformly bounded on $\Cp(I)$ where $I\subset[-\Sigma,\Sigma]$ is a closed interval.
Then, for all $z\in\Cp(I)$ with \(|z|\le 2\Sigma\), we have
\begin{equation}\label{eq.symmetric-general-bulk-stability}
    \bigl\|(I-m(z)^2S)^{-1}\bigr\|_2
    \lesssim
    \bigl(
        [\max_{1\le i\le d}|z-a_i|]\,
        \langle\im m(z)\rangle
    \bigr)^{-2}.
\end{equation}
Moreover, if the support graph of \(S\) is non-bipartite, then
\begin{equation}\label{eq.symmetric-nonbipartite-bulk-stability}
    \bigl\|(I-m(z)^2S)^{-1}\bigr\|_2
    \lesssim
    \langle\im m(z)\rangle^{-2}.
\end{equation}
\end{proposition}

\begin{proof}
We first list two bounds that will be used frequently in the following proof. It follows from \eqref{eq.m-bound} that, for all $z\in\Cp(I)$ with $|z|\le 2\Sigma$,
    \begin{equation}\label{eq.m-bulk-bound-symmetric}
      |m_{i}(z)|\sim 1\qquad i=1,\dots,d.
    \end{equation}
    Also, it follows from \eqref{eq.m-Im} that, for all $z\in \Cp(I)$ with $|z|\le 2\Sigma$,
    \begin{equation}\label{eq.im-m-comparability-symmetric}
      \im m_{i}(z)\sim\langle \im m(z)\rangle\qquad i=1,\dots,d.
    \end{equation}

    Then, it follows from \eqref{eq.I-m2S-transform}, \eqref{eq.m-bound} and direct calculation that
    \[
      \|(I-m^2S)^{-1}\|_2\le\frac{\max_i|m_i|}{\min_i|m_i|}\|R^{-1}\|_2\sim\|R^{-1}\|_2.
    \]
    Since $|\cdot|_2\sim|\cdot|_\infty$, it is sufficient to show that
    $
    |R\mathbf z|_\infty
    \gtrsim
    \bigl([\max_i |z-a_i|]\langle\im m\rangle\bigr)^2
    |\mathbf z|_\infty$
    in general, and
    $
    |R\mathbf z|_\infty
    \gtrsim
    \langle\im m\rangle^2|\mathbf z|_\infty
    $
    when the support graph of \(S\) is non-bipartite.

    We first normalize $F$ again. Define $ Q_F:= \lambda_F^{-1}D(r)^{-1}FD(r)$.
    Lemma \ref{lem:F-PF} implies that $r>0$. Combined with $|m|>0$, it follows that $Q_F\ge0$, $Q_F\mathbf{1}=\mathbf{1}$ and $Q_F$ has the same support as $S$.
    Since \(U\) and \(D(r)\) are diagonal, they commute, and therefore
    \begin{equation}\label{eq:B-conjugation}
    R
    =
    U-F
    =
    D(r)\,(U-\lambda_F Q_F)\,D(r)^{-1}.
    \end{equation}
    For any \(\mathbf z\in\C^d\), set
    $
    \mathbf w:=D(r)^{-1}\mathbf z.
    $
    Then by \eqref{eq:B-conjugation},
    \begin{equation}\label{eq:B-vs-UQ-1}
      |R\mathbf z|_\infty=|D(r)(U-\lambda_FQ_F)\mathbf w|_\infty\ge \min_i r_i\, |(U-\lambda_FQ_F)\mathbf w|_\infty.
    \end{equation}
    On the other hand, we have
    $
    |\mathbf w|_\infty=|D(r)^{-1}\mathbf z|_\infty\ge
    (\max_i r_i)^{-1}|\mathbf z|_\infty.
    $
    Since \eqref{eq.vector-comparable} implies that $\min_i r_i\sim\max_i r_i$ it remains to show that
    $
    |(U-\lambda_FQ_F)\mathbf w|_\infty\gtrsim[(\max_i |z-a_i|)\langle\im m\rangle]^{2}|\mathbf w|_\infty .
    $
    Lemma~\ref{lem:UminusQ-graph-resonance} implies 
    \begin{equation}\label{eq.s1-U-F-symmetry}
        |(U-\lambda_FQ_F)\mathbf w|_\infty
    \gtrsim\dist(U,\mathscr{R}_{Q_F})^2\,|\mathbf w|_\infty .
    \end{equation}
    Since $Q_F$ has the same support as $S$, we have
    $
    \mathscr{R}_{Q_F} = \mathscr{R}_S.
    $
    Hence, it is enough to analyze $\dist(U,\mathscr{R}_{S})$.

        Since \(S\) is symmetric, whenever \(i\to j\), we also have \(j\to i\).
        Thus the two directed edges form a directed \(2\)-cycle. By the definition of \(\mathscr R_S\), this gives, for every $V\in\mathscr{R}_S$, 
        \begin{equation}\label{eq.symmetry-vivj}
            \qquad v_iv_j=1,\qquad i\sim j.
        \end{equation}
        The irreducibility of $S$ implies the connectivity of the support undirected graph. Hence the values
        \(v_i\) alternate between two reciprocal values along paths.

    We first prove the general estimate \eqref{eq.symmetric-general-bulk-stability}.
    If the support graph of \(S\) is non-bipartite, then the stronger estimate
    \eqref{eq.symmetric-nonbipartite-bulk-stability}, proved below, immediately
    implies \eqref{eq.symmetric-general-bulk-stability}, since \(\mathbf a\) is fixed and
    \(|z|\le2\Sigma\). Hence, for the proof of
    \eqref{eq.symmetric-general-bulk-stability}, it remains to consider the case
    where the support graph of \(S\) is bipartite.

    Let
    $
        \{1,\dots,d\}=\mathbb V_+\sqcup\mathbb V_-
    $
    be the bipartition of the support graph. 
    Since the values \(v_i\) alternate between two reciprocal values along paths, it follows that there exists
    \(\gamma\in\mathbb S^1\) such that
    \begin{equation}\label{eq.bipartite-resonance-structure-symmetric}
        v_i=\gamma,\qquad i\in\mathbb V_+,
        \qquad
        v_j=\gamma^{-1},\qquad j\in\mathbb V_- .
    \end{equation}

    We assume that $\dist(U,\mathscr{R}_{S})\le c\langle\im m\rangle$, 
    where \(c>0\) is a sufficiently small constant depending only on model parameters. Indeed, if
    \(\dist(U,\mathscr{R}_{S})>c\langle\im m\rangle\), then it follows that
    \[
        s_{\rm min}(I-m^2S)
        \gtrsim
        \dist(U,\mathscr{R}_{S})^2
        \gtrsim
        \langle\im m\rangle^2
        \gtrsim
        \bigl(
            [\max_i|z-a_i|]\langle\im m\rangle
        \bigr)^2,
    \]
    where we apply $|z|\le 2\Sigma$ in the last inequality.
    
    Choose \(V_*\in\mathscr R_{S}\) such that
    $
       \max_j |u_j-v^*_j|\le 2\dist(U,\mathscr{R}_{S}) .
    $
    It follows from \eqref{eq.bipartite-resonance-structure-symmetric} that there are two values for $v^*_i$: $\gamma$ and $\gamma^{-1}$. For \(\gamma\in\mathbb S^1\), we can choose \(\theta_*\in[0,\pi)\) such that
    $
        \gamma=e^{-2\mathrm i\theta_*}.
    $
    The diagonal unitarity of \(U\) also allows us to write its entries in the form
    \[
        u_j=e^{-2\mathrm i\theta_j},
        \qquad 
        \theta_j:=\arg m_j\in(0,\pi).
    \]
    
    For \(i\in\mathbb V_+\), we assume the corresponding diagonal entry of \(V_*\) is
    \(v^*_i=\gamma=e^{-2\mathrm i\theta_*}\). Then using $|e^{-2\mathrm ia}-e^{-2\mathrm ib}|=2|\sin(a-b)|$ implies
    \[
        2|\sin(\theta_i-\theta_*)|=|e^{-2\mathrm i\theta_i}-e^{-2\mathrm i\theta_*}|
        =
        |u_i-v^*_i|
        \le \max_j |u_j-v^*_j|
        \le 2\dist(U,\mathscr R_S).
    \]
    Thus, we obtain
    $
        |\sin(\theta_i-\theta_*)|
        \lesssim \dist(U,\mathscr R_S).
    $
    Since we are working in the regime
    \(\dist(U,\mathscr R_S)\le c\langle\im m\rangle\) with \(c>0\) sufficiently small,
    the above estimate fixes the relevant branch of the phase modulo \(\pi\). Therefore,
    after choosing the representative $\gamma,\gamma^{-1}$ consistently, we have
    \begin{equation}\label{eq.phase-close-plus-symmetric}
        |\theta_i-\theta_*|
        \lesssim \dist(U,\mathscr{R}_{S}),
        \qquad i\in\mathbb V_+ .
    \end{equation}
    Similarly, we get
    \begin{equation}\label{eq.phase-close-minus-symmetric}
        |\theta_j-(\pi-\theta_*)|\lesssim \dist(U,\mathscr{R}_{S}),
        \qquad j\in\mathbb V_- .
    \end{equation}
    Since $\dist(U,\mathscr{R}_S)$ is small enough, it follows that
    \begin{equation}\label{eq.sin-theta-star-symmetric}
      \sin\theta_*
        =
        \sin\theta_i+O(\dist(U,\mathscr{R}_{S}))
        =
        \frac{\im m_i}{|m_i|}+O(\dist(U,\mathscr{R}_{S}))
        \sim
        \langle\im m\rangle,
    \end{equation}
    where we use $|m_i|\sim 1$ and $\im m_i\sim\langle\im m\rangle$ in the last comparison.

    We now use the Dyson equation. Fix \(i\in\mathbb V_+\). Since all neighbors of \(i\)
    belong to \(\mathbb V_-\), we have
    \[
        -\frac1{m_i}
        =
        z-a_i+\sum_{j\in\mathbb V_-}s_{ij}m_j .
    \]
    Multiplying by \(e^{\mathrm i\theta_*}\) and taking imaginary parts gives
    \begin{equation}\label{eq.plus-imaginary-part}
        \im\bigl(e^{\mathrm i\theta_*}(z-a_i)\bigr)
        =
        \im\Big(-\frac{e^{\mathrm i\theta_*}}{m_i}\Big)
        -
        \im\Bigg(
            \sum_{j\in\mathbb V_-}s_{ij}e^{\mathrm i\theta_*}m_j
        \Bigg).
    \end{equation}
    It follows from \eqref{eq.phase-close-plus-symmetric} and $|m_i|\sim 1$ that
    \smallskip
    $
        \left|
        \im\big(-e^{\mathrm i\theta_*}/m_i\big)
        \right|
        =
        |m_i|^{-1}|\sin(\theta_*-\theta_i)|
        \lesssim \dist(U,\mathscr{R}_{S}) .
    $
    Using \eqref{eq.phase-close-minus-symmetric} and $|m_j|\sim 1$ also yields
    $
        \left|
        \im(e^{\mathrm i\theta_*}m_j)
        \right|
        =
        |m_j|\,|\sin(\theta_*+\theta_j)|
        \lesssim \dist(U,\mathscr{R}_{S})
    $
    for all $j\in\mathbb{V}_-$.
    Since \(S\) is nonnegative and \(|m_j|\sim1\), the last term in
    \eqref{eq.plus-imaginary-part} is also \(O(\dist(U,\mathscr{R}_{S}))\). Hence we have
    \begin{equation}\label{eq.plus-z-ai-control}
        \left|
        \im\bigl(e^{\mathrm i\theta_*}(z-a_i)\bigr)
        \right|
        \lesssim \dist(U,\mathscr{R}_{S})
        \qquad i\in\mathbb V_+ .
    \end{equation}
    Similarly, for \(j\in\mathbb V_-\), multiplying the Dyson equation by \(e^{-\mathrm i\theta_*}\),
    and using that all neighbors of \(j\) belong to \(\mathbb V_+\), we obtain
    \begin{equation}\label{eq.minus-z-aj-control}
        \left|
        \im\bigl(e^{-\mathrm i\theta_*}(z-a_j)\bigr)
        \right|
        \lesssim \dist(U,\mathscr{R}_{S}),
        \qquad j\in\mathbb V_- .
    \end{equation}
    
    Since \(a_i\in\mathbb R\) and \(z=\tau+\mathrm i\eta\), we have
    \[
        \im\bigl(e^{\mathrm i\theta_*}(z-a_i)\bigr)
        =
        (\tau-a_i)\sin\theta_*+\eta\cos\theta_*,
        \qquad i\in\mathbb V_+,
    \]
    and
    \[
        \im\bigl(e^{-\mathrm i\theta_*}(z-a_j)\bigr)
        =
        -(\tau-a_j)\sin\theta_*+\eta\cos\theta_*,
        \qquad j\in\mathbb V_- .
    \]
    Therefore \eqref{eq.plus-z-ai-control} and \eqref{eq.minus-z-aj-control} imply
    $
        |\tau-a_i|\sin\theta_*
        \lesssim
        \dist(U,\mathscr{R}_{S})+\eta
    $
    for all $i=1,\dots,d$.
    Then the key estimate follows from \eqref{eq.sin-theta-star-symmetric}
    \begin{equation}\label{eq.delta-eta-controls-max-tau-ai}
        \dist(U,\mathscr{R}_{S})+\eta
        \gtrsim
        \bigl[\max_{1\le i\le d}|\tau-a_i|\bigr]\,
        \langle\im m\rangle .
    \end{equation}

    Combining Proposition \ref{pro.simple-bound-stability} with
    \eqref{eq.delta-eta-controls-max-tau-ai}, and using the boundedness of
    \(|z|\), \(a\), and \(\langle\im m\rangle\), we get
    \[
        s_{\rm min}(I-m^2S)
        \gtrsim
        \bigl(
            [\max_i|\tau-a_i|]\langle\im m\rangle
        \bigr)^2 .
    \]
    Moreover, using
    $
        s_{\rm min}(I-m^2S)
        \gtrsim
        \eta
        \gtrsim
        (\eta\langle\im m\rangle)^2 
    $
    and 
    $
        \max_i |z-a_i|
        \le
        \max_i|\tau-a_i|+\eta,
    $
    we conclude that
    $
        s_{\rm min}(I-m^2S)
        \gtrsim
        \bigl(
            [\max_i|z-a_i|]\langle\im m\rangle
        \bigr)^2 .
    $
    This proves \eqref{eq.symmetric-general-bulk-stability}.
    
    It remains to prove the stronger estimate \eqref{eq.symmetric-nonbipartite-bulk-stability}.
    If the support graph of \(S\) is non-bipartite, then it contains an undirected odd cycle. Since \(S\) is symmetric, this cycle can be oriented as a directed cycle:
        \[
            i_1\to i_2\to\dots\to i_{2\ell+1}\to i_1.
        \]
        Along this cycle, the relation \(v_{i_j} v_{i_{j+1}}=1\) implies that the product $\prod_{j=1}^{2\ell+1}v_{i_j}$
        is equal to one of the values \(v_{i_r}\). On the other hand, the resonance condition applied
        to this odd directed cycle gives
        $
            \prod_{j=1}^{2\ell+1}v_{i_j}=1.
        $
        Thus \(v_{i_r}=1\) for some vertex on the odd cycle. Using again \(v_i v_j=1\) and the
        connectivity of the underlying graph, we obtain
        $
            v_i=1
        $
        for $i=1,\dots,d$,
        which implies $\mathscr R_{S}=\{I\}.$
        Therefore it follows that
        $
        \dist(U,\mathscr{R}_{S})
        =
        \max_j |u_j-1|.
      $
    Since \(u_i=e^{-2\mathrm i\theta_i}\), we have
    \[
        |u_i-1|
        =
        |e^{-2\mathrm i\theta_i}-1|
        =
        2\sin\theta_i
        =
        2\frac{\im m_i}{|m_i|}.
    \]
    Using $|m_i|\sim 1$ and $\im m_i\sim\langle \im m\rangle$, we obtain
    $
        \dist(U,\mathscr{R}_{S})
        \gtrsim
        \langle\im m\rangle .
    $
    Putting this into \eqref{eq.s1-U-F-symmetry} gives
    $
        s_{\rm min}(I-m^2S)
        \gtrsim
        \langle\im m\rangle^2
    $, which concludes the proof.
     \end{proof}

    In the remaining part of this section, we consider the non-backtracking matrix case.
    We first introduce the lemma that gives a concrete description of the resonance set of the non-backtracking matrix when the base graph has minimum degree at least three. It shows that every resonance phase on the directed edges is generated by a vertex potential on the base graph.

    \begin{lemma}\label{lem.resonant-set-B}
    Let $\mathcal G=(\mathbb{V}_{\mathcal G},E_{\mathcal G})$ be an undirected connected base graph and assume that
    \[
    \deg_{\mathcal G}(x)\ge 3,\qquad x\in\mathbb V_{\mathcal G}.
    \]
    Let $B$ be the non-backtracking matrix of the base graph $\mathcal G$. Then
    \[
    \mathscr R_B
    =
    \left\{
    \operatorname{diag}\bigl(v_e\bigr)_{e\in\vec E_{\mathcal G}}
    :
    \exists(\omega_x)_{x\in \mathbb{V}_{\mathcal G}}\in (\mathbb S^1)^{\mathbb{V}_{\mathcal G}}
    \text{ such that }
    v_{xy}=\overline{\omega_x}\omega_y
    \right\}.
    \]
    \end{lemma}
    
    \begin{proof}
    Under the assumptions that \(\mathcal G\) is connected and
    \(\deg_{\mathcal G}(x)\ge 3\) for all \(x\), the non-backtracking graph on
    \(\vec E_{\mathcal G}\) is strongly connected, which implies \(B\) is irreducible.
    Let
    $
    V=\diag(v_e)_{e\in\vec E_{\mathcal G}}\in\mathscr R_B.
    $
    By Lemma \ref{lem.zero-set-of-edge-defect}, there exists $\gamma=(\gamma_{e})_{e\in\vec E_{\mathcal G}}\in (\mathbb S^1)^{\vec E_{\mathcal G}}$ such that
    $
    v_{xy}=\overline{\gamma_{xy}}\,\gamma_{yz}
    $
    for every allowed non-backtracking transition $(x,y)\to(y,z)$. 
    
    Fix $y\in \mathbb{V}_{\mathcal G}$. Since $\deg_{\mathcal G}(y)\ge 3$, for any two neighbors $z_1,z_2\sim y$, we may choose $x\sim y$ with
    $x\neq z_1,z_2$. Then both non-backtracking transitions $(x,y)\to(y,z_1)$ and $(x,y)\to(y,z_2)$ are allowed, hence
    \[
    \overline{\gamma_{xy}}\gamma_{yz_1}
    =
    v_{xy}
    =
    \overline{\gamma_{xy}}\gamma_{yz_2},
    \]
    so
    $
    \gamma_{yz_1}=\gamma_{yz_2}.
    $
    Therefore all outgoing edges from $y$ have the same value, and we denote this value by $\omega_y$.
    Then $\gamma_{xy}=\omega_x$ for every oriented edge $(x,y)$, and thus
    \[
    v_{xy}
    =
    \overline{\gamma_{xy}}\gamma_{yz}
    =
    \overline{\omega_x}\,\omega_y.
    \]
    This proves that every element of $\mathscr R_B$ has the stated form.
    
    Conversely, given any $\omega=(\omega_x)_{x\in\mathbb{V}_{\mathcal G}}\in (\mathbb{S}^1)^{\mathbb{V}_{\mathcal G}}$, we have
    $
    v_{xy}=\overline{\omega_x}\omega_y
    $
    by definition.
    Then for every directed cycle 
    \[
    \mathscr{C}:(x,y)=e_0\to e_1\to\dots\to e_\ell=(x,y),
    \]
    the product of the factors \(\overline{\omega_x}\omega_y\) along the cycle telescopes, and hence
    $
    \prod_{e\in\mathscr{C}}v_{e}=1.
    $
    Hence $V=\diag(v_e)_{e\in\vec E_{\mathcal G}}\in\mathscr R_B$. This completes the proof.
    \end{proof}

    Combining Lemmas~\ref{lem:UminusQ-graph-resonance} and~\ref{lem.resonant-set-B}, we obtain the following bulk stability estimate for the non-backtracking matrix case.

    \begin{proposition}[Bulk stability for non-backtracking matrix]\label{pro.bulk-stability-NB}
        Suppose that $S$ satisfies \textup{(NB)}, and the solution $m(z)$ is uniformly bounded on $\Cp(I)$ where $I\subset[-\Sigma,\Sigma]$ is a closed interval.
        Then, for every $z\in\Cp(I)$ with $|z|\le 2\Sigma$, we have
        \begin{equation}\label{eq.general-bulk-stability}
          \|\big(I-m(z)^2S\big)^{-1}\|_2\lesssim\big(\big[\max_y|z-a_y|\big]\langle\im m(z)\rangle\big)^{-2}.
        \end{equation}
        Additionally, if the base graph $\mathcal G$ is non-bipartite, we have
        \begin{equation}\label{eq.nonbipartite-bulk-stability}
           \|\big(I-m(z)^2S\big)^{-1}\|_2\lesssim\langle\im m(z)\rangle^{-2}.
        \end{equation}
    \end{proposition}

\begin{proof}
     First, we list two bounds that will be used frequently in the following proof. It follows from \eqref{eq.m-bound} that, for all $z\in\Cp(I)$ with $|z|\le 2\Sigma$, 
    \begin{equation}\label{eq.m-bulk-bound}
      |m_{xy}|\sim 1,\qquad (x,y)\in\vec E_{\mathcal G}.
    \end{equation}
    Also, it follows from \eqref{eq.m-Im} that, for all $z\in\Cp(I)$ with $|z|\le 2\Sigma$,
    \begin{equation}\label{eq.im-m-comparability}
       \im m_{xy}\sim\langle \im m\rangle,\qquad(x,y)\in\vec E_{\mathcal G}.
    \end{equation}
    To upper bound $\|(I-m^2S)^{-1}\|_2$, it is sufficient to give a lower bound for $s_{\rm min}(I-m^2S)$.
    The same argument as in Proposition \ref{prop.bulk-stability-symmetric} implies 
    \begin{equation}\label{eq.s1-U-F}
      s_{\rm min}(I-m^2S)\gtrsim \dist(U,\mathscr{R}_S)^2.
    \end{equation}
    Hence, it is enough to analyze $\dist(U,\mathscr{R}_S)$.

If
$
    \dist(U,\mathscr{R}_S)>c\langle\im m\rangle
$
for a sufficiently small constant \(c>0\), depending only on the model parameters,
then \eqref{eq.s1-U-F} gives
\begin{equation}\label{eq.big-im-case}
    s_{\rm min}(I-m^2S)
    \gtrsim
    \dist(U,\mathscr{R}_S)^2
    \gtrsim
    \langle\im m\rangle^2
    \gtrsim
    \bigl([\max_y |z-a_y|]\langle\im m\rangle\bigr)^2.
\end{equation}
Hence it remains to consider the case
$
    \dist(U,\mathscr{R}_S)\le c\langle\im m\rangle.
$

Choose \(V_*=\diag(v^*_e)_{e\in\vec E_{\mathcal G}}\in\mathscr R_S\) such that
$
    \max_{e\in\vec E_{\mathcal G}}|u_e-v^*_e|\le 2\dist(U,\mathscr{R}_S) .
$
Write
\[
    u_{xy}=e^{-2\mathrm i\theta_{xy}},
    \qquad
    \theta_{xy}:=\arg m_{xy}\in(0,\pi).
\]
By Lemma~\ref{lem.resonant-set-B}, every \(V_*\in\mathscr R_S\) satisfies
$
    v^*_{xy}v^*_{yx}=1
$
for all
$
 x\sim y.
$
For each oriented edge \((x,y)\), 
\smallskip
choose \(\varphi_{xy}\in[0,\pi]\) such that
$
    v^*_{xy}=e^{-2\mathrm i\varphi_{xy}} .
$
Since
\[
    |u_{xy}-v^*_{xy}|
    =
    |e^{-2\mathrm i\theta_{xy}}-e^{-2\mathrm i\varphi_{xy}}|
    \le 2\dist(U,\mathscr{R}_S),
\]
we first obtain
$
    \dist\bigl(\theta_{xy}-\varphi_{xy},\pi\mathbb Z\bigr)
    \lesssim \dist(U,\mathscr{R}_S) .
$
On the other hand, it follows that
\[
    \sin\theta_{xy}
    =
    \frac{\im m_{xy}}{|m_{xy}|}
    \sim
    \langle\im m\rangle .
\]
Since \(\dist(U,\mathscr{R}_S)\le c\langle\im m\rangle\), choosing \(c>0\) sufficiently small ensures that the wrong branch modulo \(\pi\) is excluded. Thus, after choosing the representative \(\varphi_{xy}\) consistently, we have
$
    |\theta_{xy}-\varphi_{xy}|
    \lesssim \dist(U,\mathscr{R}_S) .
$

Moreover, from \(v^*_{xy}v^*_{yx}=1\), we may choose the representatives so that
$
    \varphi_{yx}=\pi-\varphi_{xy}.
$
Consequently, we have
\begin{equation}\label{eq.nb-reversal-phase-from-resonance}
    \big|
    \theta_{xy}+\theta_{yx}-\pi
    \big|
    \lesssim
    \dist(U,\mathscr{R}_S),
    \qquad x\sim y.
\end{equation}

For every oriented edge \((x,y)\), the Dyson equation reads
\[
    -\frac1{m_{xy}}
    =
    z-a_y+\sum_{\substack{k\sim y\\ k\neq x}}m_{yk}.
\]
Adding the missing term \(m_{yx}\), we obtain
\begin{equation}\label{def.hy}
    -\frac1{m_{xy}}+m_{yx}
    =
    z-a_y+\sum_{k\sim y}m_{yk}
    =:h_y(z).
\end{equation}
Define 
$
    \phi_y:=\arg h_y\in(0,\pi),
$
where \(\phi_y\in(0,\pi)\) follows from
\[
    \im h_y
    =
    \eta+\sum_{k\sim y}\im m_{yk}
    >0.
\]
Multiplying \eqref{def.hy} by \(m_{xy}\), we get
$
    h_y m_{xy}
    =
    m_{xy}m_{yx}-1.
$
By \eqref{eq.nb-reversal-phase-from-resonance}, the number
\(m_{xy}m_{yx}\) lies within \(O(\dist(U,\mathscr{R}_S))\) of the negative real axis.
Together with \(|m_{xy}|\sim |m_{yx}|\sim1\), this implies
$
    |h_y|\sim 1
$
and
$
    \big|\theta_{xy}-(\pi-\phi_y)\big|
    \lesssim
    \dist(U,\mathscr{R}_S).
$
Combining this with \eqref{eq.nb-reversal-phase-from-resonance}, we obtain
\begin{equation}\label{eq.nb-outgoing-phase-close}
    \big|\theta_{yk}-\phi_y\big|
    \lesssim
    \dist(U,\mathscr{R}_S),
    \qquad (y,k)\in\vec E_{\mathcal G}.
\end{equation}

Now use the identity
\[
    h_y=z-a_y+\sum_{k\sim y}m_{yk}.
\]
Since \(e^{-\mathrm i\phi_y}h_y\) is positive real, its imaginary part vanishes:
\[
    0
    =
    \im\bigl(e^{-\mathrm i\phi_y}h_y\bigr)
    =
    \im\bigl(e^{-\mathrm i\phi_y}(z-a_y)\bigr)
    +
    \im\Bigl(\sum_{k\sim y}e^{-\mathrm i\phi_y}m_{yk}\Bigr).
\]
Writing \(z=\tau+\mathrm i\eta\), the first term equals
    \[
        \im\big(e^{-\mathrm i\phi_y}(z-a_y)\big)
        =
        \im\big(e^{-\mathrm i\phi_y}(\tau-a_y+i\eta)\big)
        =
        -(\tau-a_y)\sin\phi_y+\eta\cos\phi_y.
    \]
By \eqref{eq.nb-outgoing-phase-close} and \(|m_{yv}|\sim1\), the second term is \(O(\dist(U,\mathscr{R}_S))\). Hence
$
    |\tau-a_y|\sin\phi_y
    \lesssim
    \eta+\dist(U,\mathscr{R}_S).
$
On the other hand, it follows from $|h_y|\sim 1$ that 
\[
    \sin\phi_y
    =
    \frac{\im h_y}{|h_y|}
    \gtrsim
    \im h_y
    \sim
    \langle\im m\rangle .
\]
Therefore, we have
\begin{equation}\label{eq.nb-delta-eta-controls-tau-ay}
    \dist(U,\mathscr{R}_S)+\eta
    \gtrsim
    |\tau-a_y|\langle\im m\rangle,
    \qquad y\in \mathbb{V}_{\mathcal G}.
\end{equation}

Combining \eqref{eq.s1-U-F}, Proposition~\ref{pro.simple-bound-stability}, and
\eqref{eq.nb-delta-eta-controls-tau-ay}, we get
\[
    s_{\rm min}(I-m^2S)
    \gtrsim
    \bigl(|\tau-a_y|\langle\im m\rangle\bigr)^2
    +
    \bigl(\eta\langle\im m\rangle\bigr)^2 .
\]
Since
$
    |z-a_y|\le |\tau-a_y|+\eta,
$
it follows that
$
    s_{\rm min}(I-m^2S)
    \gtrsim
    \bigl(|z-a_y|\langle\im m\rangle\bigr)^2.
$
Taking the maximum over \(y\in \mathbb{V}_{\mathcal G}\), we obtain
\[
    s_{\rm min}(I-m^2S)
    \gtrsim
    \bigl([\max_y|z-a_y|]\langle\im m\rangle\bigr)^2.
\]
This proves \eqref{eq.general-bulk-stability}.

It remains to prove the improved estimate when \(\mathcal G\) is non-bipartite. Using \eqref{eq.big-im-case} once more, it remains to consider the case
$
    \dist(U,\mathscr R_S)\le c\langle \im m\rangle,
$
in which the phase relations derived above are valid.

    Since \(\mathcal G\) is non-bipartite, it contains an undirected odd cycle.
    We orient it cyclically and write
    \[
        \mathscr C_{\rm odd}:
        i_0\to i_1\to\cdots\to i_{\ell_{\rm odd}}=i_0,
    \]
    where \(\ell_{\rm odd}\) is odd.
Set \(e_j:=(i_j,i_{j+1})\). Since \(V_*\in\mathscr R_S\), the resonance condition gives
$
    \prod_{j=0}^{\ell_{\rm odd}-1} v^*_{e_j}=1.
$
Using \(\max_{e\in\vec E_{\mathcal G}}|u_e-v^*_e|\le 2\dist(U,\mathscr R_S)\), we obtain
\[
    \dist\Bigl(\sum_{j=0}^{\ell_{\rm odd}-1}\theta_{e_j},\,\pi\mathbb Z\Bigr)
    \lesssim
    \dist(U,\mathscr R_S).
\]
On the other hand, by \eqref{eq.nb-outgoing-phase-close} and
\eqref{eq.nb-reversal-phase-from-resonance}, consecutive phases along the odd
cycle alternate between complementary values. Pairing the edges
\(e_1,e_2\), \(e_3,e_4\), \(\dots\), \(e_{\ell_{\rm odd}-2},e_{\ell_{\rm odd}-1}\), we get
\[
    \sum_{j=0}^{\ell_{\rm odd}-1}\theta_{e_j}
    =
    k\pi+\theta_{e_0}+O(\dist(U,\mathscr R_S))
\]
for some \(k\in\mathbb Z\). Hence
$
    \dist(\theta_{e_0},\pi\mathbb Z)\lesssim\dist(U,\mathscr R_S).
$
Using $|m_{e_0}|\sim 1$ yields
\[
    \dist(\theta_{e_0},\pi\mathbb Z)
    \sim
    \sin\theta_{e_0}
    =
    \frac{\im m_{e_0}}{|m_{e_0}|}
    \sim
    \langle\im m\rangle.
\]
Therefore
$
    \dist(U,\mathscr R_S)\gtrsim\langle\im m\rangle.
$
Putting this into \eqref{eq.s1-U-F} gives
$
    s_{\rm min}(I-m^2S)\gtrsim \langle\im m\rangle^2,
$
which proves the improved estimate.
\end{proof}

\bigskip

We now isolate the only remaining bulk instability mechanism. 
The preceding stability estimates show that, in both the symmetric and the non-backtracking settings, stability can degenerate in the bulk only when the relevant graph is bipartite and the shift vector is constant, \(\mathbf a\in\mathbb R\mathbf 1\). After translating the spectral parameter, this reduces to the special point \(z=0\) with \(\mathbf a=0\). 

The next proposition studies this exceptional point in detail and proves that the degeneracy is linear: the stability operator has smallest singular value of order \(|z|\), and hence its inverse has norm of order \(|z|^{-1}\).

\begin{proposition}[Bipartite instability at the origin]\label{prop.bipartite-origin-instability}
    Suppose that $S$ is a nonnegative and irreducible matrix and that the solution $m(z)$ is uniformly bounded on $N_\epsilon(0)$, where $\epsilon>0$ depends only on model parameters. Assume moreover that $\mathbf{a}=0$, and that either \textup{(Sym)} holds and the support graph of $S$ is bipartite, or \textup{(NB)} holds and the base graph $\mathcal G$ is bipartite.
    Then, we have
    \[
        \big\|(I-m(z)^2S)^{-1}\big\|_2\sim |z|^{-1},\qquad 0<|z|\le\epsilon.
    \]
\end{proposition}

\begin{proof}
    We first explain the common algebraic structure in the two cases.
    In the symmetric bipartite case, let
    $
        \{1,\dots,d\}=\mathbb V_+\sqcup \mathbb{V}_-
    $
    be the bipartition of the support graph of \(S\), and define
    \[
        T_i=
        \begin{cases}
            1, & i\in\mathbb V_+,\\
            -1, & i\in\mathbb V_-.
        \end{cases}
    \]
    In the non-backtracking case, the index set is
    \(\vec E_{\mathcal G}\). Let 
    $
        \mathbb{V}_{\mathcal G}=\mathbb V_+\sqcup \mathbb V_-
    $
    be the bipartition of the base graph $\mathcal G$, and define
    \[
        T_{uv}
        =
        \begin{cases}
            1, & v\in\mathbb V_+,\\
            -1, & v\in\mathbb V_- .
        \end{cases}
    \]
    In both cases, \(T:=\diag(T_i)\) satisfies
    \begin{equation}\label{eq.period-two-structure}
        T^2=I,
        \qquad
        TST=-S .
    \end{equation}
    Indeed, in the symmetric case every edge of the support graph connects
    the two bipartition classes, while in the non-backtracking case every
    admissible step
    $
        (u,v)\to (v,w)
    $
    changes the bipartition class of the terminal vertex.

    Since $m(z)$ is uniformly bounded on $N_\epsilon(0)$, $m(0)$ is well defined and finite. 
    The same argument as in \eqref{eq.invariant-transformation} shows that \(m(0)\) is purely imaginary; write $m(0)=\mathrm iq$ with $q>0$.
    Evaluating the Dyson equation at \(z=0\) and using \(m(0)=\mathrm i q\), we obtain
    \[
        \frac1q = Sq .
    \]
    Hence
    $
        q^2S q=q.
    $
    Therefore \(q^2S\) is an irreducible nonnegative matrix with Perron eigenvalue \(1\) and right Perron vector \(q\).

    Since \(TST=-S\), we have $ST=-TS$. Using the commutativity of \(T\) and \(D(q^2)\), we obtain
    \[
        q^2 S T = -q^2TS = -Tq^2S .
    \]   
    Acting on $q$ and using $q^2Sq=q$ yields
    \[
        q^2S(Tq)=-Tq .
    \]
    Lemma \ref{lem.PF} implies that the eigenvalue \(-1\) of \(q^2S\) is simple. Consequently, using $m(0)=iq$ yields
    $
        I-m(0)^2S
        =
        I+q^2S
    $
    has a one-dimensional kernel spanned by \(Tq\), and the rest of its
    spectrum is separated from zero.

    Let \(l_0>0\) be the left Perron vector of \(q^2S\):
    $
        l_0^\top D(q^2)S=l_0^\top .
    $
    Then \(Tl_0\) is a left eigenvector of \(q^2S\) corresponding to the
    eigenvalue \(-1\), and hence \(Tl_0\) spans the left kernel of
    \(I+q^2S\).

    We now show that the zero eigenvalue at \(z=0\) moves away from zero linearly as \(z\) varies.
    Set
    $
        b(z):=T m(z).
    $
    Using \(ST=-TS\) and the Dyson equation, we obtain 
    \[
    \begin{aligned}
        S b(z)
        =
        STm(z)
        =
        -TSm(z)                                                     
        =
        -T\Big(-\frac1{m(z)}-z\Big)
        =
        T\Big(\frac1{m(z)}+z\Big).
    \end{aligned}
    \]
    Multiplying by \(D(m(z)^2)\) and using $D(m^2)T=TD(m^2)$, we obtain
    \[
        m(z)^2 S b(z)
        =
        T\big(m(z)+z\,m(z)^2\big).
    \]
    Therefore, it follows that
    \begin{equation}\label{eq.bad-vector-linear-opening}
    \begin{aligned}
        \big(I-m(z)^2S\big)b(z)
        =
        Tm(z)-T\big(m(z)+z\,m(z)^2\big) 
        =
        -z\,T m(z)^2 .
    \end{aligned}
    \end{equation}
    Since \(m(z)\to iq\) with \(q>0\), we have \(|b(z)|_2\sim 1\) and \(|m_i(z)|\sim1\) when $z\in N_\epsilon(0)$. This gives 
    \[
      s_{\rm min}(I-m(z)^2S)
      \le
      \frac{|(I-m(z)^2S)b(z)|_2}{|b(z)|_2}
      =\frac{|zTm(z)^2|_2}{|b(z)|_2}
      \lesssim |z|.
    \]

    It remains to prove the matching lower bound. Choose a fixed complementary
    subspace \(X\) such that
    \begin{equation}\label{eq.Cd-decomposition1}
      \mathbb C^d=\operatorname{span}\{Tq\}\oplus X .
    \end{equation}
    Since
    $
        \ker(I+q^2S)=\operatorname{span}\{Tq\},
    $
    the operator \(I+q^2S\) is injective on \(X\). Hence, by finite
    dimensionality, there exists \(c>0\) depending only on model parameters such that
    \[
        |(I+q^2S)x|_2\ge c|x|_2,
        \qquad x\in X .
    \]
    Since \(I-m(z)^2S\to I+q^2S\) as \(z\to0\), we also have, for \(z\in N_\epsilon(0)\),
    \begin{equation}\label{eq.perturbed-lower-bound}
        |(I-m(z)^2S)x|_2\ge \frac c2|x|_2,\qquad x\in X .
    \end{equation}

    Since \(b(z)=Tm(z)\to iTq\), the decomposition above is stable under this
    small perturbation. Therefore, for \(z\in N_\epsilon(0)\), we can also decompose $\C^d$:
    \begin{equation}\label{eq.Cd-decomposition2}
      \mathbb C^d=\operatorname{span}\{b(z)\}\oplus X .
    \end{equation}
    Thus every \(\mathbf{w}\in\mathbb C^d\) with $|\mathbf{w}|_2=1$ can be written uniquely as
    \begin{equation}\label{eq.z-decomposition}
      \mathbf{w}=\theta b(z)+x,\qquad x\in X.
    \end{equation}
    Since the decomposition \eqref{eq.Cd-decomposition2} is a small perturbation of the fixed decomposition \eqref{eq.Cd-decomposition1}, the corresponding coordinate map
    \[
        (\theta,x)\mapsto \theta b(z)+x,\qquad x\in X,
    \]
    has a uniformly bounded inverse for \(z\in N_\epsilon(0)\). Hence, the decomposition \eqref{eq.z-decomposition} is uniformly stable:
    $
        |\theta|+|x|_2\sim |\mathbf{w}|_2=1 .
    $

    We distinguish two cases. First suppose that
    $
        |x|_2\ge C|z||\theta|,
    $
    where \(C>0\) is a sufficiently large constant depending only on model parameters. Using
    \eqref{eq.perturbed-lower-bound} and \eqref{eq.bad-vector-linear-opening}, we get
    \[
    \begin{aligned}
       \big|\big(I-m(z)^2S\big)\mathbf{w}\big|_2
        &=
        \big|
            \theta\big(I-m(z)^2S\big)b(z)
            +
            \big(I-m(z)^2S\big)x
        \big|_2                                    \\
        &\ge
        \big|\big(I-m(z)^2S\big)x\big|_2
        -
        |\theta|
        \big|\big(I-m(z)^2S\big)b(z)\big|_2       \\
        &\ge
        \frac c2|x|_2-C_1|z||\theta|            
        \gtrsim
        |x|_2.
    \end{aligned}
    \]
    Here the last step follows by choosing \(C\) sufficiently large. Moreover,
    since \(|\theta|+|x|_2\sim1\), the condition
    \(|x|_2\ge C|z||\theta|\) implies
    $
        |x|_2\gtrsim |z|.
    $
    Thus, in this case,
    $
        \big|\big(I-m(z)^2S\big)\mathbf{w}\big|_2\gtrsim |z|.
    $

    \smallskip
    It remains to consider the complementary case
    $
        |x|_2\le C|z||\theta|.
    $
    Pairing with the left kernel vector \(Tl_0\) of \(I+q^2S\), we use
    $
        (Tl_0)^\top(I+q^2S)=0
    $
    to obtain
    \[
    \begin{aligned}
        (Tl_0)^\top\big(I-m(z)^2S\big)x
        =
        (Tl_0)^\top
        \Big[
            \big(I-m(z)^2S\big)-\big(I+q^2S\big)
        \Big]x                                      
        =
        o(1)|x|_2                               
        =
        o(|z||\theta|).
    \end{aligned}
    \]
    On the other hand,  \eqref{eq.bad-vector-linear-opening} implies 
    \[
    \begin{aligned}
        (Tl_0)^\top\big(I-m(z)^2S\big)b(z)
        =
        -z\,l_0^\top m(z)^2 .
    \end{aligned}
    \]
    Since \(m(z)\to iq\), we have
    $
        m(z)^2=-q^2+o(1),
    $
    and therefore
    \[
        (Tl_0)^\top\big(I-m(z)^2S\big)b(z)
        =
        z\,l_0^\top q^2+o(|z|).
    \]
    It follows from \(l_0>0\) and \(q>0\) that
    $
        l_0^\top q^2>0.
    $
    This gives
    $
        \left|
        (Tl_0)^\top\big(I-m(z)^2S\big)b(z)
        \right|_2
        \gtrsim |z|.
    $
    Consequently, we have
    \[
    \begin{aligned}
        \left|
        (Tl_0)^\top\big(I-m(z)^2S\big)\mathbf{w}
        \right|
        \ge
        |\theta|
        \left|
        (Tl_0)^\top\big(I-m(z)^2S\big)b(z)
        \right|                                      
         -
        \left|
        (Tl_0)^\top\big(I-m(z)^2S\big)x
        \right|                                    
        \gtrsim
        |z||\theta|.
    \end{aligned}
    \]
    In the present case, \(|x|_2\le C|z||\theta|\) together with
    $
        |\theta|+|x|_2\sim1
    $
    implies \(|\theta|\gtrsim1\). Therefore, we have
    \[
        \big|
        (Tl_0)^\top\big(I-m(z)^2S\big)\mathbf{w}
        \big|
        \gtrsim |z|.
    \]
    Since \(Tl_0\) is fixed, this implies
    $
        \big|\big(I-m(z)^2S\big)\mathbf{w}\big|_2
        \gtrsim |z|.
    $

    Combining the two cases, we have shown that for every \(|\mathbf{w}|_2=1\),
    \[
        \big|\big(I-m(z)^2S\big)\mathbf{w}\big|_2\gtrsim |z|,\qquad z\in N_\epsilon(0).
    \]
    Therefore 
    $
        s_{\rm min}\big(I-m(z)^2S\big)\gtrsim |z|.
    $
    Together with the upper bound obtained by testing on \(b(z)=Tm(z)\), we get
    $
        s_{\rm min}\big(I-m(z)^2S\big)\sim |z|.
    $
    Equivalently, we have
    $
        \big\|\big(I-m(z)^2S\big)^{-1}\big\|_2\sim |z|^{-1}
    $
    for all $0<|z|\le\epsilon$.
\end{proof}

\section{Proof of Theorem \ref{thm.stability}}

We first observe that \(m(z)\) is uniformly bounded in all regimes considered in Theorem \ref{thm.stability}.
Away from the real axis this follows immediately from the rough bound
\eqref{eq.rough-boundedness}. Thus it suffices to consider \(\eta\ll1\). Near regular
edges and cusps, the boundedness is part of the local analysis used at the beginning of
the proof of Lemma~\ref{lem:B-bad-direction-nonsym-local}. In the strict bulk, it follows
from the definition of the regular support \eqref{def.Sj} and from the fact that
\(E\) stays at distance at least \(\epsilon\) from \(\partial\mathfrak S\). In the
off-support regime, the Stieltjes transform representation and
\(\dist(E,\overline{\mathfrak S}\cup\mathcal{A})\ge\epsilon\) imply 
\[
    |m_i(z)|
    \le
    \epsilon^{-1}\int_{\mathbb R}\tilde v_i(d\tau)\lesssim 1.
\]
Hence by Theorem \ref{thm.bound}, we have $|m_i(z)|\sim (1+|z|)^{-1}$ uniformly in all four cases.

The above uniform boundedness directly implies the trivial lower stability bound in all four cases: 
\begin{equation}\label{eq.stability-lower-bound}
    \bigl\|(I-m(z)^2S)^{-1}\bigr\|_2
    \ge
    \frac{1}{\|I-m(z)^2S\|_2}
    \gtrsim 1 .
\end{equation}
Moreover, uniform boundedness and Proposition~\ref{pro.simple-bound-stability} give
$
    \|(I-m(z)^2S)^{-1}\|_2
    \lesssim 1+\eta^{-1}
    \lesssim 1 
$
for $\eta>\eta_0$.
Thus we only need to treat \(z=E+i\eta\) with \(0\le \eta\le \eta_0\), where \(\eta_0>0\) is a sufficiently small constant depending only on model parameters.
\medskip

\textit{Case 1: Off-support Stability.}
We first prove the estimate on the real axis. 
It therefore suffices to consider the regime $|\tau|\le C$, where $C>0$ is chosen sufficiently large depending only on the model parameters. Indeed, for $|\tau|>C$, it follows from Theorem \ref{thm.bound} that
$
\|m(\tau)^2 S\|_2\le 1/2
$
by taking $C$ large enough. Hence the large-$|z|$ inverse bound \eqref{eq.large-regime-inverse-bound} directly yields \eqref{eq.off-support-stability}.

Fix \(\tau\in\mathfrak O\) with $|\tau|\le C$. Since \(\tau\notin\overline{\mathfrak S}\cup\mathcal{A}\), the boundary value \(m(\tau)\) is real-valued and finite. Moreover, since $\tau$ stays a fixed positive distance $\epsilon\sim 1$ away from the support, Theorem \ref{thm.bound} implies $|m(\tau)|\sim 1$. In the remainder of the proof, unless otherwise specified, all quantities without an explicit argument are evaluated at \(\tau\).

As $|m|^2=m^2$, we have
$
    D^{-1}(|m|)\bigl(I-m^2S\bigr)D(|m|)
    =
    I-F.
$
It suffices to prove a uniform bound on \((I-F)^{-1}\).
Differentiating the Dyson equation on the real axis gives
\[
    \bigl(I-m^2S\bigr)m'=m^2 .
\]
As each \(m_i\) is a Stieltjes transform (see \eqref{def.v}) and \(\tau\) lies outside the support, it follows that
\[
    m_i'(\tau)
    =
    \int_{\mathbb R}\frac{\tilde v_i(d\omega)}{(\omega-\tau)^2}
    >0 .
\]
By \(\dist(\tau,\overline{\mathfrak S}\cup\mathcal{A})\ge \epsilon\), we have
$
    m_i'(\tau)
    \le
    \epsilon^{-2}\,\widetilde v_i(\mathbb R)
    \lesssim1 .
$
Moreover, since the measure $\widetilde{v}_i$ is supported in the compact interval \([-\Sigma,\Sigma]\) and
\(\widetilde v_i(\mathbb R)\sim1\), we also have
\[
    m_i'(\tau)
    \ge
    \frac{\widetilde v_i(\mathbb R)}{(|\tau|+\Sigma)^2}
    \gtrsim
    \frac1{(1+|\tau|)^2}.
\]
Combining $|\tau|\le C$ gives $m_i'(\tau)\sim 1$.
Consequently, using $|m_i(\tau)|\sim 1$, we define
\[
    u(\tau):=D(|m(\tau)|)^{-1}m'(\tau).
\]
Then $u(\tau)$ is componentwise positive and satisfies
$
 u(\tau)\sim 1.
$
Applying the above similarity transform to the differentiated equation yields
$
    (I-F)u=|m| .
$
Therefore we have
$
    Fu=u-|m|<u
$
componentwise. 
Since \(u_i(\tau)\sim|m_i(\tau)|\sim1\) uniformly for \(\tau\in\mathfrak O\cap[-C,C]\), there exists \(c_*>0\) such that
\[
    Fu=u-|m|\le (1-c_*)u
\]
componentwise.

Define the weighted norm
\[
    |w|_{u(\tau)}:=\max_i\frac{|w_i|}{u_i(\tau)}.
\]
Then the non-negativity of \(F\) implies
\[
    |Fw|
    \le F|w|
    \le |w|_{u}Fu
    \le (1-c_*)|w|_u u,
\]
which yields
$
    |Fw|_{u}
    \le (1-c_*)|w|_{u}
$.
Thus it follows that
\[
    \|(I-F)^{-1}\|_{u}
    \le \sum_{k\ge0}\|F\|^k_{u}
    \le \frac1{c_*}\sim 1.
\]
Since \(u_i(\tau)\sim 1\), the norms \(|\cdot|_{u(\tau)}\) and \(|\cdot|_2\) are
uniformly equivalent. Hence it follows that
$
    \|(I-F(\tau))^{-1}\|_2\lesssim1 .
$
Using again the similarity relation between \(I-m^2S\) and \(I-F\) together with $|m_i|\sim 1$,
we obtain
\[
    \|(I-m(\tau)^2S)^{-1}\|_2\lesssim1,
    \qquad \tau\in\mathfrak O .
\]

It remains to pass from the real axis to \(z=\tau+\mathrm i\eta\) with small
\(\eta\). Note that all boundary values $m(\tau)$ appearing here are understood as the continuous extensions obtained by taking the limit $\eta\downarrow 0$. Since \(\mathfrak O\) has positive distance $\epsilon$ from the support, we have
\[
    \sup_{\tau\in\mathfrak O}
    \bigl\|
        (I-m(z)^2S)
        -
        (I-m(\tau)^2S)
    \bigr\|_2
    \longrightarrow 0,
    \qquad \eta\downarrow0 .
\]
Shrinking \(\eta_0\) if necessary, we get
\[
    \sup_{\tau\in\mathfrak O}
    \bigl\|
        (I-m(\tau)^2S)^{-1}
        \bigl[
            (I-m(z)^2S)
            -
            (I-m(\tau)^2S)
        \bigr]
    \bigr\|_2
    \le \frac12
\]
for all \(0\le\eta\le\eta_0\). Hence the Neumann expansion gives
\[
    \|(I-m(z)^2S)^{-1}\|_2
    \le
    2\|(I-m(\tau)^2S)^{-1}\|_2
    \lesssim1 .
\]
This proves \eqref{eq.off-support-stability}.

\medskip

\textit{Case 2: Edge Stability.} Fix a regular edge $\tau\in\partial\mathfrak{S}$. Unless otherwise stated, all statements in this case are restricted to $z\in N_\epsilon(\tau)$.
By \eqref{eq.I-m2S-transform} and $|m_i|\sim 1$, it is sufficient to bound $\|R^{-1}\|_2$.

Next we analyze $\langle\im m\rangle$. 
Near a regular edge \(\tau\), the density has square-root behavior on the
support side. More precisely, assuming \(\tau+ x\) lies in the self-consistent support for \(x>0\), Theorem \ref{thm:singularities} gives
\[
    v_i(\tau+ x)\sim \sqrt{x},
    \qquad 0<x<\epsilon .
\]
It follows from the Stieltjes representation and
\eqref{eq.measure-decomposition} that
\[
    \im m_i(z)
    =
    \int_{\mathbb R}
    \frac{\eta v_i(\omega)\,d\omega}
    {(\omega-E)^2+\eta^2}
    +
    \sum_{k=1}^{N_i}
    \frac{\eta\Delta_{i,k}}
    {(\tau_{i,k}-E)^2+\eta^2}.
\]
Since $\tau$ is a regular edge and the set of atoms is finite,
we may choose $\epsilon>0$ sufficiently small such that
$|\tau_{i,k}-E|\gtrsim 1$ for all $z\in N_\epsilon(\tau)$.
Consequently, the atomic contribution is $O(\eta)$,
which is negligible in the edge scaling regime
$\kappa+\eta\ll 1$. Hence,
\[
    \im m_i(z)
    =
    \int_{\mathbb R}
    \frac{\eta v_i(\omega)\,d\omega}
    {(\omega-E)^2+\eta^2}
    +O(\eta).
\]

We first consider the support side, namely \(E=\tau+\kappa\). The main
contribution comes from the edge neighborhood. After the change of variables $\omega-\tau=x$ and a standard estimate on the Poisson kernel, we obtain
\[
    \im m_i(E+i\eta)
    \sim
    \int_0^\epsilon
    \frac{\eta \sqrt{x}\,dx}{(x-\kappa)^2+\eta^2}+O(\eta)\sim \sqrt{\kappa+\eta},
\]
where the $O(\eta)$ term accounts for the contribution of
the remaining spectral measure, including atoms, and is
negligible for sufficiently small $\kappa+\eta$.
This gives
$
    \langle \im m(z)\rangle
    \sim
    \sqrt{\kappa+\eta}
$
on the support side. As $\tau$ is a regular edge, it follows that $\sigma\sim 1$ under the assumption (Sym) or (NB). Therefore, in the expansion \eqref{eq:beta-expansion-local}, the first term satisfies $\mu\eta/\delta\sim \eta/\sqrt{\kappa+\eta}$, while the second term satisfies $|2\mathrm i\sigma\delta|\sim\sqrt{\kappa+\eta}\ge\eta/\sqrt{\kappa+\eta}$. Since all the remaining terms are of lower order, we obtain $|\beta|\sim\sqrt{\kappa+\eta}$. Combining this estimate with \eqref{eq.Rinverse-bound} yields \eqref{eq.edge-stability} on the support side.

On the gap side, \(E=\tau-\kappa\). Again using
\(\omega=\tau+ x\) and an elementary Poisson kernel estimate, the local contribution is now
\[
    \im m_i(E+\ii\eta)
    \sim
    \int_0^\epsilon
    \frac{\eta \sqrt{x}\,dx}{(x+\kappa)^2+\eta^2}+O(\eta)\sim \frac{\eta}{\sqrt{\kappa+\eta}},
\]
which implies
$
    \langle \im m(z)\rangle
    \sim
    \eta/\sqrt{\kappa+\eta}
$
on the gap side. 
Therefore, in the expansion \eqref{eq:beta-expansion-local}, 
the first term $\mu\eta/\delta \sim \sqrt{\kappa+\eta}$, while the second term $|2\mathrm i\sigma\delta|\sim \eta/\sqrt{\kappa+\eta}\le\sqrt{\kappa+\eta}$.
Since all the remaining terms are of lower order, we obtain $|\beta|\sim\sqrt{\kappa+\eta}$.
Combining this estimate with \eqref{eq.Rinverse-bound} yields \eqref{eq.edge-stability} on the gap side.

\medskip
\textit{Case 3: Cusp Stability.} Fix a regular cusp $\tau\in\partial\mathfrak{S}$. Unless otherwise stated, all statements in this case are restricted to $z\in N_\epsilon(\tau)$.
By \eqref{eq.I-m2S-transform} and $|m_i|\sim 1$, it is sufficient to bound $\|R^{-1}\|_2$. As $\tau$ is a regular cusp, it follows that $\psi\sim 1$ under the assumption (Sym) or (NB).  
Then, combining \eqref{eq.Rinverse-bound} and \eqref{eq:beta-expansion-local} gives $\|R^{-1}\|_2\lesssim \langle\im m\rangle^{-2}$. 

Next we analyze $\langle \im m\rangle$. Theorem \ref{thm:singularities} implies 
\[
    v_i(\tau+x)\sim |x|^{1/3},\qquad 0\le|x|<\epsilon.
\]
By the same argument as that in the edge case, we obtain
\[
    \im m_i(E+i\eta)
    \sim
    \int_{-\epsilon}^{\epsilon}
    \frac{\eta |x|^{1/3}\,dx}{(x-(E-\tau))^2+\eta^2}
    \sim
    (\kappa+\eta)^{1/3},
\]
which implies that
$
    \langle \im m(z)\rangle
    \sim
    (\kappa+\eta)^{1/3}.
$
This proves \eqref{eq.cusp-stability}.

\medskip
\textit{Case 4: Bulk Stability.}
If we are not in the exceptional case, then Proposition \ref{prop.bulk-stability-symmetric} and Proposition \ref{pro.bulk-stability-NB} give $\|(I-m^2S)^{-1}\|_2\lesssim\langle\im m\rangle^{-2}$. Since $z$ is in the strict bulk, we have $\langle \im m\rangle\sim 1$. 
Hence $\|(I-m^2S)^{-1}\|_2\lesssim 1$.

In the exceptional bipartite case with \(\mathbf a=a\mathbf 1\), a shift of the spectral parameter allows us to assume without loss of generality that \(a=0\). Thus the possible singular point is \(z=a\), which becomes the origin after the shift.

We first consider a small neighborhood of the origin. By
Proposition~\ref{prop.bipartite-origin-instability}, in the bipartite
exceptional case we have the sharp estimate
\[
    \bigl\|(I-m(z)^2S)^{-1}\bigr\|_2\sim |z|^{-1},
    \qquad z\in N_\epsilon(0).
\]
This gives the desired bound near the only possible singular point.

In the regime \(|z|\ge \epsilon\), we are away from the only possible
bipartite singularity. It follows from \eqref{eq.symmetric-general-bulk-stability} and \eqref{eq.general-bulk-stability} that
$
    \|(I-m^2S)^{-1}\|_2\lesssim\big(|z|\langle\im m(z)\rangle\big)^{-2}\lesssim 1 .
$
This completes the proof.

\section{Proof of Theorem \ref{lem.stability2}}
Write
$
u:=g-m.
$
We first consider the regime of large imaginary parts.
Let $s:=\max_e\sum_f S_{ef}$.
By the a priori assumptions
$|u|_\infty,|\mathcal R|_\infty\le c_0\le 1$
and the positivity of $\im m$, we have
\[
    \im\bigl(z+(Sg)_e+\mathcal R_e\bigr)
    \ge \eta-(s+1)c_0
    \ge \frac{\eta}{2},
    \qquad \eta\ge 2(s+1).
\]
The perturbed Dyson equation therefore gives
$|g|_\infty\le 2/\eta$.
Together with $|m|_\infty\le 1/\eta$ and the exact identity
$
    u=mg(Su+\mathcal R),
$
this yields
$
    |u|_\infty
    \le 2\eta^{-2}
    \bigl(s|u|_\infty+|\mathcal R|_\infty\bigr).
$
Since $2s/\eta^2\le 1/2$ for $\eta\ge 2(s+1)$,
we conclude that
\[
    |u|_\infty
    \le \frac{4}{\eta^2}|\mathcal R|_\infty
    \lesssim |\mathcal R|_\infty.
\]
Thus the desired stability estimate holds in the
large-height regime.
Therefore, for the remainder of the proof, we may restrict attention to \(0<\eta\le 2(s+1)\), in which regime the Dyson solution admits uniform upper and lower bounds.

Subtracting the Dyson equation for $m$ from the perturbed Dyson equation gives
the exact identity
\begin{equation}
\label{e:bootstrap-exact-dyson-expansion}
(I-m^2S)u
=
m^2\mathcal{R}+\frac{u^2}{g}.
\end{equation}

By the a priori assumption $|u|_\infty\le c_0$ with $c_0$ sufficiently small and the
uniform bounds on the Dyson solution, the entries of $g$ stay uniformly
away from zero. Hence
\begin{equation}
\label{e:abstract-defect-from-displacement}
\left|(I-m^2S)u\right|_\infty
\lesssim
|\mathcal{R}|_\infty+|u|_\infty^2.
\end{equation}
We estimate $u$ separately in the following spectral regimes.

\medskip
\noindent
\emph{Case 1: Stable bulk and off-support.}
Away from the possible bipartite singularity at zero, Theorem \ref{thm.stability} implies $\|(I-m^2S)^{-1}\|_2\sim 1$. Applying $(I-m^2S)^{-1}$ to
\eqref{e:bootstrap-exact-dyson-expansion} therefore yields
$
|u|_\infty
\lesssim
|\mathcal{R}|_\infty+|u|_\infty^2.
$
Since $|u|_\infty$ is small enough, the quadratic term can be absorbed, and hence
$
|u|_\infty\lesssim|\mathcal{R}|_\infty,
$
which proves \eqref{eq:nonlinear-bulk-stability} in the strict bulk away from the bipartite singularity.

It remains to consider the bipartite case near $z=0$.
Write
$
\mathbb V_{\mathcal G}=\mathbb V_+\sqcup\mathbb V_-
$
and let $T_i=\pm1$ on $\mathbb V_\pm$.
For $e\in\vec E_{\mathcal G}$, set
$
b_e(z):=T_{h(e)}m_e(z).
$
By the proof of Proposition \ref{prop.bipartite-origin-instability},
we may use the following facts. The operator $I-m(0)^2S$ has the
one-dimensional kernel $\operatorname{span}\{b(0)\}$ and is uniformly
stable on a fixed complementary subspace $X$. Moreover, it follows from \eqref{eq.bad-vector-linear-opening} that
\begin{equation}
\label{e:bipartite-unstable-direction}
(I-m(z)^2S)b(z)
=
-z\bigl(T_{h(e)}m_e(z)^2\bigr)_{e\in\vec E_{\mathcal G}}.
\end{equation}
For $z$ sufficiently close to zero, one also has the uniformly stable
decomposition
\[
u=\Theta b(z)+w,
\qquad
w\in X,
\qquad
|\Theta|+|w|_\infty\sim |u|_\infty.
\]
Putting the above stable decomposition into \eqref{e:bootstrap-exact-dyson-expansion} and using
\eqref{e:bipartite-unstable-direction} together with uniform stability on
$X$, we obtain
\begin{equation}
\label{e:bipartite-stable-part}
|w|_\infty
\lesssim
|\mathcal R|_\infty+|u|_\infty^2+|z||\Theta|.
\end{equation}
We next control the unstable coefficient $\Theta$ by the bipartite
balance defect. 
Since $0$ is a regular bulk point, we have
$\im\sum_{j\sim i}m_{ij}(0)>0$ for every $i\in\mathbb V_{\mathcal G}$.
By continuity and the assumption $|g-m|_\infty\le c_0$,
the denominators $z+\sum_{j\sim i}g_{ij}(z)$
are uniformly bounded away from zero for sufficiently
small $|z|$ and $c_0$.
Hence the vertex quantities are well-defined, and
a Taylor expansion around $m$ gives
\begin{equation}
\label{e:bipartite-vertex-expansion}
g_i-m_i
=
(m_i)^2\sum_{j\sim i}u_{ij}
+O(|u|_\infty^2).
\end{equation}
Substituting $u=\Theta b(z)+w$ and summing against $T_i$, we obtain
\begin{align}
\sum_{i\in\mathbb V_{\mathcal G}}
T_i\bigl(g_i-m_i\bigr)
&=
\Theta
\sum_{i\in\mathbb V_{\mathcal G}}
T_i(m_i)^2
\sum_{j\sim i}b_{ij}
+
O\bigl(|w|_\infty+|u|_\infty^2\bigr).
\label{e:bipartite-balance-expanded}
\end{align}
The coefficient of $\Theta$ is uniformly non-degenerate.
Indeed, if $t(e)=i$, then bipartiteness gives
$T_{h(e)}=-T_i$, and therefore
\[
\begin{aligned}
\sum_{i\in\mathbb V_{\mathcal G}}
T_i m_i^2
\sum_{j\sim i}b_{ij}
=
-\sum_{i\in\mathbb V_{\mathcal G}}
m_i^2
\sum_{j\sim i}m_{ij}
=
\sum_{i\in\mathbb V_{\mathcal G}}
\left(
m_i
+
zm_i^2
\right),
\end{aligned}
\]
where in the last step we used the vertex Dyson equation.
At $z=0$, this coefficient equals
$
\sum_{i\in\mathbb V_{\mathcal G}}m_i(0)\neq0,
$
since $0$ is a regular bulk point and
$\Im m_i(0)>0$ for every $i$.
Hence, by continuity, it remains uniformly bounded away from zero
for $z$ sufficiently close to zero. It follows from
\eqref{e:bipartite-balance-expanded} that
\begin{equation}
\label{e:bipartite-unstable-amplitude}
|\Theta|
\lesssim
\left|
\sum_{i\in\mathbb V_{\mathcal G}}
T_i\bigl(g_i-m_i\bigr)
\right|
+
|w|_\infty
+
|u|_\infty^2.
\end{equation}
Combining \eqref{e:bipartite-stable-part} and
\eqref{e:bipartite-unstable-amplitude}, and absorbing the term
$|z||\Theta|$ for $z$ sufficiently close to zero, we obtain
\[
|\Theta|+|w|_\infty
\lesssim
|\mathcal R|_\infty
+
\left|
\sum_{i\in\mathbb V_{\mathcal G}}
T_i\bigl(g_i-m_i\bigr)
\right|
+
|u|_\infty^2.
\]
Since the decomposition $u=\Theta b(z)+w$ is uniformly stable, we have
$
|u|_\infty
\lesssim
|\mathcal R|_\infty
+
|
\sum_{i\in\mathbb V_{\mathcal G}}
T_i\bigl(g_i-m_i\bigr)
|
+
|u|_\infty^2.
$
Finally, choosing $c_0$ sufficiently small and using
$|u|_\infty\le c_0$, the last term can be absorbed, yielding
\eqref{eq.bulk-0-stability}.

\medskip\noindent\emph{Case 2: Regular edges and cusps.}
Let $\tau$ be a regular edge or cusp and set $\kappa:=|E-\tau|$.
Theorem \ref{thm.stability} implies
$
\bigl\|(I-m^2S)^{-1}\bigr\|_2
\sim |z-\tau|^{-1/2}
$
when $z$ is near a regular edge $\tau$, whereas
$
\bigl\|(I-m^2S)^{-1}\bigr\|_2
\lesssim |z-\tau|^{-2/3}
$
when $z$ is near a regular cusp $\tau$.
It suffices to consider $|z-\tau|+|\mathcal R|_\infty$ sufficiently small,
since otherwise the desired bound follows from the a priori smallness
of $u$.
We reduce the perturbation to a scalar equation at $\tau$ and
then use vertical continuation to select the physical branch.

\medskip
\noindent
\emph{Step 1: Scalar reduction.}
Set
\[
m_\tau:=m(\tau),
\qquad
p:=\operatorname{sign}m_\tau,
\qquad
F:=|m_\tau|S|m_\tau|.
\]
All quantities in this reduction are evaluated at the fixed regular edge or regular cusp $\tau$. By
Lemma \ref{lem:B-bad-direction-nonsym-local},
$F$ has right and left Perron eigenvectors $r,l>0$, normalized by
\[
Fr=r,
\qquad
l^\top F=l^\top,
\qquad
|r|_2=1,
\qquad
\langle lr\rangle=1.
\]
Moreover, setting
$
P:=I-rl^\top
$
and
$
\mathcal K_\tau:=((I-F)|_{\Ran P})^{-1}P$,
we have $\|\mathcal K_\tau\|_{\infty}\lesssim1$.
Write
\[
\omega:=z-\tau,
\qquad
\xi:=\omega\mathbf1+\mathcal{R},
\qquad
h:=\frac{g-m_\tau}{|m_\tau|}.
\]
By the a priori closeness of $g$ to $m$ and the continuity
of $m$, the vector $h$ is sufficiently small.
Subtracting the Dyson equation $-1/m_\tau=\tau+S m_\tau$ for $m_\tau=m(\tau)$ from the perturbed Dyson equation $-1/g=\tau+Sg+\xi$, and using $g=m_\tau+|m_\tau|h=m_\tau+p m_\tau h$, we obtain the exact identity
\begin{equation}
\label{e:abstract-exact-quadratic-equation}
(I-F)h=\mathcal N(h,\xi),
\qquad
\mathcal N(h,\xi):=p\,hFh+|m_\tau|(1+ph)\xi.
\end{equation}
Decompose $h$ into its singular and stable components as
\[
h=\Theta r+w,
\qquad
\Theta:=\langle lh\rangle,
\qquad
w:=Ph.
\]
Projecting \eqref{e:abstract-exact-quadratic-equation} onto the
stable and unstable directions yields
\begin{equation}
\label{e:abstract-projected-equations}
w=\mathcal K_\tau\mathcal N(\Theta r+w,\xi),
\qquad
0=\langle l\,\mathcal N(\Theta r+w,\xi)\rangle.
\end{equation}
Then we apply the analytic implicit function theorem to the first
equation $w=\mathcal K_\tau\mathcal N(\Theta r+w,\xi)$. 
Writing this equation as \(0=\Psi(w,\Theta,\xi):=w-\mathcal K_\tau\mathcal N(\Theta r+w,\xi)\), we have \(\Psi(0,0,0)=0\) and \(D_w\Psi(0,0,0)=I\). Hence, by the analytic implicit function theorem, there exists a unique analytic map \(w=W_\tau(\Theta,\xi)\in\operatorname{Ran}P\) near \((0,0)\), with \(W_\tau(0,0)=0\).
Differentiating the equation at the origin and using \(Fr=r\), we obtain
\[
\partial_\Theta W_\tau(0,0)=0,\qquad
D_\xi W_\tau(0,0)[\xi]=\mathcal K_\tau(|m_\tau|\xi),\qquad \frac12\partial_\Theta^2W_\tau(0,0)
=\mathcal K_\tau(pr^2).
\]
Thus, the Taylor expansion of \(W_\tau\) yields
\begin{equation}\label{eq.expansion-Wtau}
  W_\tau(\Theta,\xi)
=
\Theta^2\mathcal K_\tau(pr^2)+\mathcal K_\tau(|m_\tau|\xi)
+
O_\infty\bigl(
|\Theta|^3+|\Theta||\xi|_\infty+|\xi|_\infty^2
\bigr),
\end{equation}
where vector remainders are measured in the \(\ell^\infty\) norm. In particular,
\begin{equation}
\label{e:abstract-stable-component}
|W_\tau(\Theta,\xi)|_\infty
\lesssim|\Theta|^2+|\xi|_\infty.
\end{equation}
Moreover, analyticity and the uniform boundedness of \(\mathcal K_\tau\) imply that the first derivatives of \(W_\tau\) are uniformly bounded in a sufficiently small neighborhood of the origin.
Since $\Theta r+w=h=(g-m_\tau)/|m_\tau|$, the perturbed Dyson equation is locally equivalent to
\begin{equation}
\label{e:abstract-exact-reduction}
g=m_\tau+|m_\tau|\bigl(\Theta r+W_\tau(\Theta,\xi)\bigr),
\qquad
f_\tau(\Theta,\xi)=0,
\end{equation}
where
$
f_\tau(\Theta,\xi)
:=
\left\langle
l\,\mathcal N\bigl(\Theta r+W_\tau(\Theta,\xi),\xi\bigr)
\right\rangle.
$
Substituting the expansion \eqref{eq.expansion-Wtau} of $W_\tau$ into this definition gives
\begin{equation}
\label{e:abstract-reduced-scalar-equation}
f_\tau(\Theta,\xi)
=
\sigma\Theta^2+\psi\Theta^3+\langle l|m_\tau|\xi\rangle
+
O\bigl(
|\Theta|^4+|\Theta||\xi|_\infty+|\xi|_\infty^2
\bigr),
\end{equation}
with
$
\sigma:=\langle plr^2\rangle
$
and
$
\psi:=
\left\langle
plr,\,(I+F)\mathcal K_\tau(pr^2)
\right\rangle.$
These are precisely the
coefficients in Lemma \ref{lem:B-bad-direction-nonsym-local}.
Set
\[
k:=
\begin{cases}
2,& \tau \text{ is a regular edge},\\
3,& \tau \text{ is a regular cusp}.
\end{cases}
\]
By Proposition \ref{prop:psi-positive-and-nondegenerate} and the proof of Theorem \ref{thm:singularities},
$|\sigma|\sim1$ at an edge, whereas $\sigma=0$ and $\psi\sim1$
at a cusp. Hence $f_\tau(\Theta,\xi)=0$ and the smallness of
$\Theta$ imply
$
|\Theta|^k\lesssim|\xi|_\infty
\lesssim|\omega|+|\mathcal R|_\infty.
$
Applying the same reduction to $m(z)$ with $\xi=\omega\mathbf1$,
and using \eqref{e:abstract-stable-component}, we obtain
\begin{equation}
\label{e:abstract-nonlinear-bound}
|g(z)-m(z)|_\infty
\lesssim
\bigl(|\omega|+|\mathcal R|_\infty\bigr)^{1/k}=\bigl(|z-\tau|+|\mathcal R|_\infty\bigr)^{1/k}.
\end{equation}
This rough bound does not distinguish between the local
solution branches. Indeed, the result in the folloing Step shows that the $k$ exact branches are
separated by distances of order $|\omega|^{1/k}$,
which is precisely the scale of the rough bound when
$|\mathcal R|_\infty\ll|\omega|$.
Thus, even when $\mathcal{R}=0$, the bound does not exclude
$g$ from lying near a different branch than $m(z)$.
To obtain the sharper stability estimate, we must
first identify the correct branch, which will be
achieved through the argument below.

\medskip
\noindent
\emph{Step 2: Separation of branches and error tubes.}
Consider the exact unperturbed scalar equation
$
0=\Phi_\tau(\theta,\omega)
:=
f_\tau(\theta,\omega\mathbf1).
$
It follows from \eqref{e:abstract-reduced-scalar-equation} that
\[
\Phi_\tau(\theta,\omega)
=
\gamma_\tau\theta^k+\langle l|m_\tau|\rangle\omega
+
O\bigl(
|\theta|^{k+1}+|\omega||\theta|+|\omega|^2
\bigr),
\]
where $k=2$ and $|\gamma_\tau|=|\sigma|\sim 1$ if $\tau$ is a regular edge,
while $k=3$ and $\gamma_\tau=\psi\sim 1$ if $\tau$ is a regular cusp.
The Weierstrass preparation theorem gives the local
factorization
\begin{equation}
\label{e:abstract-local-root-factorization}
\Phi_\tau(\theta,\omega)
=
\mathcal U_\tau(\theta,\omega)
\prod_{j=1}^k\bigl(\theta-\theta_j(\omega)\bigr),
\qquad
|\mathcal U_\tau(\theta,\omega)|\sim1.
\end{equation}
To obtain a more precise description of the local roots,
we introduce the rescaling
$\omega=\zeta^k$ and $\theta=\zeta v$.
By the expansion of $\Phi_\tau$ near the origin,
\[
\zeta^{-k}\Phi_\tau(\zeta v,\zeta^k)
=
\gamma_\tau v^k+\langle l|m_\tau|\rangle+O(\zeta),
\]
where the left-hand side extends analytically to $\zeta=0$,
uniformly for bounded $v$.
Since $\gamma_\tau,\langle l|m_\tau|\rangle\neq0$, the polynomial
$\gamma_\tau v^k+\langle l|m_\tau|\rangle$ has $k$ distinct roots
$\alpha_1,\ldots,\alpha_k$, each satisfying
$
\alpha_j^{k-1}\neq0.
$
The analytic implicit function theorem therefore yields
$k$ analytic solutions $v_j(\zeta)$ near $\zeta=0$,
with
$
v_j(\zeta)=\alpha_j+O(\zeta).
$
Returning to the original variables, we obtain
$
\theta_j(\omega)
=
\alpha_j\omega^{1/k}+O(|\omega|^{2/k}),
$
where $\omega^{1/k}$ is chosen consistently
in the upper half-plane. Thus these $k$ exact
local branches satisfy
\begin{equation}
\label{e:abstract-local-root-separation}
|\theta_j(\omega)|\sim|\omega|^{1/k},
\qquad
\min_{i\ne j}
|\theta_i(\omega)-\theta_j(\omega)|
\gtrsim|\omega|^{1/k}.
\end{equation}
Since $f_\tau(\Theta,\omega\mathbf1+\mathcal{R})=0$ and the derivatives
of $f_\tau$ in $\xi$ are bounded, we have
$
|\Phi_\tau(\Theta,\omega)|\lesssim|\mathcal R|_\infty\le\mathcal E.
$
Consequently, \eqref{e:abstract-local-root-factorization} yields
$
\prod_{j=1}^k|\Theta-\theta_j(\omega)|
\lesssim\mathcal E.
$
Choose a nearest root $\theta_{j_*}(\omega)$. For $j\ne j_*$,
minimality and the triangle inequality give
$
|\Theta-\theta_j(\omega)|
\ge
|\theta_j(\omega)-\theta_{j_*}(\omega)|/2
\gtrsim|\omega|^{1/k}.
$
It follows that
\begin{equation}
\label{e:abstract-distance-to-some-branch}
\min_j|\Theta-\theta_j(\omega)|
\lesssim
\frac{\mathcal E}{|\omega|^{1-1/k}}.
\end{equation}
Fix a sufficiently large constant $C$ and define
\[
\mathcal D_j(z)
:=
\left\{
\theta\in\mathbb C:
|\theta-\theta_j(z-\tau)|\le \varepsilon_{\mathrm{err}}(z)
\right\},
\qquad
\varepsilon_{\mathrm{err}}(z):=
C\frac{\mathcal{E}}{|z-\tau|^{1-1/k}}.
\]
By \eqref{e:abstract-distance-to-some-branch},
the parameter $\Theta$ belongs to at least one of these disks.
We now show that these disks are pairwise disjoint when
$\mathcal{E}\le c_*|z-\tau|$ for a sufficiently
small constant $c_*>0$.
Indeed, the root separation estimate
\eqref{e:abstract-local-root-separation} yields
$
\min_{i\ne j}
|\theta_i(z-\tau)-\theta_j(z-\tau)|
\gtrsim |z-\tau|^{1/k},
$
which implies
\[
\frac{\varepsilon_{\mathrm{err}}(z)}
{\min_{i\ne j}|\theta_i(z-\tau)-\theta_j(z-\tau)|}
\lesssim
\frac{\mathcal{E}}{|z-\tau|}.
\]
Thus, if $\mathcal{E}\le c_*|z-\tau|$ for a sufficiently
small constant $c_*>0$, the radius of each disk is
smaller than half the minimum separation between
their centers. Consequently, the disks are pairwise
disjoint, with mutual distances bounded below by
a constant multiple of $|z-\tau|^{1/k}$. See Figure \ref{fig:branch-separation}.
\begin{figure}[htbp]
    \centering
    \includegraphics[scale=0.6]{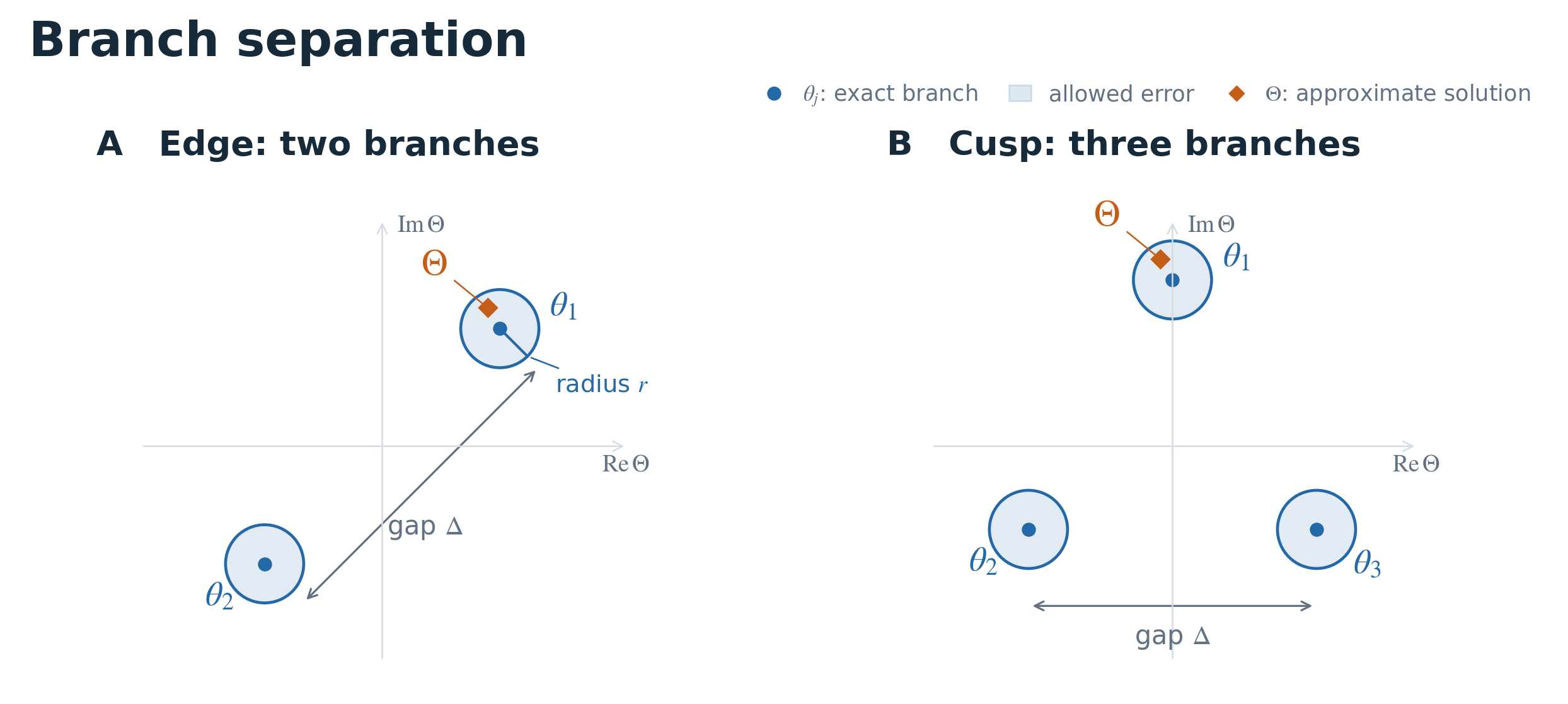}
    \caption{Local branches and disjoint error disks near a regular edge and a regular cusp.}
    \label{fig:branch-separation}
\end{figure}
It follows that $\Theta$ lies in exactly one of the
disks $\mathcal D_j(z)$. Hence the residual estimate
determines a unique nearby branch, but does
not guarantee that this branch coincides
with the branch corresponding to $m(z)$. 
This identification will be established in the following step.

\medskip
\noindent
\emph{Step 3: Selection of the physical branch by continuation.}
If $\mathcal{E}>c_*|z-\tau|$, the rough bound
\eqref{e:abstract-nonlinear-bound} gives
$|g(z)-m(z)|_\infty\lesssim\mathcal{E}^{1/k}\lesssim\mathcal{E}/
(\mathcal{E}+|z-\tau|)^{1-1/k}$, without requiring
branch selection. 
We now consider the case $\mathcal{E}\le c_*|z-\tau|$.
Let $g_t$ be the path from the continuation assumption,
with $z_t=E+\mathrm it$.
By shortening the path if necessary, we may assume
that it lies in a sufficiently small neighborhood
of $\tau$, retaining the notation $\eta_0$ for
the new upper endpoint. Since this height is fixed
and positive, Proposition~\ref{pro.simple-bound-stability}
gives
$
    \bigl\|(I-m(E+\mathrm i\eta_0)^2S)^{-1}\bigr\|_2
    \lesssim 1+\eta_0^{-1}\lesssim 1
$
uniformly for $E$ near $\tau$.
Moreover, the coordinates of $m(E+\mathrm i\eta_0)$
are uniformly bounded away from zero.
Thus, by \eqref{e:bootstrap-exact-dyson-expansion}
and the smallness of $c_0$, we obtain
$
    |g_{\eta_0}-m(z_{\eta_0})|_\infty
    \lesssim\mathcal E.
$
Set
\[
\Theta_g(t):=
\left\langle l\frac{g_t-m_\tau}{|m_\tau|}\right\rangle,
\qquad
\Theta_m(t):=
\left\langle l\frac{m(z_t)-m_\tau}{|m_\tau|}\right\rangle,
\]
and label the roots so that
$\theta_1(z_t-\tau)=\Theta_m(t)$ is the correct branch.
Since $|z_t-\tau|\ge |z-\tau|$, the assumption
$\mathcal{E}\le c_*|z-\tau|$ implies that
$\mathcal{E}\le c_*|z_t-\tau|$ throughout the path.
Together with $|\mathcal{R}(g_t,z_t)|_\infty\le\mathcal{E}$,
this allows us to apply Step~2 at every
$t\in[\eta,\eta_0]$, ensuring that $\Theta_g(t)$
lies in a unique error disk. Thus there is a unique index $j(t)$ with
$
\Theta_g(t)\in\mathcal D_{j(t)}(z_t).
$
The assumed continuity of $g_t$ makes $\Theta_g(t)$ continuous.
The centers and radii of the error disks are continuous as
well, and the disks remain disjoint. Therefore $j(t)$ is
locally constant, hence constant on $[\eta,\eta_0]$:
a change of branch would require $\Theta_g(t)$ to leave the
union of the allowed error disks.
At the upper endpoint, we have
$
|\Theta_g(\eta_0)-\Theta_m(\eta_0)|\lesssim\mathcal{E}.
$
Taking $C$ sufficiently large in the definition of the
disks, we can obtain $\Theta_g(\eta_0)\in\mathcal D_1(z_{\eta_0})$.
It follows that $j(t)=1$ throughout the segment, and hence
\begin{equation}
\label{e:abstract-scalar-linear-stability}
|\Theta_g(t)-\Theta_m(t)|
\lesssim
\frac{\mathcal{E}}{|z_t-\tau|^{1-1/k}},
\qquad t\in[\eta,\eta_0].
\end{equation}
Thus the initial estimate selects the correct branch, while
continuity keeps the approximate solution inside its error
tube. See Figure \ref{fig:branch-selection}.
\begin{figure}[htbp]
    \centering
    \includegraphics[scale=0.7]{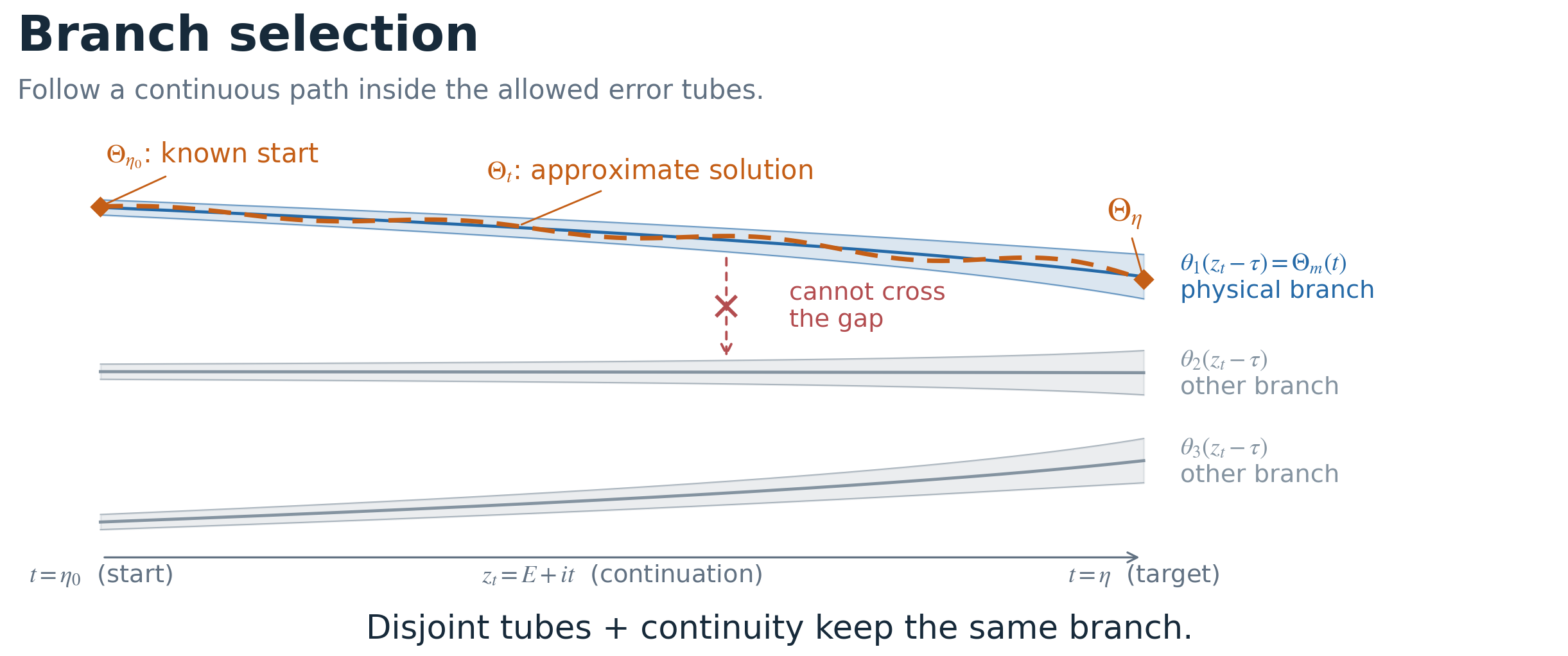}
    \caption{Selection of the correct branch by continuous continuation within disjoint error tubes.}
    \label{fig:branch-selection}
\end{figure}
Finally, applying the reconstruction
\eqref{e:abstract-exact-reduction} to $g(z)$ and $m(z)$ and using
the bounded derivatives of $W_\tau$, we obtain
\[
|g(z)-m(z)|_\infty
\lesssim
|\Theta_g(\eta)-\Theta_m(\eta)|+|\mathcal{R}|_\infty
\lesssim
\frac{\mathcal{E}}{|z-\tau|^{1-1/k}}\lesssim\frac{\mathcal{E}}
{(\mathcal{E}+|z-\tau|)^{1-1/k}},
\]
which proves \eqref{eq:nonlinear-boundary-stability}.
\appendix

\section{Counterexamples}\label{app.Counterexamples}
In this section, we show that without additional assumptions one cannot expect good bulk stability in the general case.
We consider the Dyson equation
\begin{equation}\label{eq.QVE3}
  -\frac{1}{\mathbf m}=z+S\mathbf m,\qquad \mathbf{m}\in\Cp^d,
\end{equation}
where \(S\) is the cyclic permutation matrix of size \(d\ge 3\),
\[
s_{k,k+1}=1\quad (k=1,\dots,d-1),\qquad s_{d,1}=1,
\]
and all other entries vanish. By cyclic symmetry, it is natural to look for a constant solution of the form
\[
-\frac1{m(z)}=z+m(z),\qquad m\in\Cp.
\]
Hence
\begin{equation}\label{eq.solution-expression}
  m(z)=\frac{-z+\sqrt{z^2-4}}{2},
\end{equation}
where the branch is chosen so that \(\im \sqrt{z^2-4}>0\) for \(\eta>0\). Thus, by uniqueness of the solution, 
$
\mathbf{m}(z)=m(z)\mathbf1
$
is exactly the solution to \eqref{eq.QVE3}. In particular, the associated self-consistent density is the semicircle density, so it is strictly positive on the whole bulk interval \((-2,2)\).

We now look at a special point in the bulk. Let
\[
\omega:=e^{2\pi \mathrm i/d},\qquad \theta:=\frac{2\pi}{d},\qquad z_0:=-2\cos\theta,\qquad m_\eta:=m(z_0+\mathrm i\eta),\qquad R_\eta:=I-m_\eta^2S.
\]
Since \(z_0=-2\cos\theta\in(-2,2)\), we have
$
    z_0^2-4=-4\sin^2\theta .
$
With the branch convention \(\im\sqrt{z^2-4}>0\) for \(\eta>0\), it follows that
\[
    \lim_{\eta\downarrow0}\sqrt{(z_0+\mathrm i\eta)^2-4}
    =
    2\mathrm i\sin\theta .
\]
Therefore, using \eqref{eq.solution-expression}, we get
\[
    \lim_{\eta\downarrow0}m_\eta
    =
    \frac{-z_0+2\mathrm i\sin\theta}{2}
    =
    \frac{2\cos\theta+2\mathrm i\sin\theta}{2}
    =
    e^{\mathrm i\theta}
    =
    \omega .
\]
Introduce the Fourier basis
\[
f_k:=\frac1{\sqrt d}(1,\omega^k,\omega^{2k},\dots,\omega^{(d-1)k})^\top,
\qquad k=0,\dots,d-1.
\]
Then, it follows that
$
Sf_k=\omega^k f_k,
$
so \(R_\eta\) is diagonal in the Fourier basis:
\begin{equation}\label{eq.fourier-decomposition}
  R_\eta f_k=\beta_k(\eta)f_k,\qquad \beta_k(\eta):=1-m_\eta^2\omega^k.
\end{equation}

Now we claim the singular direction is \(f_{d-2}\). 
Write \(m_\eta:=\omega(1+\delta_\eta)\). Since \(m_\eta\) satisfies
$
    m_\eta+1/{m_\eta}=-(z_0+\mathrm i\eta)
$
and \(-z_0=2\cos\theta=\omega+\omega^{-1}\), we get
\[
    \omega(1+\delta_\eta)
    +
    \omega^{-1}(1+\delta_\eta)^{-1}
    =
    \omega+\omega^{-1}-\mathrm i\eta .
\]
Using \((1+\delta_\eta)^{-1}=1-\delta_\eta+O(\delta_\eta^2)\), this gives
$
    (\omega-\omega^{-1})\delta_\eta
    =
    -\mathrm i\eta+O(\delta_\eta^2).
$
As \(\omega-\omega^{-1}=2\mathrm i\sin\theta\), we conclude that
\[
    \delta_\eta
    =
    -\frac{\eta}{2\sin\theta}+O(\eta^2).
\]
Then using \eqref{eq.fourier-decomposition} yields
\[
    \beta_{d-2}(\eta)=1-m_\eta^2\omega^{d-2}=1-(\omega^{-1}m_\eta)^2=1-(1+\delta_\eta)^2=\frac{\eta}{\sin\theta}+O(\eta^2).
\]

For \(k\neq d-2\), \eqref{eq.fourier-decomposition} implies that
\[
\beta_k(\eta)=1-m_\eta^2\omega^k=1-(m_\eta\omega^{-1})^2\omega^{k+2}.
\]
Since \(\lim_{\eta\downarrow 0}m_\eta=\omega\), we have
$
\lim_{\eta\downarrow 0}\beta_k(\eta)= 1-\omega^{k+2}\neq0
$
for $k\neq d-2$
and hence there exists \(c>0\) such that
$
|\beta_k(\eta)|\ge c
$
for all sufficiently small \(\eta>0\). Consequently,
\[
|\beta_{d-2}(\eta)|\sim\eta,
\qquad
|\beta_k(\eta)|\sim1\quad (k\neq d-2),
\]
and thus
$
\|R_\eta^{-1}\|_2\sim \eta^{-1}.
$
So the loss of stability at the bulk point \(z_0\) is genuine, even though it is confined to a single Fourier mode.

\bibliography{Ref-Dyson}
\bibliographystyle{plain}  
\end{document}